\documentclass[reqno]{amsart}

\usepackage{amssymb,amsmath}
\usepackage{euscript}
\usepackage{graphicx}
\usepackage{tikz}
\usepackage{tikz-cd}
\usepackage{hyperref}
\usepackage{mathrsfs}
\usepackage{dsfont}
\usepackage{mathtools}
\usepackage{microtype}
\tikzset{
  trim node/.default=1cm,
  trim node/.style={
    overlay,
    append after command={% restore smaller bounding box
      ([xshift={+#1}]\tikzlastnode.north west)
      ([xshift={+-#1}]\tikzlastnode.south east)}},
  down and trim/.default=1cm,
  down and trim/.style={
    yshift=-(\pgfmatrixcurrentcolumn-1)*1.5\baselineskip,
    trim node={#1}},
  downup and trim/.default=1cm,
  downup and trim/.style={
    yshift=iseven(\pgfmatrixcurrentcolumn) ? -1.5\baselineskip : 0pt,
    trim node={#1}},
  -|/.style={to path={-|(\tikztotarget)\tikztonodes}},
  |-/.style={to path={|-(\tikztotarget)\tikztonodes}},
  -| sl/.style={-|, xslant=-1},
  |- sl/.style={|-, xslant= 1},
  center picture/.style={
    trim left=(current bounding box.center),
    trim right=(current bounding box.center)}}

\newtheorem{thm}{Theorem}[section]
\newtheorem{cor}[thm]{Corollary}
\newtheorem{lem}[thm]{Lemma}
\newtheorem{prop}[thm]{Proposition}
\newtheorem{defin}[thm]{Definition}
\newtheorem{def-lem}[thm]{Definition-Lemma}
\newtheorem{conj}[thm]{Conjecture}

\theoremstyle{remark}
\newtheorem{notation}{Notation}

\newtheorem{warning}[thm]{Warning}
\newtheorem{rem}[thm]{Remark}

\newtheorem*{claim}{Claim}
\newtheorem{ques}{Question}
\newtheorem{assu}{Assumption}

\numberwithin{equation}{section}

\newcommand{\bbB}{\mathbb{B}}
\newcommand{\bbC}{\mathbb{C}}
\newcommand\bbD{\mathbb{D}}

\newcommand{\bbF}{\mathbb{F}}
\newcommand{\bbG}{\mathbb{G}}
\newcommand{\bbH}{\mathbb{H}}
\newcommand{\bbI}{\mathbb{I}}
\newcommand{\bbM}{\mathbb{M}}

\newcommand{\bbP}{\mathbb{P}}
\newcommand{\bbR}{\mathbb{R}}

\newcommand{\bbS}{\mathbb{S}}

\newcommand{\bbW}{\mathbb{W}}
\newcommand{\bbX}{\mathbb{X}}
\newcommand{\bbY}{\mathbb{Y}}
\newcommand{\bbZ}{\mathbb{Z}}

\newcommand{\bfJ}{\mathbf{J}}

\newcommand{\scrD}{\mathscr{D}}

\newcommand{\scrF}{\mathscr{F}}

\newcommand{\scrM}{\mathscr{M}}
\newcommand{\scrO}{\mathscr{O}}

\newcommand{\scrR}{\mathscr{R}}

\newcommand{\calA}{\mathcal{A}}

\newcommand{\calG}{\mathcal{G}}

\newcommand{\calL}{\mathcal{L}}

\newcommand{\calR}{\mathcal{R}}

\newcommand{\calW}{\mathcal{W}}

\newcommand{\frakc}{\mathfrak{c}}
\newcommand{\frakD}{\mathfrak{D}}
\newcommand{\frakE}{\mathfrak{E}}
\newcommand{\frakF}{\mathfrak{F}}

\newcommand{\frakM}{\mathfrak{M}}

\newcommand{\frako}{\mathfrak{o}}

\newcommand{\frakS}{\mathfrak{S}}

\newcommand{\frakX}{\mathfrak{X}}

\DeclareMathOperator{\coker}{coker}
\DeclareMathOperator{\crit}{Crit}
\DeclareMathOperator{\grad}{grad}

\DeclareMathOperator{\ind}{ind}

\DeclareMathOperator{\ob}{ob}

\newcommand{\eval}{\mathrm{Eval}}
\newcommand{\eneval}{\mathrm{EnEval}}
\newcommand{\flow}{\mathrm{Flow}}
\newcommand{\fr}{\mathrm{fr}}

\newcommand{\identity}{\mathrm{Id}}

\newcommand{\lelogpss}{\mathrm{LePSS}_{\log}}

\newcommand{\spectra}{\mathrm{Sp}}

\newcommand{\std}{\mathrm{std}}
\newcommand{\topological}{\mathrm{top}}

\newcommand{\unit}{\mathds{1}}

\newcommand{\abs}[1]{\lvert#1\rvert}

\usepackage[textwidth=25mm, textsize=tiny]{todonotes}

\title[Ample divisor complements, Floer spectra, and relative GW theory]{Ample divisor complements, Floer spectra, \\ and relative Gromov-Witten theory}
\author{Kenneth Blakey}
\address{Department of Mathematics, MIT, 182 Memorial Drive, Cambridge, MA 02139, U.S.A.} 
\email{kblakey@mit.edu}

\begin{document}

\begin{abstract}
We spectrally lift Ganatra-Pomerleano's low-energy log PSS morphism to compute the associated graded of Floer homotopy types of ample smooth divisor complements. Moreover, we show the obstruction to splitting into the associated graded is encoded in a stable homotopy class defined via (higher-dimensional) genus 0 relative Gromov-Witten moduli spaces. We compute numerous examples of splittings, including the affine part of all smooth projective hypersurfaces of degree at least 2.
\end{abstract}

\maketitle
\tableofcontents

\section{Introduction}
Symplectic cohomology, a variant of Hamiltonian Floer cohomology, is a powerful yet difficult to compute invariant of exact symplectic manifolds with contact type boundary. Let $(M,D)$ be a smooth complex projective variety $M$ together with an ample smooth divisor $D$. In this case, the complement $X\equiv M-D$ is Liouville (in fact, Stein), therefore we may define its symplectic cohomology. It was originally an idea of Seidel \cite{Sei02a} that the symplectic cohomology of $X$ should be related to the topology and (relative) Gromov-Witten invariants of $(M,D)$. This has been explored in work of Diogo \cite{Dio12} and Diogo-Lisi \cite{DL19}. Work of Ganatra-Pomerleano \cite{GP20,GP21} has extended this relationship to the case of an ample strict normal crossings divisor. One main takeaway from the aforementioned works is that there is a spectral sequence, with $E_1$-page related to the topology of $(M,D)$, which converges to the symplectic cohomology of $X$. Moreover, degeneration at the $E_1$-page can be detected by the vanishing of certain genus 0 relative Gromov-Witten invariants of $(M,D)$.

Meanwhile, Floer homotopy is an idea, originally due to Cohen-Jones-Segal \cite{CJS95}, which defines a spectral refinement of various Floer-type (co)homologies. In the exact setting, and under additional assumptions, it was sketched in \emph{loc. cit.}, with details eventually given by Large \cite{Lar21}, how coherent stable framings on the compactified moduli spaces of Floer trajectories leads to a spectral refinement of symplectic cohomology called the \emph{(framed) Floer homotopy type}. In fact, Floer homotopy has seen immense recent interest and rapid development, c.f. \cite{AB21,AB24,ADP24,ADP25,BB25,Bla1,Bla2,Bla24,Bon24,Bon25,BP26,CK23,PS24b,PS25a,PS25b,PS25c,Rez24}. However, most work has focused on foundations or applications -- not on computations.\footnote{Note, the work of Porcelli and the present author \cite{BP26} may be seen as both foundational and computational. \emph{Loc. cit.}, which appeared after the present article first appeared as a preprint, answers a natural perturbation of a question asked here, cf. Question \ref{ques:splittingques} and Remark \ref{rem:bp26}.} The purpose of the present article is to compute Floer homotopy types of ample smooth divisor complements using genus 0 relative Gromov-Witten moduli spaces. Cf. Theorems \ref{thm:main} and \ref{thm:main2} for our main results and cf. Subsection \ref{subsec:introcomputations} for an overview of our computations.

\subsection{Main results}
\begin{assu}\label{assu:main}
Our setup is the following. 
\begin{enumerate}
\item $M$ is a complex $n$-dimensional projective variety.
\item $D\subset M$ is an ample smooth divisor. Recall, this means there exists an ample complex line bundle $\calL_D\to M$ together with a choice of global section $s_D\in H^0(M,\calL_D)$ whose divisor of zeroes is $\kappa D$, where $\kappa\in\bbZ_{>0}$.
\item For any $n'\in\bbZ$, we define 
    \begin{equation}
    \calL_{n'}\equiv\scrO_M(n'D)\equiv\calL_D^{\otimes_\bbC n'}.
    \end{equation}
We will assume there exists a non-negative integer $d\in\bbZ_{\geq0}$, integers $m,n_1,\ldots,n_\alpha\in\bbZ$, oriented real vector bundles $F_1,\ldots,F_\alpha\to M$ whose restriction to $X\equiv M-D$ are spin (i.e., we fix a choice of orientation and spin structure), and a choice of isomorphism of complex vector bundles 
    \begin{equation}
    TM\oplus\underline{\bbC}^d\oplus\calL_m\cong\bigoplus_{\nu=1}^\alpha E_\nu\otimes_\bbC\calL_{n_\nu},\;\;E_\nu\equiv F_\nu\otimes_\bbR\underline{\bbC};
    \end{equation}
here, an underline denotes a trivial vector bundle with the corresponding fiber, e.g., 
    \begin{equation}
    \underline{\bbC}\equiv M\times\bbC\to M.
    \end{equation}
\end{enumerate}
\end{assu}

\begin{rem}[Positivity of intersection]
Observe, part (2) of Assumption \ref{assu:main} implies that, if $v:\bbC P^1\to D$ is any (pseudo)holomorphic sphere, then we have
    \begin{equation}
    v_*[\bbC P^1]\cdot D>0.
    \end{equation}
\end{rem}

\begin{rem}
If we were in the situation when, in addition to parts (1) and (2) of Assumption \ref{assu:main}, no (pseudo)holomorphic sphere intersects the divisor transversely in a single point, then we would not have to assume part (3) of Assumption \ref{assu:main} since the situation is simpler, cf. Lemma \ref{lem:main}.
\end{rem}

Observe, $s_D$ induces an identification $\calL\vert_X\cong\underline{\bbC}$; hence, we have a fixed $\bbR$-polarization $\calL\vert_X\cong\calL_\bbR\otimes_\bbR\underline{\bbC}$, where $\calL_\bbR\to X$ is a real line bundle with a fixed identification $\calL_\bbR\cong\underline{\bbR}$. Let $v:\bbD\to M$ be any embedded disk whose image is the fiber over a point in $D$ of the disk bundle $D_DM$ associated to the normal bundle $N_DM$ such that $v\vert_{S^1}$ goes once \emph{counterclockwise} around $D$ (in particular, $v_p\cdot D=1$). We define
    \begin{equation}
    \mu_m\equiv\mu(v^*\calL_m,v\vert_{S^1}^*\calL_{m,\bbR}),\;\;\mu_\nu\equiv\mu(v^*\calL_{n_\nu},v\vert_{S^1}^*\calL_{n_\nu,\bbR}),
    \end{equation}
where $\mu(\cdot,\cdot)$ is the boundary Maslov index. Note, we have that 
    \begin{equation}
    \mu_m=2m,\;\;\mu_\nu=2n_\nu. 
    \end{equation}
Under Assumption \ref{assu:main}, we have a stable $\bbR$-polarization $\Lambda$ of $X$:
    \begin{equation}
    TX\oplus\underline{\bbC}^{d+1}\cong\Bigg(\bigoplus_{\nu=1}^\alpha F_\nu\otimes_\bbR\calL_{n_\nu,\bbR}\Bigg)\otimes_\bbR\underline{\bbC}\equiv\Lambda\otimes_\bbR\underline{\bbC};
    \end{equation}
hence, the Floer homotopy type $\frakF^\Lambda$ (associated to $\Lambda$) of $X$ exists. We define the following two real virtual bundles on $M$:
    \begin{align}
    V_k&\equiv\underline{\bbR}^{-d-1+2k(d+1-m)}+TM^{\oplus k}+\sum_{\nu=1}^\alpha F_\nu^{\oplus 1+k(2n_\nu-2)}-N_DM^{\oplus k}, \\
    V^{\bbG\bbW}&\equiv TM+\sum_{\nu=1}^\alpha F_\nu^{\oplus 2n_\nu}-\underline{\bbR}^{2m}-N_DM,
    \end{align}
with $k\in\bbZ_{\geq0}$. 

\begin{notation}
We use the following notation throughout the present article. 
\begin{itemize}
\item We often abuse notation and think of $N_DM$ as a vector bundle defined on the entirety of $M$ since $\scrO_M(D)\vert_D=N_DM$.
\item Moreover, we often abuse notation to also use $V_k$ and $V^{\bbG\bbW}$ to denote the restriction to $X$.
\item For any dualizable spectrum $\frakX$, we denote by $\frakD\frakX$ its dual.
\item For any (virtual) vector bundle $E$ over a space $Y$, we denote by $Y^E$ its associated Thom spectrum.
\item Finally, we denote by $S_DM$ the circle bundle of $N_DM$. 
\end{itemize}
\end{notation}

Our main result is the following.

\begin{thm}\label{thm:main}
Assume Assumption \ref{assu:main}. There exists an element 
    \begin{equation}
    \calG\calW\in\pi_0\Big((S_DM)^{-V^{\bbG\bbW}+TX}\wedge\Sigma\frakD X_+\Big),
    \end{equation}
referred to as a \emph{spectral Gromov-Witten obstruction}, with the following property: suppose $\calG\calW=0$, then $\frakF^\Lambda$ splits as a wedge sum of Spanier-Whitehead duals of Thom spectra:
    \begin{equation}\label{eqn:mainthm}
    \frakF^\Lambda\simeq\frakD X^{-V_0}\vee\bigvee_{k\geq1}\frakD(S_DM)^{-V_{\overline{w}(k)}},\;\;\overline{w}(k)\equiv\kappa k.
    \end{equation}
\end{thm}

\begin{rem}
Essentially, $\calG\calW$ is defined via considering (twisted) framed bordism classes constructed from genus 0 relative Gromov-Witten moduli spaces, cf. Section \ref{sec:sgwo}. Also, the pieces in \eqref{eqn:mainthm} arise as the associated graded of the action filtration of $\frakF^\Lambda$, i.e., $\calG\calW$ is the obstruction to this filtration splitting.
\end{rem}

\begin{rem}
We refer the reader to \cite{AMS24,AB25} for other instances of ``spectral (symplectic) Gromov-Witten theory'' appearing in the literature.
\end{rem}

In Part \ref{subsubsec:auxenhancedspheres}, we will see that $\calG\calW$ equivalently gives an element of 
    \begin{equation}
    \pi_0\Big((S_DM)^{-V_{\overline{w}(1)}}\wedge\Sigma\frakD X^{-V_0}\Big)=H_0\Big(\Sigma\frakD X^{-V_0};(S_DM)^{-V_{\overline{w}(1)}}\Big).
    \end{equation}
In particular, $\calG\calW$ may be viewed as the image of the coevaluation 
    \begin{equation}
    \bbS\to(S_DM)^{-V_{\overline{w}(1)}}\wedge\frakD(S_DM)^{-V_{\overline{w}(1)}}
    \end{equation}
under the map 
    \begin{equation}
    H_0\Big(\frakD(S_DM)^{-V_{\overline{w}(1)}};(S_DM)^{-V_{\overline{w}(1)}}\Big)\to H_0\Big(\Sigma\frakD X^{-V_0};(S_DM)^{-V_{\overline{w}(1)}}\Big)
    \end{equation}
induced by $\calG\calW$ itself. Utilizing this, we prove another criterion for splitting.

\begin{thm}\label{thm:main2}
Given a Lagrangian $L\subset X$, we denote by $\frakF^{\Lambda\vert_{T^*L}}$ the Floer homotopy type of $T^*L$ using the restricted stable $\bbR$-polarization $\Lambda\vert_{T^*L}$ of $T^*L$. Suppose there exists exact Lagrangians $L_1,\ldots,L_\mu\subset X$ such that:
\begin{enumerate}
\item the inclusions induce an injection
    \begin{equation}
    H_0\Big(\Sigma\frakD X^{-V_0};(S_DM)^{-V_{\overline{w}(1)}}\Big)\to \bigoplus_{\rho=1}^\mu H_0\Big(\Sigma\frakD L^{-V_0}_\rho;(S_DM)^{-V_{\overline{w}(1)}}\Big),
    \end{equation}
\item and the (suspension of the) PSS morphism is injective: 
    \begin{equation}
    H_0\Big(\Sigma\frakD L^{-V_0}_\rho;(S_DM)^{-V_{\overline{w}(1)}}\Big)\to H_0\Big(\Sigma\frakF^{\Lambda\vert_{T^*L_\rho}};(S_DM)^{-V_{\overline{w}(1)}}\Big),\;\;\rho\in\{1,\ldots,\mu\};
    \end{equation}
\end{enumerate}
then
    \begin{equation}
    \frakF^\Lambda\simeq\frakD X^{-V_0}\vee\bigvee_{k\geq1}\frakD(S_DM)^{-V_{\overline{w}(k)}}.
    \end{equation}
\end{thm}

\begin{rem}
In particular, we would like to emphasize that the splitting criterion detailed in Theorem \ref{thm:main2} does not require a computation of $\calG\calW$ itself.
\end{rem}
    
\subsection{Sneak peak of computations}\label{subsec:introcomputations}
In Section \ref{sec:computations}, we use Theorem \ref{thm:main} to compute the Floer homotopy type of many ample smooth divisor complements.

\begin{itemize}
\item $M$ is $\bbC P^{n+1}$ and $D$ is a smooth hypersurface of degree at least 2, such that $X$ admits a stable $\bbR$-polarization, cf. Proposition \ref{prop:warmup1}. (Observe, in the case that $D$ is degree 2, $X$ is Weinstein equivalent to $T^*\bbR P^n$; in particular, this splitting, in the degree 2 case, is already known via algebro-topological methods, c.f. Remark \ref{rem:algebrotopological}.)
\item $M$ is a smooth hypersurface in $\bbC P^{n+1}$ of degree at least 2, and $D$ is the intersection of $M$ and a smooth hypersurface of degree at least 2, such that $X$ admits a stable $\bbR$-polarization, cf. Proposition \ref{prop:warmup2}.
\item $M$ is the standard projective quadric hypersurface $\{z_0^2=z_1^2+\cdots+z_{n+1}^2\}$ in $\bbC P^{n+1}$ and $D$ is the hyperplane section $\{z_0=0\}$, cf. Proposition \ref{prop:spheresplitting}. (Observe, $X$ is Weinstein equivalent to $T^*S^n$; in particular, this splitting is already known via algebro-topological methods, c.f. Remark \ref{rem:algebrotopological}.)
\item $M$ is either: (1) an even degree del Pezzo surface (which is not $\bbC P^1\times\bbC P^1$ -- this case is handled by the $T^*S^2$ computation) or (2) a degree 3 del Pezzo surface, and $D$ is a generic smooth representative of $M$'s anticanonical divisor, cf. Proposition \ref{prop:delpezzo1} resp. \ref{prop:delpezzo2}. The missing cases of degree 1, 5, and 7 del Pezzo surfaces are slightly out of reach of the present article, cf. Conjecture \ref{conj:main} and Subsection \ref{subsec:splittings}.
\end{itemize}

By combining our computation for $T^*S^n$ with Theorem \ref{thm:main2}, we may compute the Floer homotopy type of the affine part of any smooth projective hypersurface of degree at least 3, cf. Proposition \ref{prop:affine}. This answers (the spectral extension of) a question of Ganatra-Pomerleano, at least in the case of a smooth divisor, cf. \cite[Question 1.8]{GP20}.

\begin{rem}
Let $H\bbZ$ be the $\bbZ$-Eilenberg MacLane spectrum. It is a classical result of Adams that, for any $H\bbZ$-module spectra $\frakX$, there is a homotopy equivalence $\bigvee\Sigma^jH\pi_j\frakX\simeq\frakX$. The reason for this digression is that there is no difference between having a splitting of $SH^{*+n}(X;\bbZ)=H_{-*}(\frakF^\Lambda;\bbZ)$ and a splitting of $\frakF^\Lambda_X\wedge H\bbZ$; hence, we will elide the distinction.
\end{rem}

\subsection{Questions, speculations, etc}
An immediate (and arguably the most interesting) question is the following. 

\begin{ques}\label{ques:splittingques}
\emph{A priori}, it may be the case that $\frakF^\Lambda\wedge H\bbZ$ splits but $\calG\calW$ does not vanish (i.e., $\frakF^\Lambda$ does not split). Does this happen?
\end{ques}

\begin{rem}
Of course, if we can find such an example, Theorem \ref{thm:main2} provides restrictions on the types of Lagrangians which can arise in such an example. Moreover, given that the obstruction to splitting is $\calG\calW$, such examples are lurking in situations where higher-dimensional (relative) Gromov-Witten moduli spaces are ``interesting bordism-theoretically''.
\end{rem}

\begin{rem}\label{rem:bp26}
In \cite[Part 1.3.2]{BP26}, which appeared after the present article first appeared as a preprint, Porcelli and the present author produced an example of an ample smooth divisor complement whose canonical $MU$ Floer homotopy type $\frakF^{MU}$ satisfies $\frakF^{MU}\wedge_{MU} H\bbZ$ splits but $\calG\calW^{MU}$ does not vanish. The example is $M=\bbC P^n\times\bbC P^n$, $D$ the ample $(1,1)$-hypersurface 
    \begin{equation}
    \bigg\{\big([x_0:\cdots:x_n],[y_0:\cdots:y_n]\big):\sum_{i=0}^n x_iy_i=0\bigg\},
    \end{equation}
and $X=T^*\bbC P^n$, where $n\in\bbZ$ is odd and at least 3. In particular, the non-splitting in \emph{loc. cit.} is detected by a necessarily higher-dimensional complex bordism class of genus 0 relative Gromov-Witten moduli space. Note, this example does not satisfy Assumption \ref{assu:main} of the present article since it fails part (3); in particular, this computation essentially relies on repeating the present article using the canonical $MU$ Floer homotopy type, instead of a framed Floer homotopy type, to obtain the analogous splitting/obstruction results.
\end{rem}

Due to the computations performed in the present article, it seems reasonable to conjecture the following. (We have some remarks in Part \ref{part:remarks}.)

\begin{conj}\label{conj:main}
Let $M$ be a del Pezzo surface of degree 1, 5, or 7 and $D$ a generic smooth representative of $M$'s anticanonical divisor. There exists a stable $\bbR$-polarization $\Lambda$ of $X$ and real vector bundles $V_1,\ldots,V_k,\ldots\to M$ such that
    \begin{equation}
    \frakF^\Lambda\simeq\frakD X^{-V_0}\vee\bigvee_{k\geq1}\frakD(S_DM)^{-V_k}.
    \end{equation}
\end{conj}

As was shown in \cite{BB25}, considering the Floer homotopy type associated to different stable $\bbR$-polarizations can yield more geometric information. An interesting question is the following.

\begin{ques}
Does there exist a pair $(M,D)$ such that: 
\begin{enumerate}
\item $SH^*(X;\bbZ)$ splits, 
\item there exists a choice of $\Lambda$ such that $\frakF^\Lambda$ splits, 
\item and there exists a choice of $\Lambda'$ such that $\frakF^{\Lambda'}$ does not split?
\end{enumerate}
\end{ques}

As stated earlier, we provide splittings of Floer homotopy types associated to $T^*S^n$ and $T^*\bbR P^n$. These are (cotangent bundles of) two examples of \emph{Zoll manifolds}, i.e., smooth manifolds whose geodesic flow is a free $S^1$-action. Two other examples are $T^*\bbC P^n$ and $T^*\bbH P^n$.

\begin{ques}
Does there exist an algebro-geometric model which realizes $T^*\bbC P^n$ resp. $T^*\bbH P^n$ as the complement of an ample smooth divisor such that the results of the current article (or some generalization of) apply, and can be computed, to produce a splitting of the associated Floer homotopy type?
\end{ques}

\begin{rem}\label{rem:algebrotopological}
The author believes this can be accomplished. Recall, the Floer homotopy type of a cotangent bundle is related to the stable homotopy type of the free loop space of the base; this is known as the spectral Viterbo isomorphism. Hence, once one finds a suitable algebro-geometric model, hope for proving a splitting is provided by the fact that $\calL\bbC P^n$ resp. $\calL\bbH P^n$ admit stable splittings via algebro-topological methods, cf. \cite{BO18, Coh87}. (This is actually what led the author to consider $T^*S^n$ and $T^*\bbR P^n$, i.e., it is known $\calL S^n$ resp. $\calL\bbR P^n$ admit stable splittings via algebro-topological methods, c.f. \cite{BCS01,Coh87}.) 
\end{rem}

\subsection*{Acknowledgments}
The author would like to thank his advisor Paul Seidel for suggesting and mentoring this project. The author would also like to thank Shaoyun Bai for reading an early draft of the present article. Finally, the author would like to also thank Ciprian Bonciocat, Liam Keenan, Daniel Pomerleano, and Noah Porcelli for helpful discussions. This work was partially supported by an NSF Graduate Research Fellowship award.

\section{Recollections of the cohomological story, after Ganatra-Pomerleano}
In this section, we will recall the construction of the spectral sequence which converges to the symplectic cohomology of an ample smooth divisor complement. Moreover, we will recall the construction of Ganatra-Pomerleano's (GP's) \cite{GP20,GP21} ``low-energy log PSS morphism'' which is used to compute this spectral sequence's $E_1$-page. Cf. \cite{Dio12,DL19,Sei02a} for related discussions.

\subsection{Topological limit model for symplectic cohomology}
\subsubsection{Log-smooth pairs}
Let $(M,D)$ be a smooth complex $n$-dimensional projective variety together with an ample smooth divisor $D\subset M$, i.e., there exists an ample complex line bundle $\calL_D\to M$ together with a choice of global section $s_D\in H^0(M,\calL_D)$ whose divisor of zeroes is $\kappa D$:
    \begin{equation}\label{eqn:kappa}
    \operatorname{div}(s_D)=\kappa D,\;\;\kappa\in\bbZ_{>0}.
    \end{equation}
We denote by $X$ the complement $M-D$; we call $M$ a \emph{(log-)smooth compactification} of $X$.\footnote{In general, \cite[Main Theorem I]{Hir64} says any smooth affine variety admits a \emph{log-smooth compactification}, i.e., a compactification to a smooth complex projective variety such that the original affine variety is the complement of a strict normal crossings divisor which supports an ample complex line bundle.} Moreover, we will assume the canonical bundle of $M$ is supported on $D$:
    \begin{equation}\label{eqn:nutilde}
    K_M\equiv\Lambda^n_\bbC T^*M\cong\scrO_M(-\widetilde{\nu}D),\;\;\widetilde{\nu}\in\bbZ_{>0}.
    \end{equation}

\begin{rem}
This is subsumed by Assumption \ref{assu:main}: 
    \begin{equation}
    K_M\cong\scrO_M\big(-(-m+n_1\operatorname{rank}_\bbR F_1+\cdots+n_\alpha\operatorname{rank}_\bbR F_\alpha)D\big).
    \end{equation}
\end{rem}

We have that $M$ admits a symplectic structure $\omega$. Consider the potential $-\log\abs{\abs{s_D}}$; we have the following equality after restricting to $X$:
    \begin{equation}
    \omega=-dd^c\big(-\log\abs{\abs{s_D}}\big).
    \end{equation}
By setting
    \begin{equation}
    \theta\equiv-d^c\big(-\log\abs{\abs{s_D}}\big),
    \end{equation}
we induce the structure of a finite-type convex symplectic manifold on $X$, cf. \cite[Appendix A]{McL12}. On the other hand, by holomorphically embedding $X$ into $\bbC^N$, $N\gg0$, we may give $X$ the structure of a Stein manifold by taking the restriction to $X$ of the ambient standard K\"ahler form on $\bbC^N$. Observe, the former finite-type convex symplectic manifold structure on $X$ and the latter ``canonical'' Stein structure on $X$ are deformation equivalent. 

\cite[Theorem 5.20]{McL12} shows this Stein structure, perhaps after deforming, is ``nice'' (or ``adapted to $D$''); this means the following. Let $\pi:N_DM\to D$ be the normal bundle of $D$ inside of $M$. There exists a tubular neighborhood $D_{D,\delta}M\subset N_DM$ such that: 
\begin{itemize}
\item $\pi\vert_{D_{D,\delta}M}:D_{D,\delta}M\to D$ is a symplectic fibration with structure group $U(1)$,
\item each fiber is a disk of radius $\delta\in\bbR_{>0}$, 
\item and the restriction of $\theta$ to a fiber is 
    \begin{equation}
    \dfrac{1}{2\pi}(\pi r^2-\kappa)d\varphi,
    \end{equation}
where $(r,\varphi)$ are polar coordinates on the disk (with $\partial_\varphi$ oriented \emph{clockwise}\footnote{This convention, which is opposite of the usual one, follows \cite{GP20}.}).
\end{itemize}
We denote by $\rho\equiv r^2:D_{D,\delta}M\to[0,\delta^2]$ the extension of the various squared radial functions on the fibers to the entire tubular neighborhood. Three convenient properties of a nice Stein structure are the following, cf. \cite[Lemma 2.4]{GP20}.

\begin{enumerate}
\item The symplectic orthogonal complement of the tangent space of a fiber of $D_{D,\delta}M$ is contained in a level set of $\rho$.
\item If $f\in C^\infty(D_{D,\delta}M)$ is a Hamiltonian which is a smooth function of $\rho$, then the associated Hamiltonian vector field $X_f$ (satisfying $\omega(X_f,\cdot)=-df$) is tangent to the fibers of $D_{D,\delta}M$ and of the form $2(\partial f/\partial\rho)\partial_\varphi$.
\item Any two Hamiltonians which are smooth functions of $\rho$ have commuting Hamiltonian vector fields.
\end{enumerate}

\subsubsection{Divisorially-adapted Floer data}
Let $\epsilon\in\bbR_{>0}$ be sufficiently small; we define 
    \begin{equation}
    D_{D,\epsilon}M\equiv\rho^{-1}\big([0,\epsilon\kappa/2\pi]\big).
    \end{equation}
Moreover, we denote by $\widehat{X}_\epsilon$ the region bounded by $\partial(M-D_{D,\epsilon}M)$ which is contained in $X$. \cite[Lemma 3.11]{GP21} says $\theta\vert_{\widehat{X}_\epsilon}$ determines a Liouville domain structure on $\widehat{X}_\epsilon$. Let $\widehat{X}^\circ_\epsilon$ denote the complement, in $\widehat{X}_\epsilon$, of a collar neighborhood of $\partial\widehat{X}_\epsilon$ determined by flowing along the Liouville vector field $Z$ of $\theta$ for some sufficiently small fixed negative time $t_\epsilon\in\bbR_{<0}$. We have a map 
    \begin{align}
    r_\epsilon:X-\widehat{X}^\circ_\epsilon&\to\bbR \\
    p&\mapsto e^{t_p}, \nonumber
    \end{align}
where $t_p$ is the time it takes to flow along $Z$ from $\partial\widehat{X}_\epsilon$ to $p$. By \cite[Lemmas 3.14 \& 3.15]{GP21}, $r_\epsilon$ is a function of $\rho$ and smoothly extends across $D$. Let $h\in C^\infty(\bbR)$ be a non-negative smooth function; we say $h$ is \emph{linear adapted to $r_\epsilon$} of slope $\lambda\in\bbR_{\geq0}$ if the following three conditions hold: 
\begin{enumerate}
\item $h(r)=0$ for $r\leq e^{t_\epsilon}$, 
\item $\partial_rh\geq0$ and $\partial_r^2h\geq0$ for $r>e^{t_\epsilon}$, 
\item and there exists $K_\epsilon\in(1,\min_Dr_\epsilon)$, which is sufficiently close to 1, such that $h(r)=\lambda(r-1)$ for $r\geq K_\epsilon$.\footnote{Observe, $\min_Dr_\epsilon>1$ since $r_\epsilon\vert_{\partial\widehat{X}_\epsilon}=1$.}
\end{enumerate}
We may extend $h\circ r_\epsilon$ smoothly to a non-negative Hamiltonian on all of $M$, denoted 
    \begin{equation}
    h:M\to\bbR_{\geq0};
    \end{equation}
note, $h$ is linear outside the compact subset 
    \begin{equation}
    \widehat{X}^\circ_\epsilon\cup r_\epsilon^{-1}\big((-\infty,K_\epsilon]\big)
    \end{equation}
and its Hamiltonian flow preserves $D$. We refer to $K_\epsilon$, which we will assume is fixed, as the \emph{linearity level} of $h$. Finally, we fix an auxiliary constant $\mu_\epsilon\in(K_\epsilon,\min_Dr_\epsilon)$ to define two open neighborhoods of $D$:
    \begin{equation}
    V^1_\epsilon\equiv r^{-1}_\epsilon\big((K_\epsilon,+\infty)\big),\;\; V^0_\epsilon\equiv r^{-1}_\epsilon\big((\mu_\epsilon,+\infty)\big);
    \end{equation}
we refer to $V_\epsilon\equiv V^1_\epsilon-V^0_\epsilon$ as the \emph{shell region}.

Observe, the 1-periodic Hamiltonian orbits of $h$ come in two families. First, the \emph{divisorial} 1-periodic orbits $\chi(D;h)$ are completely contained in $D$. Second, the \emph{non-divisorial} 1-periodic orbits $\chi(X;h)$, which are all of the other 1-periodic orbits, come in two subfamilies themselves: 
\begin{itemize}
\item the constant 1-periodic orbits, which are seen to be the subset 
    \begin{equation}
    \chi_0(X;h)\equiv M-\Bigg\{p\in M:r_\epsilon(p)\geq\sup_{q\in M-\widehat{X}^\circ_\epsilon}\Big\{r_\epsilon(q):h\big(r_\epsilon(q)\big)=0\Big\}\Bigg\}\subset X;
    \end{equation}
\item and the non-constant 1-periodic orbits, which ``wind'' around $D$. These correspond to multiply-covered circles in any $D_{D,\delta}M\vert_p$. We will denote by $\chi_k(X;h)$ the 1-periodic orbits which wind exactly $k$ times. We will also denote by $\chi_{\leq k}(X;h)$ the 1-periodic orbits which wind at most $k$ times.
\end{itemize}
For any $x\in\chi(X;h)$, we refer to the non-negative integer 
    \begin{equation}
    \overline{w}(x)\equiv\kappa k(x)\in\bbZ_{\geq0},
    \end{equation}
such that $x\in\chi_{k(x)}(X;h)$, as the \emph{weighted winding number} of $x$, where we recall the definition of $\kappa$ in \eqref{eqn:kappa}. Of course, $k(x)$ is simply the \emph{winding number} of $x$.

By \cite[Section 4.1]{GP20}, we may perturb $h$, via a $C^2$-small time-dependent perturbation, into a(n $S^1$-dependent) Hamiltonian $H\in C^\infty(S^1\times M)$ with non-degenerate 1-periodic orbits such that 
\begin{enumerate}
\item the perturbation is disjoint from $V_\epsilon$ and, inside $M-V^1_\epsilon$, is supported in sufficiently small isolating neighborhoods of each $\chi_k(X;h)$, 
\item and the Hamiltonian flow of $H$ preserves $D$.
\end{enumerate}
We denote by $\chi(D;H)$ resp. $\chi(X;H)$ the set of divisorial resp. non-divisorial 1-periodic orbits; moreover, we call $H$ \emph{admissible} of slope $\lambda$.

Meanwhile, let $J$ be an $\omega$-tame almost complex structure on $M$. We say that $J$ is \emph{admissible} for $\widehat{X}_\epsilon$ if: 
\begin{enumerate}
\item $J$ preserves $TD$, 
\item the image of the Nijenhuis tensor, when restricted to $TM\vert_D$, is contained in $TD$, 
\item and $J$ is of contact type on the closure of $V_\epsilon$, i.e., $\theta\circ J=-dr_\epsilon$.
\end{enumerate}
Note, the space of such structures is non-empty and contractible via standard arguments, cf. for example \cite[Appendix A]{Ion15}.

\subsubsection{Symplectic cohomology}
We continue to look at data 
    \begin{equation}
    \big(\widehat{X}_\ell\equiv\widehat{X}_{\epsilon_\ell},H^\ell,J^\ell\big),\;\;\ell\in\bbZ_{\geq0}
    \end{equation}
as just defined; we will assume $\lambda_\ell$ is not the length of a Reeb chord on $\partial\widehat{X}_\ell$. For any $x\in\chi(X;H^\ell)$, we define
    \begin{equation}
    \deg(x)\equiv n-\operatorname{CZ}(x),
    \end{equation}
where $\operatorname{CZ}(\cdot)$ is the Conley-Zehnder index. Given any two $x,y\in\chi(X;H^\ell)$, we consider the moduli space $\widetilde{\scrF}^{H^\ell,J^\ell}(y,x)$ of Floer trajectories connecting $x$ to $y$, i.e., maps $u:\Theta\equiv\bbR_s\times S^1_t\to X$ satisfying 
    \begin{equation}
    \begin{cases}
    \partial_su+J^\ell\big(\partial_tu-X_{H^\ell}(u)\big)=0 & \\
    \lim_{s\to-\infty}u(s,t)=x(t) & \\
    \lim_{s\to+\infty}u(s,t)=y(t) &
    \end{cases}
    .
    \end{equation}
This is, for generic data, a smooth manifold of dimension $\deg(x)-\deg(y)$ whose tangent bundle is classified by the index bundle of the family of surjective Fredholm operators 
    \begin{equation}
    D\overline{\partial}_{H^\ell,J^\ell}\equiv\big\{D(\overline{\partial}_{H^\ell,J^\ell})_u:W^{1,2}(\Theta;u^*TM)\to L^2(\Theta;u^*TM)\big\}
    \end{equation}
given by linearizing. When $x\neq y$, there is a free proper $\bbR$-action on $\widetilde{\scrF}^{H^\ell,J^\ell}(y,x)$ given by time-shift in the $s$-coordinate; we denote by $\scrF^{H^\ell,J^\ell}(y,x)$ the $\bbR$-quotient. Let $\bbF^{H^\ell,J^\ell}(y,x)$ be the Gromov-compactification of $\scrF^{H^\ell,J^\ell}(y,x)$ given by allowing breakings at non-divisorial 1-periodic orbits. 

Let $\frako_x$ be the orientation line associated to $x$, i.e., the determinant line of the operator $D(\overline{\partial}_{H^\ell,J^\ell})_x$ given by linearizing at $x$, cf. \cite[Appendix C.6]{Abo10}. Note, $D(\overline{\partial}_{H^\ell,J^\ell})_x$ is the asymptotic operator of $D(\overline{\partial}_{H^\ell,J^\ell})_u$ at $-\infty$, for any Floer trajectory $u$ starting at $x$. We denote by $\abs{\frako_x}$ the free $\bbZ$-module generated by the two possible orientations of $\frako_x$ modulo the relation that the sum of opposite orientations vanishes. The $(H^\ell,J^\ell)$-Floer cochain complex 
    \begin{equation}
    CF^j(X;H^\ell,J^\ell;\bbZ)\equiv\bigoplus_{\deg(x)=j}\abs{\frako_x}
    \end{equation}
has the codifferential whose $\abs{\frako_y}-\abs{\frako_x}$ component is given by
    \begin{equation}
    \delta_{y,x}\equiv\sum_{\substack{u\in\bbF^{H^\ell,J^\ell}(y,x) \\ \deg(x)=\deg(y)+1}}\mu_u, 
    \end{equation}
where $\mu_u:\frako_y\xrightarrow{\sim}\frako_x$ is the isomorphism induced on orientation lines by a Floer trajectory; we denote by $HF^*(X;H^\ell,J^\ell;\bbZ)$ the $(H^\ell,J^\ell)$-Floer cohomology.

\begin{rem}
\emph{A priori}, $(H^\ell,J^\ell)$-Floer cohomology depends on the linearity level (i.e., $K_\ell$) and the shell region (i.e., $V_\ell$); \cite[Lemma 2.14]{GP20} shows it does not.
\end{rem}

One reason for introducing the topological limit model for symplectic cohomology is the following. For any $\overline{w}\in\bbR$, we define 
    \begin{equation}
    F_{\overline{w}}CF^*(X;H^\ell,J^\ell;\bbZ)\equiv\bigoplus_{x\in\chi(X;H^\ell),\overline{w}(x)\leq\overline{w}}\abs{\frako_x}.
    \end{equation}
A straightforward action estimate (which holds when $K_\ell$ is sufficiently close to 1 and $t_{\epsilon_\ell}$ is sufficiently small, cf. \cite[Lemma 2.10]{GP20}),
    \begin{equation}
    \calA_\ell(x)\equiv-\int_{S^1}x^*\theta+\int^1_0 H^\ell\big(t,x(t)\big)dt\approx-\overline{w}(x)\Bigg(1-\dfrac{\epsilon^2_\ell}{2}\Bigg),
    \end{equation}
where ``$\approx$'' indicates equality up to arbitrarily small error, shows the submodules \{$F_{\overline{w}}CF^*(X;H^\ell,J^\ell;\bbZ)\}$ give a filtration of $CF^*(X;H^\ell,J^\ell;\bbZ)$ when $\epsilon_\ell$ is sufficiently small. We refer to this filtration as the \emph{weight filtration}.

Now, consider two tuples of data:
    \begin{equation}
    \big(\widehat{X}_\ell,H^\ell,J^\ell\big),\;\;\big(\widehat{X}_{\ell'},H^{\ell'},J^{\ell'}\big),
    \end{equation}
where $\lambda_\ell<\lambda_{\ell'}$. We fix positive integers $\overline{w}_\ell,\overline{w}_{\ell'}\in\bbZ_{>0}$, $\overline{w}_\ell<\overline{w}_{\ell'}$, such that 
    \begin{equation}
    \overline{w}_\ell<\lambda_\ell<\overline{w}_\ell+1.
    \end{equation}
Finally, we may assume that $X_\ell\subset X_{\ell'}$ and $V_\ell\cap V_{\ell'}=\emptyset$. In the usual way, we may define a Floer continuation map 
    \begin{equation}
    \frakc_{\ell,\ell'}:HF^*(X;H^\ell,J^\ell;\bbZ)\to HF^*(X;H^{\ell'},J^{\ell'};\bbZ)
    \end{equation}
by looking at the isomorphisms induced on orientations lines from $x\in\chi(X;H^\ell)$ to $y'\in\chi(X;H^{\ell'})$ via Floer continuation cylinders, i.e., maps $u:\Theta\to X$ satisfying
    \begin{equation}
    \begin{cases}
    \partial_su+J_s\big(\partial_tu-X_{H_s}(u)\big)=0 & \\
    \lim_{s\to-\infty}u(s,t)=x(t) & \\
    \lim_{s\to+\infty}u(s,t)=y'(t) &
    \end{cases}
    .
    \end{equation}
Here, $(H_s,J_s)$ is Floer continuation data connecting $(H^\ell,J^\ell)$ to $(H^{\ell'},J^{\ell'})$. We denote by $\overline{\frakc}_{\ell,\ell'}(y',x)$ the moduli space of broken Floer continuation cylinders connecting $x$ to $y'$. 

\begin{rem}
Note, proving the existence of Floer continuation data in this case is slightly more subtle; \emph{a priori}, $\widehat{X}_\ell$ and $\widehat{X}_{\ell'}$ are different, so the continuation data is only monotonic ``up to arbitrarily small error''. Also, proving the Gromov-compactness statement in this case is slightly more subtle; it involves using the standard bubbling analysis in $M$, which is a closed symplectic manifold, to build a Gromov-compactification and then arguing (1) no sphere bubbling can occur and (2) no breaking along divisorial 1-periodic orbits can occur; cf. \cite[Section 2.3]{GP20}.
\end{rem}

Observe, the existence of Floer continuation maps immediately implies $(H^\ell,J^\ell)$-Floer cohomology is independent of $\epsilon_\ell$ once it is sufficiently small. We define 
    \begin{equation}
    SH^*(X;\bbZ)\equiv\varinjlim_\ell HF^*(X;H^\ell,J^\ell;\bbZ);
    \end{equation}
clearly, the Floer continuation maps respect the weight filtration, i.e., 
    \begin{equation}
    F_{\overline{w}}SH^*(X;\bbZ)\equiv\varinjlim_\ell F_{\overline{w}}HF^*(X;H^\ell,J^\ell;\bbZ)
    \end{equation}
gives a filtration of $SH^*(X;\bbZ)$ by weight. Finally, there is a straightforward isomorphism 
    \begin{equation}
    SH^*(X;\bbZ)\cong SH^*(\widehat{X}_\ell;\bbZ),
    \end{equation}
for any $\ell$, which is essentially constructed by considering Floer continuation data.\footnote{I.e., the topological limit model for symplectic cohomology actually computes symplectic cohomology.}

\subsection{Log cohomology}
Let $S_{D,\overline{\epsilon}}M\subset D_{D,\overline{\epsilon}}M$, where $\overline{\epsilon}\in\bbR_{>0}$ is sufficiently small, denote the associated circle bundle. Note, 
    \begin{equation}
    \partial\widehat{X}_{\overline{\epsilon}}=S_{D,\overline{\epsilon}}M;
    \end{equation}
moreover, we have the following chain of homotopy equivalences:
    \begin{equation}
    \widehat{X}_{\overline{\epsilon}}\simeq\widehat{X}^\circ_{\overline{\epsilon}}\simeq X.
    \end{equation}

\begin{rem}
For brevity, we sometimes will abuse notation and write $(X,\partial X)$ for $(\widehat{X}_{\overline{\epsilon}},\partial\widehat{X}_{\overline{\epsilon}})$ and $S_DM$ for $S_{D,\overline{\epsilon}}M$. Also, let $\pi_S:S_DM\to D$ be the projection.
\end{rem}
    
We fix a Morse function $f_X\in C^\infty(\widehat{X}_{\overline{\epsilon}})$ resp. $f_S\in C^\infty(S_{D,\overline{\epsilon}}M)$ and a (sufficiently generic) Riemannian metric $g_X$ on $\widehat{X}_{\overline{\epsilon}}$ resp. $g_S$ on $S_{D,\overline{\epsilon}}M$. We will assume $\grad f_X$ points strictly outwards on $\partial\widehat{X}_{\overline{\epsilon}}$. We will now recall the Morse-theoretic model for the log cohomology of $(M,D)$, cf. \cite[Section 3.2]{GP20}.

Given any critical point $a\in\crit(f_X)$, we denote by $I(a)$ the Morse index. We denote by $W^u(a;f_X)$ resp. $W^s(a;f_X)$ the unstable resp. stable manifold of $a$. Given any two $a,b\in\crit(f_X)$, we consider the moduli space $\widetilde{\scrD\scrM}_X(b,a)$ of Morse trajectories connecting $a$ to $b$, i.e., maps $\gamma:\bbR_s\to\widehat{X}_{\overline{\epsilon}}$ satisfying 
    \begin{equation}
    \begin{cases}
    \partial_s\gamma+\grad f(\gamma)=0 & \\
    \lim_{s\to-\infty}\gamma(s)=a & \\
    \lim_{s\to+\infty}\gamma(s)=b &
    \end{cases}
    .
    \end{equation}
This is, for generic data, a smooth manifold of dimension $I(a)-I(b)$ whose tangent bundle is classified by the index bundle of the family of surjective Fredholm operators 
    \begin{equation}
    D\Xi\equiv\big\{D\Xi_\gamma:W^{1,2}(\bbR;\gamma^*T\widehat{X}_{\overline{\epsilon}})\to L^2(\bbR;T\widehat{X}_{\overline{\epsilon}})\big\}
    \end{equation}
given by linearizing. When $a\neq b$, there is a free proper $\bbR$-action on $\widetilde{\scrD\scrM}_X(b,a)$ given by time-shift; we denote by $\scrD\scrM_X(b,a)$ the $\bbR$-quotient. Let $\bbD\bbM_X(b,a)$ be the Gromov-compactification of $\scrD\scrM_X(b,a)$ given by allowing breakings at critical points. The Morse cohomology $HM^*(X;\bbZ)$ is the homology of the cochain complex 
    \begin{equation}
    CM^j(X;\bbZ)\equiv\bigoplus_{I(a)=j}\abs{\frako_a}
    \end{equation}
with the standard codifferential using the induced isomorphism on orientation lines. Note, we have an identification
    \begin{equation}
    HM^*(X;\bbZ)\cong H^*(X;\bbZ).
    \end{equation}

The \emph{log cohomology} of $(M,D)$, denoted $H^*_\mathrm{log}(M,D)$, is the homology of the cochain complex 
    \begin{equation}
    C^*_\mathrm{log}(M,D)\equiv CM^*(X;\bbZ)\oplus\Bigg(\bigoplus_{k\geq1}CM^*(S_DM;\bbZ)\cdot t^k\Bigg),
    \end{equation}
where the $t^k$'s are formal variables, with the natural codifferential induced by the Morse codifferentials of the summands and grading given by
    \begin{equation}
    \begin{cases}
    \deg(\abs{\frako_a})\equiv I(a), & a\in\crit(f_X) \\
    \deg(\abs{\frako_a}t^k)\equiv I(a)+2k(1-\widetilde{\nu}), & a\in\crit(f_S)
    \end{cases}
    ,
    \end{equation}
where we recall the definition of $\widetilde{\nu}$ in \eqref{eqn:nutilde}. Finally, we observe the log cochain complex is filtered in a rather simple way: 
    \begin{equation}
    F_kC^*_\mathrm{log}(M,D)\equiv CM^*(X;\bbZ)\oplus\Bigg(\bigoplus_{k'\leq k}CM^*(S_DM;\bbZ)\cdot t^{k'}\Bigg).
    \end{equation}
   
\subsection{The spectral sequence}
We have already observed that $SH^*(X;\bbZ)$ has a weight filtration; there is a cochain-level construction underlying this as follows. Consider the cochain complex 
    \begin{equation}
    SC^*(X;\bbZ)\equiv\bigoplus_\ell CF^*(X;H^\ell,J^\ell;\bbZ)[\mathfrak{t}],
    \end{equation}
where $\mathfrak{t}$ is a formal variable of degree $-1$ satisfying $\mathfrak{t}^2=0$, together with the codifferential 
    \begin{equation}
    x_1+\mathfrak{t}x_2\mapsto(-1)^{\deg(x_1)}\delta_\ell x_1+(-1)^{\deg(x_2)}\big(\mathfrak{t}\delta_\ell x_2+\frakc_{\ell,\ell+1}(x_2)-x_2\big)
    \end{equation}
for $x_1+\mathfrak{t}x_2\in CF^*(X;H^\ell,J^\ell;\bbZ)[\mathfrak{t}]$. It is straightforward to see we have an identification  
    \begin{equation}
    H_*(SC^*(X;\bbZ))\cong SH^*(X;\bbZ).
    \end{equation}
In particular, the weight filtration determines a (cohomologically graded\footnote{We use the convention $F^pSC^*(X;\bbZ)\equiv F_{-p}SC^*(X;\bbZ)$.} and multiplicative) spectral sequence $E^{p,q}_1\Rightarrow SH^*(X;\bbZ)$ with $E_1$-page given by the associated graded of the weight filtration:
    \begin{equation}
    \bigoplus_q E^{p,q}_1\equiv\varinjlim_\ell Gr_{-p}HF^*(X;H^\ell,J^\ell;\bbZ).
    \end{equation}
    
\subsection{The low-energy log PSS morphism}
We are finally ready to recall the construction GP's low-energy log PSS morphism, cf. \cite[Section 3.3]{GP20}. We define $\Sigma$ to be the punctured Riemann surface $\bbC P^1-\{0\}$ with a negative cylindrical end $\epsilon^-$ near the puncture. Let $\eta\in\Omega^1(\Sigma)$ be a 1-form which (1) vanishes in a neighborhood of $+\infty$ and (2) restricts to $dt$ at $\epsilon^-$. Let $J^\ell_\Sigma\equiv\{J^\ell_z\}_{z\in\Sigma}$ be a $\Sigma$-dependent $\omega$-tame almost complex structure; we say it is \emph{admissible} for higher-order $H^\ell$-thimbles if: 
\begin{enumerate}
\item $J^\ell_\Sigma$ preserves $TD$,
\item the image of the Nijenhuis tensor, when restricted to $TM\vert_D$, is contained in $TD$,
\item $J^\ell_\Sigma$ is $z$-independent near $+\infty$,
\item and $J^\ell\equiv J^\ell_\Sigma\vert_{\epsilon^-}$ is admissible for $\widehat{X}_\ell$. 
\end{enumerate}
Note, the space of such structures is non-empty and contractible.

For any $x\in\chi(X;H^\ell)$ and $k\in\bbZ_{\geq0}$, we consider the moduli space $\scrR^{H^\ell,J^\ell_\Sigma}_{GP,k}(x)$ of maps $u:\Sigma\to M$ satisfying
    \begin{equation}
    \begin{cases}
    \big(du-X_\ell\otimes\eta\big)^{0,1}=0 & \\
    \lim_{s\to-\infty}u(\epsilon^-(s,t))=x(t) & \\
    u^{-1}(D)=\{+\infty\} & \\
    \textrm{$u$ intersects $D$ with multiplicity $k$ at $+\infty$} &
    \end{cases}
    .
    \end{equation}
This is, for generic data, a smooth manifold. For $k=0$, we have a natural evaluation map   
    \begin{equation}
    \eval^0_{+\infty}:\scrR^{H^\ell,J^\ell_\Sigma}_{GP,0}(x)\to\widehat{X}_{\overline{\epsilon}}.
    \end{equation}
For $k\geq1$, we have an enhanced evaluation map 
    \begin{equation}
    \eneval^k_{+\infty}:\scrR^{H^\ell,J^\ell_\Sigma}_{GP,k}(x)\to S_{D,\overline{\epsilon}}M
    \end{equation}
given by fixing an element in $S_{+\infty}\Sigma\equiv\big(T_{+\infty}\Sigma-\{0\}\big)/\bbR_{>0}$ and looking at the real-oriented projectivization of the $k$-th normal jet of $u$. 

We consider the various hybrid moduli spaces:
    \begin{equation}
    \begin{cases}
    \scrR^{H^\ell,J^\ell_\Sigma}_{GP,0}(a,x)\equiv(\eval^0_{+\infty})^{-1}\big(W^s(a;f_X)\big), & \\
    \scrR^{H^\ell,J^\ell_\Sigma}_{GP,k}(a,x)\equiv(\eneval^k_{+\infty})^{-1}\big(W^s(a;f_S)\big), & k\geq1
    \end{cases}
    .
    \end{equation}
These are, for generic data, smooth manifolds of dimension $\deg(x)-\deg(\abs{\frako_a}t^k)$. We have an energy estimate (which holds when $\epsilon_\ell$ is sufficiently small, cf. \cite[Lemma 3.14]{GP20}):
    \begin{equation}
    E_\mathrm{top}(u)\equiv\int_\Sigma u^*\omega\approx\overline{w}(k)-\overline{w}(x)\Bigg(1-\dfrac{\epsilon_\ell^2}{2}\Bigg),\;\;u\in\scrR^{H^\ell,J^\ell_\Sigma}_{GP,k}(a,x).
    \end{equation}
In particular, we have the following result.

\begin{lem}[Lemma 3.15 in \cite{GP20}]
Suppose $\overline{w}(k)<\lambda_\ell$ and $\epsilon_\ell$ is sufficiently small, then, for any $\abs{\frako_a}t^k\in C^*_\mathrm{log}(M,D)$ and $x\in\chi_{\overline{w}(k)}(X;H^\ell)$, we see: 
\begin{itemize}
\item if $\deg(x)=\deg(\abs{\frako_a}t^k)$, then $\scrR^{H^\ell,J^\ell_\Sigma}_{GP,k}(a,x)$ is compact;
\item and if $\deg(x)=\deg(\abs{\frako_a}t^k)+1$, then $\scrR^{H^\ell,J^\ell_\Sigma}_{GP,k}(a,x)$ admits a Gromov compactification by allowing 
\begin{enumerate}
\item breaking along non-divisorial 1-periodic orbits of weighted winding number equal to $\overline{w}(k)$, 
\item and breaking along critical points (of $f_X$ resp. $f_S$ if $k=0$ resp. $k\geq1$).
\end{enumerate}
\end{itemize}
\end{lem}
Hence, in the usual way, the various moduli spaces $\scrR^{H^\ell,J^\ell_\Sigma}_{GP,k}(a,x)$ assemble into a cochain map 
    \begin{equation}
    \lelogpss^{GP,k,\ell}:Gr_kC^*_\mathrm{log}(M,D)\to Gr_{\overline{w}(k)}CF^*(X;H^\ell,J^\ell;\bbZ)
    \end{equation}
which is compatible with Floer continuation maps, i.e., we have constructed GP's \emph{low-energy log PSS $k$-morphism}:
    \begin{equation}
    \lelogpss^{GP,k}:Gr_kH^*_\mathrm{log}(M,D)\to\bigoplus_qE^{\overline{w}(k),q}_1,
    \end{equation}
cf. \cite[Lemma 3.16]{GP20}. Finally, by taking a direct sum over $k$, we have constructed GP's \emph{low-energy log PSS morphism}:
    \begin{equation}
    \lelogpss^{GP}:H^*_\mathrm{log}(M,D)\to\bigoplus_{k,q}E^{\overline{w}(k),q}_1.
    \end{equation}
\cite[Theorem 1.1]{GP20} says both the low-energy log PSS $k$-morphism and the low-energy log PSS morphism are isomorphisms. 

\begin{rem}
In fact, \emph{loc. cit.} shows the low-energy log PSS morphism is an isomorphism of rings, where $\oplus_{p,q}E^{p,q}_1$ is endowed with the usual pair-of-pants product and $H^*_\mathrm{log}(M,D)$ is endowed with the ``usual'' product counting ``Y-shaped'' gradient flow trees which takes into account the $t^k$ variables, cf. \cite[Section 3.2]{GP20}.
\end{rem}

\section{Rapid review of Floer homotopy on Liouville manifolds}
In order to set notation, we will quickly review the basics of Floer homotopy.
\begin{enumerate}
\item We work in the Abouzaid-Blumberg \cite{AB24} framework for Floer homotopy.
\item By a smooth manifold with corners, we will always mean a ``$\langle k\rangle$-manifold'' (in the sense of \cite{Lau00}).
\item In the remainder of the present article, we work under Assumption \ref{assu:main}.
\end{enumerate}

\subsection{Morse homotopy}
We have that $\bbD\bbM_X(c,a)$ can be given the structure of a compact smooth manifold with corners whose codimension 1 boundary strata are enumerated by gluing maps of the form
    \begin{equation}
    \bbD\bbM_X(c,b)\times\bbD\bbM_X(b,a)\to\bbD\bbM_X(c,a),
    \end{equation}
cf. for instance \cite{Weh12}. Let $\bbD\bbM_X$ be the (unstructured) flow category\footnote{Cf. \cite[Definition 3.4]{AB24} or \cite[Definition 4.5]{Bla24}.} with objects $\crit(f_X)$ and morphism spaces $\bbD\bbM_X(b,a)$. Since we will use similar ideas in the sequel when constructing (twisted) stable framings on other moduli spaces, we will explain how to lift $\bbD\bbM_X$ to a framed flow category\footnote{Cf. \cite[Definition 3.7]{AB24} or \cite[Definition 4.7]{Bla24}.} a bit pedantically and skip the explanations later. There are standard stable framings
    \begin{equation}\label{eq:morseframing}
    T\bbD\bbM_X(b,a)+\underline{\bbR}+\underline{\bbR}^{-I(a)}\cong\underline{\bbR}^{-I(b)}
    \end{equation}
constructed as follows, cf. \cite[Section 4.3]{Bla24}. 

For notational convenience, we define 
    \begin{equation}
    \partial_b\equiv\bbD\bbM_X(c,b)\times\bbD\bbM_X(b,a).
    \end{equation}
We have the following two basic relations. First, a short exact sequence of real vector bundles
    \begin{equation}
    0\to\underline{\bbR}\to T\bbD\bbM_X(c,a)\vert_{\partial_b}\to T\bbD\bbM_X(c,b)\times T\bbD\bbM_X(b,a)\to0
    \end{equation}
given by taking a collar neighborhood of a codimension 1 boundary stratum. Second, a short exact sequence of real vector bundles 
    \begin{equation}
    0\to\underline{\bbR}\to\ind D\Xi^{ca}\to T\bbD\bbM_X(c,a)\vert_{\operatorname{int}\bbD\bbM_X(c,a)}\to0
    \end{equation}
given by the translational direction. There are straightforward relations when passing to higher codimension boundary strata.

\begin{prop}\label{prop:morseindex}
There is an extension of $\ind D\Xi^{ca}$ to $\bbD\bbM_X(c,a)$ whose restriction to the interior of a codimension 1 boundary stratum is of the form 
    \begin{equation}
    \ind D\Xi^{ca}\vert_{\operatorname{int}\partial_b}=\ind D\Xi^{cb}+\ind D\Xi^{ba}.
    \end{equation}
The natural diagrams associated to codimension 1 boundary strata,
    \begin{equation}
    \begin{tikzcd}[column sep=tiny]
    & & \underline{\bbR}\arrow[d] \\
    \underline{\bbR}\arrow[d,"\Delta"]\arrow[r] & \ind D\Xi^{ca}\vert_{\operatorname{int}\partial_b}\arrow[r]\arrow[d,equals] & T\bbD\bbM_X(c,a)\vert_{\operatorname{int}\partial_b}\arrow[d] \\
    \underline{\bbR}^2\arrow[r] & \ind D\Xi^{cb}\vert_{\operatorname{int}\partial_b}+\ind D\Xi^{ba}\vert_{\operatorname{int}\partial_b}\arrow[r] & T\bbD\bbM_X(c,b)\vert_{\operatorname{int}\partial_b}\times T\bbD\bbM_X(b,a)\vert_{\operatorname{int}\partial_b},
    \end{tikzcd}
    \end{equation}
commute. Moreover, the natural diagrams associated to higher codimension boundary strata commute.
\end{prop}

\begin{proof}
This follows from \cite[Section 7]{Lar21} (also, cf. \cite[Section 8]{PS24b} and \cite[Section 6]{PS25c}) which handles the more complicated Floer-theoretic case.
\end{proof}

Let $\gamma\in\bbD\bbM_X(b,a)$ and consider $D\Xi_\gamma$ together with its asymptotic operator $D\Xi_a$ resp. $D\Xi_b$ at $-\infty$ resp. $+\infty$. Consider a Fredholm operator 
    \begin{equation}
    T_{\frakM,a}:W^{1,2}(\bbR;a^*TX)\to L^2(\bbR;a^*TX)
    \end{equation}
with asymptotic operator $D\Xi_a$ resp. $\tau\cdot\identity$ at $-\infty$ resp. $+\infty$, where $\tau\in\bbR_{>0}$ is sufficiently small. Clearly, these define a constant family $T_{\frakM,a}$ of Fredholm operators on $\bbD\bbM_X(b,a)$. We have a canonical\footnote{Here, and always, canonical will mean up to a contractible space of choices.} isomorphism of real virtual bundles 
    \begin{equation}
    \ind T_{\frakM,b}\cong\underline{\bbR}^{I(a)}.
    \end{equation}
Consider the glued together operator 
    \begin{equation}
    T_{\frakM,a}^\vee\#D\Xi_\gamma\#T_{\frakM,b},
    \end{equation}
where the $\vee$ indicates the ``dual'' operator:
    \begin{equation}
    T_{\frakM,a}^\vee:W^{1,2}(\bbR;a^*TX)\to L^2(\bbR;a^*TX)
    \end{equation}
with asymptotic operator $\tau\cdot\identity$ resp. $D\Xi_a$ at $-\infty$ resp. $+\infty$. The glued together operator is canonically homotopic, in the space of index 0 Fredholm operators with fixed asymptotics, to the standard invertible operator $\partial_t+\tau\cdot\identity$. I.e., we have a canonical isomorphism of real virtual vector spaces
    \begin{equation}
    \ind D\Xi_\gamma+\bbR^{-I(a)}\cong\bbR^{-I(b)}\implies T_\gamma\bbD\bbM_X(b,a)+\bbR+\bbR^{-I(a)}\cong\bbR^{-I(b)}.
    \end{equation}
Since this isomorphism is canonical, we may upgrade it to families, i.e., we have a canonical isomorphism of real virtual bundles
    \begin{equation}
    T\bbD\bbM_X(b,a)+\underline{\bbR}+\underline{\bbR}^{-I(a)}\cong\underline{\bbR}^{-I(b)}.
    \end{equation}
Finally, this stable framing agrees with the already constructed stable framing on a codimension 1 boundary stratum when restricted to that boundary stratum by Proposition \ref{prop:morseindex} and canonical isomorphisms of the form 
    \begin{equation}
    \ind T_{\frakM,a}^\vee+\ind T_{\frakM,a}\cong0.
    \end{equation}

By \cite[Proposition 1.10]{AB24}, there is an equivalence of stable $\infty$-categories between the stable $\infty$-category $\spectra$ of spectra and the stable $\infty$-category $\flow^\fr$ of framed flow categories. We have that $\bbD\bbM_X$ corresponds to $\frakD X_+$, where we recall $\frakD$ represents taking the Spanier-Whitehead dual. Finally, given any real virtual bundle $E:X\to BO$, we may twist the standard framing on $\bbD\bbM_X$ to get a new framed flow category $\bbD\bbM_{X,E}$ by adding $E\vert_a$ resp. $E\vert_b$ to the left resp. right side of \eqref{eq:morseframing}; we have that $\bbD\bbM_{X,E}$ corresponds to $\frakD X^{-E}$.

\begin{rem}
The analogous statements in this subsection hold for $\bbD\bbM_{S_DM}$.
\end{rem}

\subsection{Spectral symplectic cohomology}\label{subsec:spectralSH}
We begin with the following lemma.

\begin{lem}\label{lem:index-theoretic}
Let $S_k$, $k\in\bbZ_{>0}$, be a genus 0 Riemann surface with $k$ (unordered) boundary components $\{\partial_1S_k,\ldots,\partial_kS_k\}$. Suppose $E\equiv F\otimes_\bbR\underline{\bbC}\to S_k$ is a complex vector bundle, of complex rank at least 3, such that $F$ is oriented and spin, then, for any Cauchy Riemann operator (with totally real boundary conditions)
    \begin{equation}
    D:W^{m,p}(S_k;E,F)\to W^{m-1,p}(S_k;E),
    \end{equation}
we have a canonical isomorphism of real virtual vector spaces
    \begin{equation}
    \ind D\cong F\vert_{z_0}^{\oplus 1-k},\;\;z_0\in S_k.
    \end{equation}
\end{lem}

\begin{proof}
Since we will use similar ideas in the sequel when gluing operators at interior points, we will explain the following a bit pedantically and skip the explanations later. First, consider the case $k=2$, i.e., $S_2\cong[-1,1]\times S^1$. Let $\Gamma_1:S^1\to S_2$ be a closed loop which is a transverse push-off of $\partial_1S_2$. Since $F\vert_{\Gamma_1}\to\Gamma_1$ is an oriented real vector bundle over a circle, we may choose a trivialization of $F\vert_{\Gamma_1}$ (observe, this is actually a choice; however, our proof will be independent of this choice). Since $F\vert_{\Gamma_1}$ is trivialized, we may collapse $\Gamma_1$ to a point in order to identify $D$ with the gluing of two operators 
    \begin{equation}
    D_j:W^{m,p}(\bbD_j;E\vert_{\bbD_j},F\vert_{\bbD_j})\to W^{m-1,p}(\bbD_j;E\vert_{\bbD_j}),\;\;j=1,2,
    \end{equation}
where (1) $\bbD_j$ is biholomorphic to the unit disk $\bbD\subset\bbC$ and (2) $\bbD_1$ and $\bbD_2$ glue together at the origin 0 to recover $S_2$. When gluing $D_1$ to $D_2$ at 0, we can only glue sections which agree at 0. In particular, the glued together operator 
    \begin{equation}
    D_1\#D_2:W^{m,p}(S_2;E,F)\to W^{m-1,p}(S_2;E),
    \end{equation}
which we may canonically deform to obtain $D$, has index bundle canonically isomorphic to 
    \begin{equation}
    \ind D\cong\ind D_1+\ind D_2-\Delta^\perp\cong F\vert_{z_1}+F\vert_{z_2}-E\vert_0,\;\;z_j\in\partial\bbD_j,
    \end{equation}
where the first isomorphism uses the canonical deformation and the second isomorphism uses the facts that (1) the totally real boundary conditions of $D_j$ extends over $\bbD_j$ and (2) the short exact sequence 
    \begin{equation}
    0\to E\vert_0\xrightarrow{\Delta} E\vert_0^{\oplus2}\to\Delta^\perp\to0.
    \end{equation}
When performing the connect sum $\bbD_1\#\bbD_2\cong S_2$, we have an identification 
    \begin{equation}
    E\vert_0\cong E\vert_{z_0}\cong F\vert_{z_0}^{\oplus2},\;\;z_0\in\Gamma_1,
    \end{equation}
where the second isomorphism uses the fact that $E$ is the complexification of $F$. Finally, since $F$ is a real spin vector bundle over a surface of real rank at least 3, we may choose a flat connection $\nabla$ on $F$ in order to parallel transport: 
    \begin{equation}
    F\vert_{z_1}+F\vert_{z_2}-F\vert_{z_0}^{\oplus2}\cong0.
    \end{equation}
We would like to remark that, as promised, the end result is independent of the choice of trivialization of $F\vert_{\Gamma_1}$ since each isomorphism of index bundles along the way was canonical. The lemma now follows in the case $k=2$; the general case follows by induction using analogous considerations.
\end{proof}

\begin{rem}
In the present article, we will make much use of Lemma \ref{lem:index-theoretic}. In particular, by stabilizing, we may assume $\Lambda$ has real rank at least 3.
\end{rem}

We have that $\bbF^{H^\ell,J^\ell}(z,x)$ can be given the structure of a compact smooth manifold with corners whose codimension 1 boundary strata are enumerated by gluing maps of the form
    \begin{equation}
    \bbF^{H^\ell,J^\ell}(z,y)\times\bbF^{H^\ell,J^\ell}(y,x)\to\bbF^{H^\ell,J^\ell}(z,x),
    \end{equation}
cf. for instance \cite[Section 6]{Lar21} (also, cf. \cite[Section 8]{PS24b} and \cite[Section 6]{PS25c}). Let $\bbF^{H^\ell,J^\ell}$ be the (unstructured) flow category with objects $\chi(X;H^\ell)$ and morphism spaces $\bbF^{H^\ell,J^\ell}(y,x)$. There are standard ($\Lambda$-dependent) stable framings
    \begin{equation}\label{eqn:floerframingstandard}
    T\bbF^{H^\ell,J^\ell}(y,x)+\underline{\bbR}+\ind T_{\frakF,x}\cong\ind T_{\frakF,y},
    \end{equation}
where $T_{\frakF,x}$ resp. $T_{\frakF,y}$ is to be defined, constructed as follows, cf. \cite[Section 5.4]{Bla24} for the open-string construction. 

For notational convenience, we define 
    \begin{equation}
    \partial_y\equiv\bbF^{H^\ell,J^\ell}(z,y)\times\bbF^{H^\ell,J^\ell}(y,x).
    \end{equation}
We have the following two basic relations. First, a short exact sequence of real vector bundles
    \begin{equation}
    0\to\underline{\bbR}\to T\bbF^{H^\ell,J^\ell}(z,x)\vert_{\partial_y}\to T\bbF^{H^\ell,J^\ell}(z,y)\times T\bbF^{H^\ell,J^\ell}(y,x)\to0
    \end{equation}
given by taking a collar neighborhood of a codimension 1 boundary stratum. Second, a short exact sequence of real vector bundles 
    \begin{equation}
    0\to\underline{\bbR}\to\ind D\overline{\partial}^{zx}_{H^\ell,J^\ell}\to T\bbF^{H^\ell,J^\ell}(z,x)\vert_{\operatorname{int}\bbF^{H^\ell,J^\ell}(z,x)}\to0
    \end{equation}
given by the translational direction. There are straightforward relations when passing to higher codimension boundary strata. 

\begin{prop}\label{prop:floerindex}
There is an extension of $\ind D\overline{\partial}^{zx}_{H^\ell,J^\ell}$ to $\bbF^{H^\ell,J^\ell}(z,x)$ whose restriction to the interior of a codimension 1 boundary stratum is of the form 
    \begin{equation}
    \ind D\overline{\partial}^{zx}_{H^\ell,J^\ell}\vert_{\operatorname{int}\partial_y}=\ind D\overline{\partial}^{zy}_{H^\ell,J^\ell}+\ind D\overline{\partial}^{yx}_{H^\ell,J^\ell}.
    \end{equation}
The natural diagrams associated to codimension 1 boundary strata,
    \begin{equation}
    \begin{tikzcd}[column sep=small, center picture]
    & & \underline{\bbR}\arrow[d] \\
    \underline{\bbR}\arrow[d,"\Delta"]\arrow[r] & \ind D\overline{\partial}^{zx}_{H^\ell,J^\ell}\vert_{\operatorname{int}\partial_y}\arrow[r]\arrow[d,equals] & T\bbF^{H^\ell,J^\ell}(z,x)\vert_{\operatorname{int}\partial_y}\arrow[d] \\
    \underline{\bbR}^2\arrow[r] & \ind D\overline{\partial}^{zy}_{H^\ell,J^\ell}\vert_{\operatorname{int}\partial_y}+\ind D\overline{\partial}^{yx}_{H^\ell,J^\ell}\vert_{\operatorname{int}\partial_y}\arrow[r] & T\bbF^{H^\ell,J^\ell}(z,y)\vert_{\operatorname{int}\partial_y}\times T\bbF^{H^\ell,J^\ell}(y,x)\vert_{\operatorname{int}\partial_y}
    ,
    \end{tikzcd}
    \end{equation}
commute. Moreover, the natural diagrams associated to higher codimension boundary strata commute.
\end{prop}

\begin{proof}
\cite[Section 7]{Lar21} (also, cf. \cite[Section 8]{PS24b} and \cite[Section 6]{PS25c}).
\end{proof}

Let $u\in\bbF^{H^\ell,J^\ell}(y,x)$ and consider $D(\overline{\partial}_{H^\ell,J^\ell})_u$ together with its asymptotic operator $D(\overline{\partial}_{H^\ell,J^\ell})_x$ resp. $D(\overline{\partial}_{H^\ell,J^\ell})_y$ at $-\infty$ resp. $+\infty$. First, consider a Cauchy Riemann operator 
    \begin{equation}
    \Phi^{\overline{\partial}}:W^{1,2}(\Theta;\underline{\bbC}^{d+1})\to L^2(\Theta;\underline{\bbC}^{d+1})
    \end{equation}
given by $\partial_s+i\partial_t+\tau\cdot\identity$. Clearly, these define a constant family $\Phi^{\overline{\partial}}$ of Fredholm operators on $\bbF^{H^\ell,J^\ell}(y,x)$. We have a canonical isomorphism 
    \begin{equation}
    \ind\Phi^{\overline{\partial}}\cong0.
    \end{equation}
Let $\Theta_+\equiv[0,+\infty)\times S^1$. Second, consider a Cauchy Riemann operator
    \begin{equation}
    T_{\frakF,x}:W^{1,2}\big(\Theta_+;x^*(TX\oplus\underline{\bbC}^{d+1}),x^*\Lambda\big)\to L^2\big(\Theta_+;x^*(TX\oplus\underline{\bbC}^{d+1}),x^*\Lambda\big)
    \end{equation}
with asymptotic operator $D(\overline{\partial}_{H^\ell,J^\ell})_x\oplus\Phi^{\overline{\partial}}$ at $+\infty$ and totally real boundary condition $x^*\Lambda$, where we abuse notation and also denote by $\Phi^{\overline{\partial}}$ its asymptotic operators. Clearly, these define a constant family $T_{\frakF,x}$ of Fredholm operators on $\bbF^{H^\ell,J^\ell}(y,x)$. Let $\Theta_-\equiv(-\infty,0]\times S^1$. We may also consider the ``dual'' Cauchy Riemann operator:
    \begin{equation}
    T_{\frakF,x}^\vee:W^{1,2}\big(\Theta_-;x^*(TX\oplus\underline{\bbC}^{d+1}),x^*\Lambda\big)\to L^2\big(\Theta_-;x^*(TX\oplus\underline{\bbC}^{d+1}),x^*\Lambda\big)
    \end{equation}
with asymptotic operator $D(\overline{\partial}_{H^\ell,J^\ell})_x\oplus\Phi^{\overline{\partial}}$ at $-\infty$ and totally real boundary condition $x^*\Lambda$. Again, these clearly define a constant family $T_{\frakF,x}^\vee$ of Fredholm operators on $\bbF^{H^\ell,J^\ell}(y,x)$. Lemma \ref{lem:index-theoretic} shows we have a canonical isomorphism 
    \begin{equation}\label{eqn:floerframingaux}
    \ind T_{\frakF,x}+\ind T_{\frakF,x}^\vee\cong0.
    \end{equation}
Consider the glued together operator 
    \begin{equation}
    T_{\frakF,x}\#D(\overline{\partial}_{H^\ell,J^\ell})\#T_{\frakF,y}^\vee,
    \end{equation}
then Lemma \ref{lem:index-theoretic} shows we have a canonical isomorphism of real virtual vector spaces 
    \begin{equation}
    \ind D(\overline{\partial}_{H^\ell,J^\ell})_u+\ind T_{\frakF,y}\cong\ind T_{\frakF,x}\implies T_u\bbF^{H^\ell,J^\ell}(y,x)+\bbR+\ind T_{\frakF,y}\cong\ind T_{\frakF,x}.
    \end{equation}
Since this isomorphism is canonical, we may upgrade it to families: 
    \begin{equation}\label{eqn:floerframingaux2}
    T\bbF^{H^\ell,J^\ell}(y,x)+\underline{\bbR}+\ind T_{\frakF,y}\cong\ind T_{\frakF,x}.
    \end{equation}
Finally, this stable framing agrees with the already constructed stable framing on a codimension 1 boundary stratum when restricted to that boundary stratum by Proposition \ref{prop:floerindex} and the canonical isomorphisms of the form \eqref{eqn:floerframingaux}.

We denote by $\bbF^{H^\ell,J^\ell,\Lambda}$ the framed flow category just described. Under the equivalence between spectra and framed flow categories, $\bbF^{H^\ell,J^\ell,\Lambda}$ determines a stable homotopy type $\frakF^{H^\ell,J^\ell,\Lambda}$ whose integral homology is the integral $(H^\ell,J^\ell)$-Floer cohomology with \emph{grading reversed and shifted up by $n$}: 
    \begin{equation}
    H_*(\frakF^{H^\ell,J^\ell,\Lambda};\bbZ)\cong HF^{-*+n}(X;H^\ell,J^\ell;\bbZ).
    \end{equation}

\begin{rem}\label{rem:shiftconvention}
Of course, we could remove the shifting up by $n$ via adding $-\underline{\bbR}^n$ to both sides of \eqref{eqn:floerframingstandard}; but this would just add annoyances for later since we would have to stabilize many future twisted stable framings by adding $-\underline{\bbR}^n$ to both sides.
\end{rem}

\begin{rem}
In the case that $\Lambda$ is not oriented/spin, our Floer homotopy type would recover $(H^\ell,J^\ell)$-Floer cohomology, with grading reversed and shifted up by $n$, with coefficients twisted by a local system, cf. \cite[Proposition 4.3]{Bla1}.
\end{rem}

The Floer continuation maps also admit framed flow bimodule\footnote{Cf. \cite[Definition 4.19]{AB24}.} lifts
    \begin{equation}
    \overline{\frakc}_{\ell,\ell'}:\bbF^{H^\ell,J^\ell,\Lambda}\to\bbF^{H^{\ell'},J^{\ell'},\Lambda}.
    \end{equation}
Explicitly, this means $\overline{\frakc}_{\ell,\ell'}(x,y)$ can be given the structure of a compact smooth manifold with corners whose codimension 1 boundary strata are enumerated by gluing maps of the form 
    \begin{align}
    \bbF^{H^\ell,J^\ell,\Lambda}(x,x')\times\overline{\frakc}_{\ell,\ell'}(x',y)&\to\overline{\frakc}_{\ell,\ell'}(x,y), \\
    \overline{\frakc}_{\ell,\ell'}(x,y')\times\bbF^{H^{\ell'},J^{\ell'},\Lambda}(y',y)&\to\overline{\frakc}_{\ell,\ell'}(x,y);
    \end{align}
moreover, there are standard ($\Lambda$-dependent) stable framings
    \begin{equation}\label{eqn:floercontinuationframing}
    T\overline{\frakc}_{\ell,\ell'}(x,y')+\ind T_{\frakF,y'}\cong\ind T_{\frakF,x}
    \end{equation}
which agree with the already constructed stable framing on a codimension 1 boundary stratum when restricted to that boundary stratum. These determine a stable homotopy type 
    \begin{equation}
    \frakF^\Lambda\equiv\varinjlim_\ell\frakF^{H^\ell,J^\ell,\Lambda}
    \end{equation}
whose integral homology is the integral symplectic cohomology of $X$ with \emph{grading reversed and shifted up by $n$}.

\begin{rem}
Again, in the case that $\Lambda$ is not oriented/spin, our Floer homotopy type would recover symplectic cohomology, with grading reversed and shifted up by $n$, with coefficients twisted by a local system.
\end{rem}

Finally, the weight filtration $F_{\overline{w}}SH^*(X;\bbZ)$ of $SH^*(X;\bbZ)$ lifts to a weight filtration $\frakF^\Lambda_{\overline{w}}$ of $\frakF^\Lambda$, as follows. Let $\bbF^{H^\ell,J^\ell,\Lambda}_{\overline{w}}$ be the framed flow subcategory of $\bbF^{H^\ell,J^\ell,\Lambda}$ such that the objects of $\bbF^{H^\ell,J^\ell,\Lambda}_{\overline{w}}$ are $x\in\chi_{\leq\overline{w}}(X;H^\ell)$; this determines a weight filtration of $\bbF^{H^\ell,J^\ell,\Lambda}$. Clearly, the Floer continuation framed flow bimodules respect the weight filtration:
    \begin{equation}
    \overline{\frakc}_{\ell,\ell'}:\bbF^{H^\ell,J^\ell,\Lambda}_{\overline{w}}\to\bbF^{H^{\ell'},J^{\ell'},\Lambda}_{\overline{w}},
    \end{equation}
i.e., 
    \begin{equation}
    \frakF^\Lambda_{\overline{w}}\equiv\varinjlim_\ell\frakF^{H^\ell,J^\ell,\Lambda}_{\overline{w}}
    \end{equation}
gives a filtration of $\frakF^\Lambda$ by weight.

\section{Spectral low-energy log PSS morphism}
In this section, we will describe how to spectrally lift GP's low-energy log PSS morphism to compute the associated graded of the weight filtration of $\frakF^\Lambda$.

\subsection{Canonical transverse operators}
There is a natural family of Fredholm operators $D^\Lambda$ on $S_DM$ defined as follows. Let $p\in S_DM$; we denote by $\gamma_p:S^1\to S_DM$ the loop going \emph{clockwise} around $D$ which starts at $p$ and covers the fiber $S_DM\vert_{\pi_S(p)}$ exactly once. Also, we denote by $v_p:\bbD\to D_DM$ the disk filling of $\gamma_p$ which covers the fiber $D_DM\vert_{\pi_S(p)}$ exactly once; note, $v_p(1)=p$ and $v_p\cdot D=1$.

\begin{defin}\label{defin:canonicaltransverseoperator}
A \emph{canonical transverse operator} at $p$ is a Cauchy Riemann operator 
    \begin{equation}
    D^\Lambda_p:W^{2,2}_0\big(\bbD;v_p^*(TM\oplus\underline{\bbC}^d\oplus\calL_m),\gamma_p^*\Lambda\big)\to W^{1,2}\big(\bbD;v_p^*(TM\oplus\underline{\bbC}^d\oplus\calL_m)\big),
    \end{equation}
where 
    \begin{equation}
    W^{2,2}_0\big(\bbD;v_p^*(TM\oplus\underline{\bbC}^d\oplus\calL_m),\gamma_p^*\Lambda\big)\subset W^{2,2}\big(\bbD;v_p^*(TM\oplus\underline{\bbC}^d\oplus\calL_m),\gamma_p^*\Lambda\big)
    \end{equation}
is the subspace of $W^{2,2}$-sections of $v_p^*(TM\oplus\underline{\bbC}^d\oplus\calL_m)$ over $\bbD$ with totally real boundary conditions $\gamma_p^*\Lambda$ and a constraint on sections at 0 to lie in $(TD\oplus\underline{\bbC}^d\oplus\calL_m)\vert_{\pi_S(p)}$. Observe, the totally real boundary conditions of $D^\Lambda_p$ go \emph{clockwise} around the disk since the restriction of $v_p$ to the boundary is oriented \emph{clockwise}.
\end{defin}

\begin{rem}
Observe, the (numerical) index of $D^\Lambda$ is equal to
    \begin{equation}
    n+d-1-2\sum_{\nu=1}^\alpha n_\nu\operatorname{rank}_\bbR F_\nu.
    \end{equation}
\end{rem}

\begin{prop}\label{prop:canonicaltransverseoperator}
We have a canonical isomorphism of real virtual bundles
    \begin{equation}
    \ind D^\Lambda\cong\pi_S^*\Bigg(-N_DM+\sum_{\nu=1}^\alpha F_\nu^{\oplus 1-2n_\nu}\Bigg).
    \end{equation}
\end{prop}

\begin{proof}
Observe, proving this proposition is equivalent to having the totally real boundary conditions of $D^\Lambda_p$ go \emph{counterclockwise} (i.e., having the restriction of $v_p$ to the boundary oriented \emph{counterclockwise}) and showing 
    \begin{equation}
    \ind D^\Lambda\cong\pi_S^*\Bigg(-N_DM+\sum_{\nu=1}^\alpha F_\nu^{\oplus 1+2n_\nu}\Bigg).
    \end{equation}
(Note, implicit in the previous statement is that switching the boundary orientation does not affect the constraint term related to sections at 0.) First, we have a canonical isomorphism of real virtual bundles
    \begin{equation}
    \ind D^\Lambda\cong-\pi_S^*N_DM+\ind\widetilde{D}^\Lambda,
    \end{equation}
where $\widetilde{D}^\Lambda$ is the family of Fredholm operators on $S_DM$ defined analogously to $D^\Lambda$ except we drop the constraint at 0. Now, we may canonically decompose $\ind\widetilde{D}^\Lambda$ as a direct sum $\oplus_\nu\ind\widetilde{D}^\Lambda_\nu$, where $\widetilde{D}^\Lambda_\nu$ is the family of Fredholm operators on $S_DM$ defined pointwise via
    \begin{equation}
    \widetilde{D}^\Lambda_{\nu,p}:W^{1,2}\big(\bbD;v^*_p(E_\nu\otimes_\bbC\calL_{n_\nu}),\gamma_p^*(F_\nu\otimes_\bbR\calL_{n_\nu,\bbR})\big)\to L^2\big(\bbD;v_p^*(E_\nu\otimes_\bbC\calL_{n_\nu})\big);
    \end{equation}
therefore, it suffices to prove we have a canonical isomorphism of real virtual bundles
    \begin{equation}
    \ind\widetilde{D}^\Lambda_\nu\cong\pi_S^*F_\nu^{\oplus 1+2n_\nu}
    \end{equation}
    
We have that $v^*_p\calL_{n_\nu}=\scrO_\bbD(n_\nu)$. Consider the following short exact sequence of locally free sheaves: 
    \begin{equation}
    0\to\scrO_\bbD\to\scrO_\bbD(n_\nu)\to\scrO_\bbD(n_\nu)\vert_0\to0,
    \end{equation}
where $\scrO_\bbD(n_\nu)\vert_0$ denotes the skyscraper sheaf supported at $0$ with stalk $\bbC^{n_\nu}$. The last map in the aforementioned short exact sequence is given by taking the residue at $0$ of a section of $\scrO_\bbD(n_\nu)$. We may now tensor the aforementioned short exact sequence by $v_p^*E_\nu$. Since tensoring via $v_p^*E_\nu$ preserves exactness, we have the associated long exact sequence in sheaf cohomology with totally real boundary conditions: 
    \begin{align}
    0\to\ker\overline{\partial}_{(v_p^*E_\nu,v_p^*F_\nu)}&\to\ker\overline{\partial}_{(v_p^*(E_\nu\otimes_\bbC\calL_{n_\nu}),v_p^* (F_\nu\otimes_\bbR\calL_{n_\nu,\bbR}))}\to \\
    H^0(\bbD,\scrO_\bbD(n_\nu)\vert_0\otimes v_p^*E_\nu)&\to\coker\overline{\partial}_{(v_p^*E_\nu,v_p^*F_\nu)}\to \nonumber \\
    \coker\overline{\partial}_{(v_p^*(E_\nu\otimes_\bbC\calL_{n_\nu}),v_p^*(F_\nu\otimes_\bbR\calL_{n_\nu,\bbR}))}&\to H^1(\bbD,\scrO_\bbD(n_\nu)\vert_0\otimes v_p^*E_\nu)\to\cdots; \nonumber
    \end{align}
here, we are using the identification between sheaf cohomology with totally real boundary conditions and twisted Dolbeut cohomology with totally real boundary conditions. Observe, $\scrO_\bbD(n_\nu)\vert_0\otimes v_p^*E_\nu$ is a skyscraper sheaf supported at 0 with stalk $E_\nu\vert_{\pi_S(p)}^{\oplus n_\nu}$, hence we may compute:
    \begin{align}
    H^1(\bbD,\scrO_\bbD(n_\nu)\vert_0\otimes v_p^*E_\nu)&=0, \\
    H^0(\bbD,\scrO_\bbD(n_\nu)\vert_0\otimes v_p^*E_\nu)&=E_\nu\vert_{\pi_S(p)}^{\oplus n_\nu};
    \end{align}
note, the latter equality depends on the parameterization of the disk (which we fixed to be $v_p$). Since the totally real boundary conditions of $v_p^*E_\nu$ extend across $\bbD$, we have the following canonical isomorphisms of real virtual vector spaces:
    \begin{align}
    \coker\overline{\partial}_{(v_p^*E_\nu,v_p^*F_\nu)}&\cong0, \\
    \ker\overline{\partial}_{(v_p^*E_\nu,v_p^*F_\nu)}&\cong F_\nu\vert_{\pi_S(p)};
    \end{align}
using the identification $E_\nu\vert_{\pi_S(p)}^{\oplus n_\nu}\cong F_\nu\vert_{\pi_S(p)}^{\oplus2n_\nu}$, the proposition follows.
\end{proof}

\subsection{Marked thimbles}
Recall, $\Sigma$ is $\bbC P^1-\{0\}$ together with a negative cylindrical end $\epsilon^-$. We define $\Sigma_k$, $k\in\bbZ_{\geq0}$, to be $\Sigma$ together with $k$ distinct marked points $z_1,\ldots,z_k\in\Sigma$. We require the marked points to be fixed. Let $\eta_k\in\Omega^1(\Sigma_k)$ be a 1-form which (1) vanishes in a neighborhood of each marked point and (2) restricts to $dt$ at $\epsilon^-$. Let $J^\ell_k\equiv\{J^\ell_{k,z}\}_{z\in\Sigma_k}$ be a $\Sigma_k$-dependent $\omega$-tame almost complex structure; we say it is \emph{admissible} for marked $H^\ell$-thimbles if: 
\begin{enumerate}
\item $J^\ell_k$ preserves $TD$,
\item the image of the Nijenhuis tensor, when restricted to $TM\vert_D$, is contained in $TD$,
\item $J^\ell_k$ is $z$-independent near each marked point,
\item and $J^\ell\equiv J^\ell_k\vert_{\epsilon^-}$ is admissible for $\widehat{X}_\ell$. 
\end{enumerate}
Note, the space of such structures is non-empty and contractible.

\begin{figure}[ht]
%% Creator: Inkscape 1.4.2 (f4327f4, 2025-05-13), www.inkscape.org
%% Accompanies image file 'markedthimble.pdf' (pdf, eps, ps)
\begingroup%
  \makeatletter%
  \providecommand\color[2][]{%
    \errmessage{(Inkscape) Color is used for the text in Inkscape, but the package 'color.sty' is not loaded}%
    \renewcommand\color[2][]{}%
  }%
  \providecommand\transparent[1]{%
    \errmessage{(Inkscape) Transparency is used (non-zero) for the text in Inkscape, but the package 'transparent.sty' is not loaded}%
    \renewcommand\transparent[1]{}%
  }%
  \providecommand\rotatebox[2]{#2}%
  \newcommand*\fsize{\dimexpr\f@size pt\relax}%
  \newcommand*\lineheight[1]{\fontsize{\fsize}{#1\fsize}\selectfont}%
  \ifx\svgwidth\undefined%
    \setlength{\unitlength}{197.84350538bp}%
    \ifx\svgscale\undefined%
      \relax%
    \else%
      \setlength{\unitlength}{\unitlength * \real{\svgscale}}%
    \fi%
  \else%
    \setlength{\unitlength}{\svgwidth}%
  \fi%
  \global\let\svgwidth\undefined%
  \global\let\svgscale\undefined%
  \makeatother%
  \begin{picture}(1,0.55539015)%
    \lineheight{1}%
    \setlength\tabcolsep{0pt}%
    \put(0.0199441,0.2554846){\color[rgb]{0,0,0}\makebox(0,0)[lt]{\lineheight{1.25}\smash{\begin{tabular}[t]{l}$x$\end{tabular}}}}%
    \put(0.63518612,0.40469456){\color[rgb]{0,0,0}\makebox(0,0)[lt]{\lineheight{1.25}\smash{\begin{tabular}[t]{l}$z_1$\end{tabular}}}}%
    \put(0.63020467,0.12807405){\color[rgb]{0,0,0}\makebox(0,0)[lt]{\lineheight{1.25}\smash{\begin{tabular}[t]{l}$z_k$\end{tabular}}}}%
    \put(0.91026282,0.00801864){\color[rgb]{0,0,0}\makebox(0,0)[lt]{\lineheight{1.25}\smash{\begin{tabular}[t]{l}$D$\end{tabular}}}}%
    \put(0,0){\includegraphics[width=\unitlength,page=1]{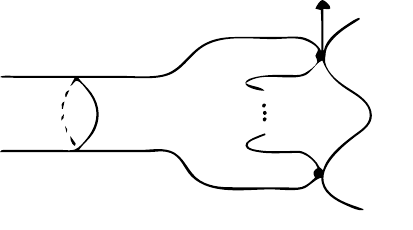}}%
  \end{picture}%
\endgroup%

\caption{This is a $k$-marked thimble, $k\geq1$. The arrow at $z_1$ indicates our auxiliary choice of real tangent ray made in order to define the enhanced evaluation map; in future figures, we will omit this choice from the picture.}
\label{fig:markedthimble}
\end{figure}
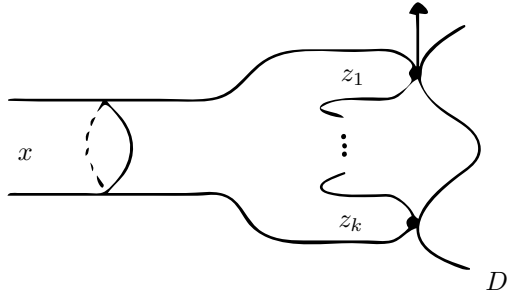

For any $x\in\chi(X;H^\ell)$, we consider the moduli space $\scrR^{H^\ell,J^\ell_k}(x)$ of \emph{$k$-marked $H^\ell$-thimbles}, i.e., maps $u:\Sigma_k\to M$ satisfying
    \begin{equation}
    \begin{cases}
    \big(du-X_\ell\otimes\eta_k\big)^{0,1}=0 & \\
    \lim_{s\to-\infty}u(\epsilon^-(s,t))=x(t) & \\
    u^{-1}(D)=\{z_1,\ldots,z_k\} & \\
    \textrm{$u$ intersects $D$ transversely at $z_1,\ldots,z_k$} &
    \end{cases}
    ,
    \end{equation}
see Figure \ref{fig:markedthimble}. Analogously to \cite[Section 4.4]{GP21} (which follows \cite[Section 6]{CM07}), we may look at a section of the appropriate Banach bundle over the space of $W^{m,p}$-maps ($mp>2$) from $\Sigma_k$ to $M$ which are (1) asymptotic to $x$ and (2) intersect $D$ only at the marked points (such that these intersections are transverse) to see that, for generic data, $\scrR^{H^\ell,J^\ell_k}(x)$ is a smooth manifold whose tangent bundle is classified by the index bundle of the family of surjective Fredholm operators 
    \begin{equation}
    D^{\scrR,k}\equiv\big\{D^{\scrR,k}_u:W^{2,2}_k(\Sigma_k;u^*TM)\to W^{1,2}(\Sigma_k;\Lambda^{0,1}_{\Sigma_k}\otimes u^*TM)\big\},
    \end{equation}
where $W^{2,2}_k(\Sigma_k;u^*TM)$ is the subspace of $W^{2,2}$-sections of $u^*TM$ over $\Sigma_k$ with a constraint to lie in $u^*TD$ at each marked point.

We will fix an element in $S_{z_1}\Sigma_k$ in order to define an enhanced evaluation map at $z_1$:
    \begin{equation}
    \eneval_{z_1}:\scrR^{H^\ell,J^\ell_k}(x)\to S_DM.
    \end{equation}
See Figure \ref{fig:markedthimble}.
    
Let $u\in\scrR^{H^\ell,J^\ell_k}(x)$; we will require two families of Fredholm operators. First, consider a Cauchy Riemann operator 
    \begin{equation}
    D_{\calL_m,k,u}:W^{1,2}(\Sigma_k;u^*\calL_m)\to L^2(\Sigma_k;u^*\calL_m)
    \end{equation}
with asymptotic operator $\partial_t+\tau\cdot\identity$ at the negative puncture; here, we use our fixed identification $\calL_m\vert_X\cong\underline{\bbC}$ given by our choice of $s_D$ to define the asymptotic operator. Clearly, these define a family of Fredholm operators $D_{\calL_m,k}$ on $\scrR^{H^\ell,J^\ell_k}(x)$. 
    
\begin{lem}\label{lem:hybridmarkedthimbles}
We have a canonical isomorphism of real virtual bundles
    \begin{equation}
    \ind D_{\calL_m,k}\cong\eval_{z_1}^*\calL_m^{\oplus k}+\underline{\bbR}^{1+k(2m-2)}.
    \end{equation}
\end{lem}

\begin{proof}
The proof is relatively straightforward, it just requires us to set up some notation. We consider a Cauchy Riemann operator 
    \begin{equation}
    D^{\calL_m,+}_x:W^{1,2}(\Theta_+;x^*\calL_m,x^*\calL_{m,\bbR})\to L^2(\Theta_+;x^*\calL_m)
    \end{equation}
with asymptotic operator $\partial_t+\tau\cdot\identity$ at the positive puncture. Clearly, these define a constant family of Fredholm operators $D^{\calL_m,+}_x$ on $\scrR^{H^\ell,J^\ell_k}(x)$ such that 
    \begin{equation}
    \ind D^{\calL_m,+}_x\cong0.
    \end{equation}
We also consider a Cauchy Riemann operator
    \begin{equation}
    D^{\calL_{m,\bbR}}_{p_j}:W^{1,2}(\bbD;v_{p_j}^*\calL_m,\gamma_{p_j}^*\calL_{m,\bbR})\to L^2(\bbD;v_{p_j}^*\calL_m),
    \end{equation}
where $p_j\in S_DM\vert_{u(z_j)}$ with $p_1=\eneval_{z_1}(u)$; recall, the totally real boundary conditions are oriented \emph{clockwise}. Note, our argument will ultimately be independent of the auxiliary choice of $p_j$. Analogously to the proof of Proposition \ref{prop:canonicaltransverseoperator}, we have a canonical isomorphism of real virtual vector spaces
    \begin{equation}
    \ind D^{\calL_{m,\bbR}}_{p_j}\cong\bbR^{1-2m}.
    \end{equation}
We will consider the glued together operator
    \begin{equation}
    D^{\calL_m,+}_x\#D_{\calL_m,k,u}\#\oplus_jD^{\calL_{m,\bbR}}_{p_j},
    \end{equation}
where the notation indicates we glue $D_{\calL_m,k,u}$ and $D^{\calL_{m,\bbR}}_{p_j}$ at $z_j$, whose index bundle is canonically isomorphic to the real virtual vector space
    \begin{equation}
    \ind D_{\calL_m,k,u}+\bbR^{k(1-2m)}-\calL_m\vert_{\eval_{z_1}(u)}^{\oplus k}.
    \end{equation}
Now, by Lemma \ref{lem:index-theoretic}, we also have that the index bundle of the aforementioned glued together operator is canonically isomorphic to the real virtual vector space $\calL_{m,\bbR}\vert_{\eneval_{z_1}(u)}^{\oplus 1-k}$; using the identification $\calL_{m,\bbR}\vert_X\cong\underline{\bbR}$ given by our choice of $s_D$, the lemma follows.
\end{proof}

Second, consider a Cauchy Riemann operator
    \begin{equation}
    \Phi^\scrR:W^{1,2}(\Sigma_k;\underline{\bbC}^d,\underline{\bbR}^d)\to L^2(\Sigma_k;\underline{\bbC}^d),
    \end{equation}
with asymptotic operator $\partial_t+\tau\cdot\identity$ at the negative puncture. Clearly, these define a constant family of Fredholm operators $\Phi^\scrR$ on $\scrR^{H^\ell,J^\ell_k}(x)$ such that
    \begin{equation}
    \ind\Phi^\scrR\cong\underline{\bbR}^d.
    \end{equation}

We may now construct a twisted stable framing of $\scrR^{H^\ell,J^\ell_k}(x)$. For notational convenience, we define
    \begin{align}
    T_k&\equiv TM^{\oplus k}+\underline{\bbC}^{dk}-N_DM^{\oplus k}-\sum_{\nu=1}^\alpha F_\nu^{\oplus k(1-2n_\nu)}, \\
    \widetilde{T}_k&\equiv\underline{\bbR}^{-d-1+2k(1-m)}+\sum_{\nu=1}^\alpha F_\nu^{\oplus 1-k}.
    \end{align}

\begin{rem}
Observe, $V_k=\widetilde{T}_k+T_k$.
\end{rem}

Again, let $u\in\scrR^{H^\ell,J^\ell_k}(x)$. Consider the glued together operator 
    \begin{equation}
    T_{\frakF,x}\#\big(D^{\scrR,k}_u\oplus\Phi^\scrR\oplus D_{\calL_m,k,u}\big)\#\oplus_j D^\Lambda_{p_j},
    \end{equation}
where $p_j\in S_DM\vert_{u(z_j)}$ with $p_1=\eneval_{z_1}(u)$. Again, our argument will ultimately be independent of the auxiliary choice of $p_j$. Our various index bundle computations show we have the following canonical isomorphism of real virtual bundles (after using the identification $\Lambda\vert_X\cong\oplus_\nu F_\nu$ induced by our choice of $s_D$):
\begin{itemize}
\item if $k=0$:
    \begin{equation}
    T\scrR^{H^\ell,J^\ell_0}(x)+\ind T_{\frakF,x}\cong\eval_{+\infty}^*\widetilde{T}_0;
    \end{equation}
\item if $k\geq1$:
    \begin{equation}
    T\scrR^{H^\ell,J^\ell_k}(x)+\ind T_{\frakF,x}\cong\eneval_{z_1}^*\widetilde{T}_k+\eval_{z_1}^*T_k.
    \end{equation}
\end{itemize}

\subsection{Compactifying low-energy marked thimbles}
We will now construct the Gromov-compactification $\overline{\scrR}^{H^\ell,J^\ell_{\overline{w}(x)}}(x)$ of $\scrR^{H^\ell,J^\ell_{\overline{w}(x)}}(x)$; we call this the \emph{low-energy compactification}. Again, we have an energy estimate (which holds when $\epsilon_\ell$ is sufficiently small):
    \begin{equation}
    E_\topological(u)\approx\overline{w}(k)-\overline{w}(x)\Bigg(1-\dfrac{\epsilon^2_\ell}{2}\Bigg),\;\;u\in\scrR^{H^\ell,J^\ell_k}(x).
    \end{equation} 
This can be seen in the same way as \cite[Lemma 3.14]{GP20}; we first approximate $E_\topological(u)$ in the complement of small disjoint disks surrounding the marked points, then we let the radii of these disks go to 0. In particular, the topological energy of $u\in\scrR^{H^\ell,J^\ell_{\overline{w}(x)}}(x)$ is arbitrarily small.

\begin{prop}\label{prop:lowenergycompactification}
For $\overline{w}(x)<\lambda_\ell$ and $\epsilon_\ell$ sufficiently small, the moduli space $\scrR^{H^\ell,J^\ell_{\overline{w}(x)}}(x)$ admits a Gromov-compactification $\overline{\scrR}^{H^\ell,J^\ell_{\overline{w}(x)}}(x)$ by allowing breakings at non-divisorial 1-periodic orbits with weighted winding number equal to $\overline{w}(x)$ such that the corresponding cylinder components are completely contained in $X$, i.e., the interior of a codimension $r\in\bbZ_{>0}$ boundary stratum is of the form 
    \begin{equation}
    \scrR^{H^\ell,J^\ell_{\overline{w}(y_r)}}(y_r)\times\scrF^{H^\ell,J^\ell}(y_r,y_{r-1})\times\cdots\times\scrF^{H^\ell,J^\ell}(y_1,x),\;\;\overline{w}(y_j)=\overline{w}(x)\;\;\forall j.
    \end{equation}
See Figure \ref{fig:brokenmarkedthimble}.
\end{prop}

\begin{proof}
Since $M$ is compact, standard Gromov-compactness shows that there is a Gromov-compactification $\overline{\scrR}^{H^\ell,J^\ell_{\overline{w}(x)}}(x)$ of $\scrR^{H^\ell,J^\ell_{\overline{w}(x)}}(x)$ whose elements consist of a marked thimble followed by a sequence of Floer trajectories together with various trees of sphere bubbles glued at various points. First, the argument of \cite[Lemma 4.14]{GP21} shows no breaking at divisorial 1-periodic orbits may occur. Second, our energy estimate shows that no sphere bubbling may occur. (Recall, our marked points on the thimble are fixed; hence, no sphere bubbling due to domain degenerations can occur.) In particular, any element of $\overline{\scrR}^{H^\ell,J^\ell_{\overline{w}(x)}}(x)$ has a well-defined intersection number with $D$ which must equal $\overline{w}(x)$. Therefore, again using our energy estimate and positivity of intersection, we see the only possible accumulation points of a sequence in $\scrR^{H^\ell,J^\ell_{\overline{w}(x)}}(x)$ are of the form specified in the statement of this proposition.
\end{proof}

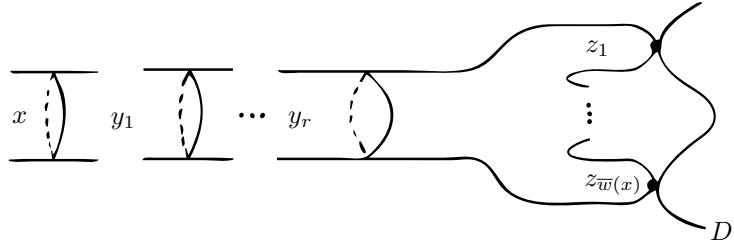
\begin{figure}[ht]
%% Creator: Inkscape 1.4.2 (f4327f4, 2025-05-13), www.inkscape.org
%% Accompanies image file 'brokenmarkedthimble.pdf' (pdf, eps, ps)
\begingroup%
  \makeatletter%
  \providecommand\color[2][]{%
    \errmessage{(Inkscape) Color is used for the text in Inkscape, but the package 'color.sty' is not loaded}%
    \renewcommand\color[2][]{}%
  }%
  \providecommand\transparent[1]{%
    \errmessage{(Inkscape) Transparency is used (non-zero) for the text in Inkscape, but the package 'transparent.sty' is not loaded}%
    \renewcommand\transparent[1]{}%
  }%
  \providecommand\rotatebox[2]{#2}%
  \newcommand*\fsize{\dimexpr\f@size pt\relax}%
  \newcommand*\lineheight[1]{\fontsize{\fsize}{#1\fsize}\selectfont}%
  \ifx\svgwidth\undefined%
    \setlength{\unitlength}{266.9357444bp}%
    \ifx\svgscale\undefined%
      \relax%
    \else%
      \setlength{\unitlength}{\unitlength * \real{\svgscale}}%
    \fi%
  \else%
    \setlength{\unitlength}{\svgwidth}%
  \fi%
  \global\let\svgwidth\undefined%
  \global\let\svgscale\undefined%
  \makeatother%
  \begin{picture}(1,0.37059743)%
    \lineheight{1}%
    \setlength\tabcolsep{0pt}%
    \put(0.00472965,0.16920435){\color[rgb]{0,0,0}\makebox(0,0)[lt]{\lineheight{1.25}\smash{\begin{tabular}[t]{l}$x$\end{tabular}}}}%
    \put(0.14400279,0.1679847){\color[rgb]{0,0,0}\makebox(0,0)[lt]{\lineheight{1.25}\smash{\begin{tabular}[t]{l}$y_1$\end{tabular}}}}%
    \put(0.39450001,0.16913743){\color[rgb]{0,0,0}\makebox(0,0)[lt]{\lineheight{1.25}\smash{\begin{tabular}[t]{l}$y_r$\end{tabular}}}}%
    \put(0.81394974,0.26941565){\color[rgb]{0,0,0}\makebox(0,0)[lt]{\lineheight{1.25}\smash{\begin{tabular}[t]{l}$z_1$\end{tabular}}}}%
    \put(0.81273026,0.08089931){\color[rgb]{0,0,0}\makebox(0,0)[lt]{\lineheight{1.25}\smash{\begin{tabular}[t]{l}$z_{\overline{w}(x)}$\end{tabular}}}}%
    \put(0.98992513,0.00475865){\color[rgb]{0,0,0}\makebox(0,0)[lt]{\lineheight{1.25}\smash{\begin{tabular}[t]{l}$D$\end{tabular}}}}%
    \put(0,0){\includegraphics[width=\unitlength,page=1]{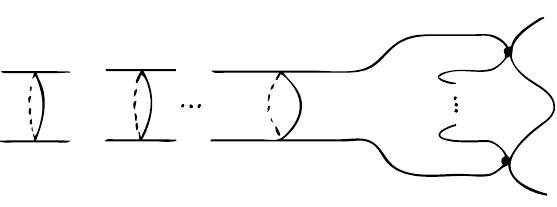}}%
  \end{picture}%
\endgroup%

\caption{This is a broken (low-energy) $k$-marked thimble. Note, each $y_j$ has winding number equal to $x$'s winding number.}
\label{fig:brokenmarkedthimble}
\end{figure}

We now move on to investigating the smooth structure of $\overline{\scrR}^{H^\ell,J^\ell_{\overline{w}(x)}}(x)$.

\begin{prop}\label{prop:smoothstructure0}
$\overline{\scrR}^{H^\ell,J^\ell_{\overline{w}(x)}}(x)$ can be given the structure of a compact smooth manifold with corners whose codimension 1 boundary is enumerated by gluing maps of the form
    \begin{equation}
    \overline{\scrR}^{H^\ell,J^\ell_{\overline{w}(y)}}(y)\times\bbF^{H^\ell,J^\ell}(y,x)\to\overline{\scrR}^{H^\ell,J^\ell_{\overline{w}(x)}}(x),\;\;\overline{w}(y)=\overline{w}(x).
    \end{equation}
\end{prop}

\begin{proof}
This follows as in \cite[Section 6]{Lar21} (also, cf. \cite[Section 8]{PS24b} and \cite[Section 6]{PS25c}); when considering the gluing of a Floer trajectory to a marked thimble, since the gluing is occurring in the complement of the divisor, the gluing analysis is exactly the same as the gluing analysis of cylindrical ends in \emph{loc. cit.}
\end{proof}

For notational convenience, we define 
    \begin{equation}
    \partial_{\overline{w}(x),yx}\equiv\overline{\scrR}^{H^\ell,J^\ell_{\overline{w}(y)}}(y)\times\bbF^{H^\ell,J^\ell}(y,x).
    \end{equation}
We have the following two basic relations. First, a short exact sequence of real vector bundles
    \begin{equation}
    0\to\underline{\bbR}\to T\overline{\scrR}^{H^\ell,J^\ell_{\overline{w}(x)}}(x)\vert_{\partial_{\overline{w}(x),yx}}\to T\overline{\scrR}^{H^\ell,J^\ell_{\overline{w}(y)}}(y)\times T\bbF^{H^\ell,J^\ell}(y,x)\to0
    \end{equation}
given by taking a collar neighborhood of a codimension 1 boundary stratum. Second, a short exact sequences of vector bundles 
    \begin{equation}
    0\to\underline{\bbR}\to\ind D\overline{\partial}_{H^\ell,J^\ell}\to T\bbF^{H^\ell,J^\ell}(y,x)\vert_{\operatorname{int}\bbF^{H^\ell,J^\ell}(y,x)}\to0
    \end{equation}
given by the translational direction. There are straightforward relations when passing to higher codimension boundary strata.

\begin{prop}
There is an extension of $\ind D^{\scrR,\overline{w}(x)}$ to $\overline{\scrR}^{H^\ell,J^\ell_{\overline{w}(x)}}(x)$ whose restriction to the interior of a codimension 1 boundary stratum is of the form
    \begin{equation}
    \ind D^{\scrR,\overline{w}(x)}\vert_{\operatorname{int}\partial_{\overline{w}(x),yx}}=\ind D^{\scrR,\overline{w}(y)}\vert_{\operatorname{int}\partial_{\overline{w}(x),yx}}+\ind\overline{\partial}_{H^\ell,J^\ell}\vert_{\operatorname{int}\partial_{\overline{w}(x),yx}}.
    \end{equation}
The natural diagrams associated to codimension 1 boundary strata,
    \begin{equation}
    \begin{tikzcd}[column sep=small, center picture]
    & & \underline{\bbR}\arrow[dd] \\
    & & \\
    & \ind D^{\scrR,\overline{w}(x)}\vert_{\operatorname{int}\partial_{\overline{w}(x),yx}}\arrow[r,"\sim"]\arrow[d,equals] & T\overline{\scrR}^{H^\ell,J^\ell_{\overline{w}(x)}}(x)\vert_{\operatorname{int}\partial_{\overline{w}(x),yx}}\arrow[d] \\
    \underline{\bbR}\arrow[r] & \ind D^{\scrR,\overline{w}(y)}\vert_{\operatorname{int}\partial_{\overline{w}(x),yx}}+\ind\overline{\partial}_{H^\ell,J^\ell}\vert_{\operatorname{int}\partial_{\overline{w}(x),yx}}\arrow[r] & T\overline{\scrR}^{H^\ell,J^\ell_{\overline{w}(y)}}(y)\vert_{\operatorname{int}\partial_{\overline{w}(x),yx}}\times T\bbF^{H^\ell,J^\ell}(y,x)\vert_{\operatorname{int}\partial_{\overline{w}(x),yx}},
    \end{tikzcd}
    \end{equation}
commute. Moreover, the natural diagrams associated to higher codimension boundary strata commute.
\end{prop}

\begin{proof}
Again, this follows as in \cite[Section 7]{Lar21} (also, cf. \cite[Section 8]{PS24b} and \cite[Section 6]{PS25c}) since the gluing analysis of Proposition \ref{prop:smoothstructure0} is identical.
\end{proof}

\begin{rem}
In words, the previous proposition is simply saying that the collar directions of a gluing map associated to a boundary stratum are identified with the translational directions in the direct sum of index bundles of that boundary stratum.
\end{rem}

\begin{cor}\label{corollary:extendingtsf}
Our constructed twisted stable framing on $\scrR^{H^\ell,J^\ell_{\overline{w}(x)}}(x)$ extends to a twisted stable framing on $\overline{\scrR}^{H^\ell,J^\ell_{\overline{w}(x)}}(x)$ which, when restricted to a codimension 1 boundary stratum, agrees with the already constructed twisted stable framing on that boundary stratum.
\end{cor}

\begin{proof}
This follows essentially because our twisted stable framings were constructed via gluing various Fredholm operators. We consider a codimension 1 boundary stratum of the form 
    \begin{equation}
    \overline{\scrR}^{H^\ell,J^\ell_{\overline{w}(y)}}(y)\times\bbF^{H^\ell,J^\ell}(y,x)\to\overline{\scrR}^{H^\ell,J^\ell_{\overline{w}(x)}}(x).
    \end{equation}
Let 
    \begin{equation}
    \big(u_2,[u_1]\big)\in\operatorname{int}\big(\overline{\scrR}^{H^\ell,J^\ell_{\overline{w}(y)}}(y)\times\bbF^{H^\ell,J^\ell}(y,x)\big).
    \end{equation}
The prescribed twisted stable framing at $\big(u_2,[u_1]\big)$ is achieved by considering the two glued together operators 
    \begin{align}
    T_{\frakF,x}\#\big(D(\overline{\partial}_{H^\ell,J^\ell})_{u_1}\oplus\Phi^{\overline{\partial}}\big)&\#T_{\frakF,y}^\vee, \\
    T_{\frakF,y}\#\big(D^{\scrR,\overline{w}(y)}_{u_2}\oplus\Phi^\scrR\oplus D_{\calL_m,\overline{w}(y),u_2}\big)&\#\oplus_j D^\Lambda_{p_j}
    \end{align}
and then adding together their index bundles. Observe, the sum of the index bundles of the two aforementioned glued together operators is canonically isomorphic to the index bundle of the glued together operator
    \begin{equation}
    T_{\frakF,x}\#\big(D(\overline{\partial}_{H^\ell,J^\ell})_{u_1}\oplus\Phi^{\overline{\partial}}\big)\# \\
    \big(D^{\scrR,\overline{w}(y)}_{u_2}\oplus\Phi^\scrR\oplus D_{\calL_m,\overline{w}(y),u_2}\big)\#\oplus_j D^\Lambda_{p_j}.
    \end{equation}
The corollary follows by observing the glued together operator 
    \begin{equation}
    \big(D(\overline{\partial}_{H^\ell,J^\ell})_{u_1}\oplus\Phi^{\overline{\partial}}\big)\#\big(D^{\scrR,\overline{w}(y)}_{u_2}\oplus\Phi^\scrR\oplus D_{\calL_m,\overline{w}(y),u_2}\big)
    \end{equation}
may be canonically deformed to the operator 
    \begin{equation}
    D^{\scrR,\overline{w}(x)}_{u_1\#u_2}\oplus\Phi^\scrR\oplus D_{\calL_m,\overline{w}(x),u_1\#u_2},
    \end{equation}
where $u_1\#u_2$ is the gluing of $u_1$ and $u_2$.
\end{proof}

\subsection{Hybrid low-energy marked thimbles}
First, for any $x\in\chi_0(X;H^\ell)$ and $a\in\crit(f_X)$, we consider the moduli space $\overline{\scrR}^{H^\ell,J^\ell_0}(a,x)$ of \emph{hybrid 0-marked $H^\ell$-thimbles}, i.e., the standard Gromov-compactification of 
    \begin{equation}
    \eval_{+\infty}^{-1}\big(W^s(a;f_X)\big),\;\;\eval_{+\infty}:\overline{\scrR}^{H^\ell,J^\ell_0}(x)\to X
    \end{equation}
given by allowing breaking at critical points. (Recall, elements of $\overline{\scrR}^{H^\ell,J^\ell_0}(x)$ \emph{do not} intersect $D$ at $+\infty$, see Figure \ref{fig:hybrid0markedthimble}.) This is, for generic data, a compact smooth manifold with corners. By splitting the short exact sequence 
    \begin{equation}
    0\to T\overline{\scrR}^{H^\ell,J^\ell_0}(a,x)\to T\overline{\scrR}^{H^\ell,J^\ell_0}(x)\to T\overline{W}^u(a;f_X)\to0,
    \end{equation}
we may endow $\overline{\scrR}^{H^\ell,J^\ell_0}(a,x)$ with a twisted stable framing, 
    \begin{equation}
    T\overline{\scrR}^{H^\ell,J^\ell_0}(a,x)+\ind T_{\frakF,x}\cong\underline{\bbR}^{-I(a)}+\eval_{+\infty}^*\widetilde{T}_0,
    \end{equation}
which, when restricted to a codimension 1 boundary stratum, agrees with the already constructed twisted stable framing on that boundary stratum. In particular, the moduli spaces of hybrid 0-marked $H^\ell$-thimbles assemble into a framed flow bimodule 
    \begin{equation}
    \lelogpss^{0,\ell}:\bbD\bbM_{X,V_0}\to\bbF^{H^\ell,J^\ell,\Lambda}_0.
    \end{equation}
In the standard way, we may check that $\lelogpss^{0,\ell}$ is compatible with Floer continuation framed flow bimodules by considering Floer data interpolating between (1) the gluing of Floer continuation data to hybrid 0-marked $H^\ell$-thimble Floer data and (2) hybrid 0-marked $H^{\ell+1}$-thimble Floer data. In particular, we have an induced map of spectra
    \begin{equation}
    \lelogpss^0:\frakD X^{-V_0}\to\frakF^\Lambda_0.
    \end{equation}

\begin{figure}[ht]
%% Creator: Inkscape 1.4.2 (f4327f4, 2025-05-13), www.inkscape.org
%% Accompanies image file 'hybrid0markedthimble.pdf' (pdf, eps, ps)
\begingroup%
  \makeatletter%
  \providecommand\color[2][]{%
    \errmessage{(Inkscape) Color is used for the text in Inkscape, but the package 'color.sty' is not loaded}%
    \renewcommand\color[2][]{}%
  }%
  \providecommand\transparent[1]{%
    \errmessage{(Inkscape) Transparency is used (non-zero) for the text in Inkscape, but the package 'transparent.sty' is not loaded}%
    \renewcommand\transparent[1]{}%
  }%
  \providecommand\rotatebox[2]{#2}%
  \newcommand*\fsize{\dimexpr\f@size pt\relax}%
  \newcommand*\lineheight[1]{\fontsize{\fsize}{#1\fsize}\selectfont}%
  \ifx\svgwidth\undefined%
    \setlength{\unitlength}{216.16376982bp}%
    \ifx\svgscale\undefined%
      \relax%
    \else%
      \setlength{\unitlength}{\unitlength * \real{\svgscale}}%
    \fi%
  \else%
    \setlength{\unitlength}{\svgwidth}%
  \fi%
  \global\let\svgwidth\undefined%
  \global\let\svgscale\undefined%
  \makeatother%
  \begin{picture}(1,0.16434759)%
    \lineheight{1}%
    \setlength\tabcolsep{0pt}%
    \put(0.01928818,0.07011026){\color[rgb]{0,0,0}\makebox(0,0)[lt]{\lineheight{1.25}\smash{\begin{tabular}[t]{l}$x$\end{tabular}}}}%
    \put(0.37898451,0.06791322){\color[rgb]{0,0,0}\makebox(0,0)[lt]{\lineheight{1.25}\smash{\begin{tabular}[t]{l}$+\infty$\end{tabular}}}}%
    \put(0.92414137,0.06038325){\color[rgb]{0,0,0}\makebox(0,0)[lt]{\lineheight{1.25}\smash{\begin{tabular}[t]{l}$a$\end{tabular}}}}%
    \put(0,0){\includegraphics[width=\unitlength,page=1]{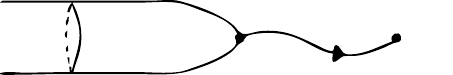}}%
  \end{picture}%
\endgroup%

\caption{This is a (low-energy) hybrid 0-marked thimble; here, $x$ has winding number 0.}
\label{fig:hybrid0markedthimble}
\end{figure}
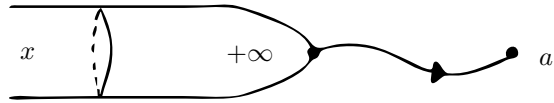

Second, for any $x\in\chi_k(X;H^\ell)$ with $k\geq1$ and $a\in\crit(f_S)$, we consider the moduli space $\overline{\scrR}^{H^\ell,J^\ell_{\overline{w}(x)}}(a,x)$ of \emph{hybrid $\overline{w}(x)$-marked $H^\ell$-thimbles}, i.e., the standard Gromov-compactification of 
    \begin{equation}
    \eneval_{z_1}^{-1}\big(W^s(a;f_S)\big),\;\;\eneval_{z_1}:\overline{\scrR}^{H^\ell,J^\ell_{\overline{w}(x)}}(x)\to S_DM
    \end{equation}
given by allowing breaking at critical points, see Figure \ref{fig:hybridkmarkedthimble}. This is, for generic data, a compact smooth manifold with corners. Again, by splitting the short exact sequence 
    \begin{equation}
    0\to T\overline{\scrR}^{H^\ell,J^\ell_{\overline{w}(x)}}(a,x)\to T\overline{\scrR}^{H^\ell,J^\ell_{\overline{w}(x)}}(x)\to T\overline{W}^u(a;f_S)\to0,
    \end{equation}
we may endow $\overline{\scrR}^{H^\ell,J^\ell_{\overline{w}(x)}}(a,x)$ with a twisted stable framing,
    \begin{equation}
    T\overline{\scrR}^{H^\ell,J^\ell_{\overline{w}(x)}}(a,x)+\ind T_{\frakF,x}\cong\underline{\bbR}^{-I(a)}+\eneval_{z_1}^*\widetilde{T}_{\overline{w}(x)}+\eval_{z_1}^*T_{\overline{w}(x)},
    \end{equation}
which, when restricted to a codimension 1 boundary stratum, agrees with the already constructed twisted stable framing on that boundary stratum. In particular, the moduli spaces of hybrid $k$-marked $H^\ell$-thimbles assemble into a framed flow bimodule 
    \begin{equation}
    \lelogpss^{k,\ell}:\bbD\bbM_{S_DM,V_{\overline{w}(k)}}\to\bbF^{H^\ell,J^\ell,\Lambda}_{k,k-1},
    \end{equation}
where $\bbF^{H^\ell,J^\ell,\Lambda}_{k,k-1}$ is the associated graded of the weight filtration of $\bbF^{H^\ell,J^\ell,\Lambda}$, cf. the end of Subsection \ref{subsec:spectralSH}; note, $\bbF^{H^\ell,J^\ell,\Lambda}_{k,k-1}$ is the framed flow category with objects $x\in\chi_k(X;H^\ell)$ and the obvious morphisms. Again, we may check that $\lelogpss^{k,\ell}$ is compatible with Floer continuation framed flow bimodules. In particular, we have an induced map of spectra
    \begin{equation}
    \lelogpss^k:\frakD (S_DM)^{-V_{\overline{w}(k)}}\to\frakF^\Lambda_{k,k-1}.
    \end{equation}

\begin{figure}[ht]
%% Creator: Inkscape 1.4.2 (f4327f4, 2025-05-13), www.inkscape.org
%% Accompanies image file 'hybridkmarkedthimble.pdf' (pdf, eps, ps)
\begingroup%
  \makeatletter%
  \providecommand\color[2][]{%
    \errmessage{(Inkscape) Color is used for the text in Inkscape, but the package 'color.sty' is not loaded}%
    \renewcommand\color[2][]{}%
  }%
  \providecommand\transparent[1]{%
    \errmessage{(Inkscape) Transparency is used (non-zero) for the text in Inkscape, but the package 'transparent.sty' is not loaded}%
    \renewcommand\transparent[1]{}%
  }%
  \providecommand\rotatebox[2]{#2}%
  \newcommand*\fsize{\dimexpr\f@size pt\relax}%
  \newcommand*\lineheight[1]{\fontsize{\fsize}{#1\fsize}\selectfont}%
  \ifx\svgwidth\undefined%
    \setlength{\unitlength}{238.29847441bp}%
    \ifx\svgscale\undefined%
      \relax%
    \else%
      \setlength{\unitlength}{\unitlength * \real{\svgscale}}%
    \fi%
  \else%
    \setlength{\unitlength}{\svgwidth}%
  \fi%
  \global\let\svgwidth\undefined%
  \global\let\svgscale\undefined%
  \makeatother%
  \begin{picture}(1,0.41260099)%
    \lineheight{1}%
    \setlength\tabcolsep{0pt}%
    \put(0.01838682,0.18855521){\color[rgb]{0,0,0}\makebox(0,0)[lt]{\lineheight{1.25}\smash{\begin{tabular}[t]{l}$x$\end{tabular}}}}%
    \put(0.4703086,0.29346563){\color[rgb]{0,0,0}\makebox(0,0)[lt]{\lineheight{1.25}\smash{\begin{tabular}[t]{l}$z_1$\end{tabular}}}}%
    \put(0.48160679,0.09171484){\color[rgb]{0,0,0}\makebox(0,0)[lt]{\lineheight{1.25}\smash{\begin{tabular}[t]{l}$z_{\overline{w}(k)}$\end{tabular}}}}%
    \put(0.93191478,0.29185166){\color[rgb]{0,0,0}\makebox(0,0)[lt]{\lineheight{1.25}\smash{\begin{tabular}[t]{l}$a$\end{tabular}}}}%
    \put(0.69773254,0.00665729){\color[rgb]{0,0,0}\makebox(0,0)[lt]{\lineheight{1.25}\smash{\begin{tabular}[t]{l}$D$\end{tabular}}}}%
    \put(0,0){\includegraphics[width=\unitlength,page=1]{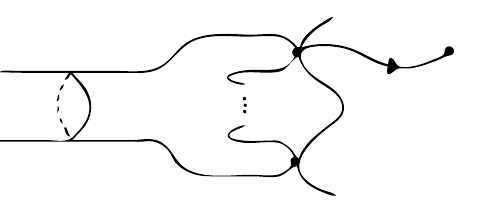}}%
  \end{picture}%
\endgroup%

\caption{This is a (low-energy) hybrid $k$-marked thimble, $k\geq1$; here, $x$ has winding number $k$. Technically, the curve going into $a$ should lie in $S_DM$.}
\label{fig:hybridkmarkedthimble}
\end{figure}
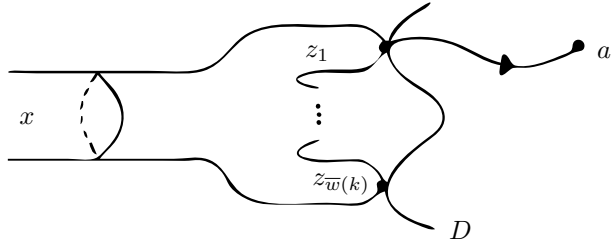
    
\subsection{Computing the associated graded}
This subsection is dedicated to proving the following result.

\begin{prop}\label{prop:he}
$\lelogpss^k$ is a homotopy equivalence.
\end{prop}

\begin{rem}
Before the proof, perhaps some remarks are in order.
\begin{enumerate}
\item At this point, the reader may be wondering why we opt to use thimbles with multiple marked points instead of thimbles with higher tangency constraints. The reason is simple (but perhaps will become clearer in Section \ref{sec:sgwo}): the compactification of (high-energy) moduli spaces with higher tangency constraints can exhibit orbifold points caused by sphere bubbling -- this makes building twisted stable framings more difficult.
\item Now, using thimbles with multiple marked points instead of thimbles with higher tangency constraints is good enough, at least in the present article, for the following reason: in low-energy moduli spaces, we may ``trade'' intersection points and multiplicity. In particular, the proof of Proposition \ref{prop:he} uses this idea to show our low-energy log PSS morphism is a homotopy equivalence by leveraging that GP's low-energy log PSS morphism is an isomorphism of rings. Related ideas can be found in \cite[Section 6]{PS24a}.
\item We would like to emphasize that we said ``in \emph{low-energy} moduli spaces, we may `trade' intersection points and multiplicity.'' In particular, this is simply not true in high-energy moduli spaces: Borman-Sheridan classes with higher-tangency constraints are not always products of Borman-Sheridan classes with lower tangency constraints, cf. \cite{Sei21}.
\end{enumerate}
\end{rem}

\begin{figure}
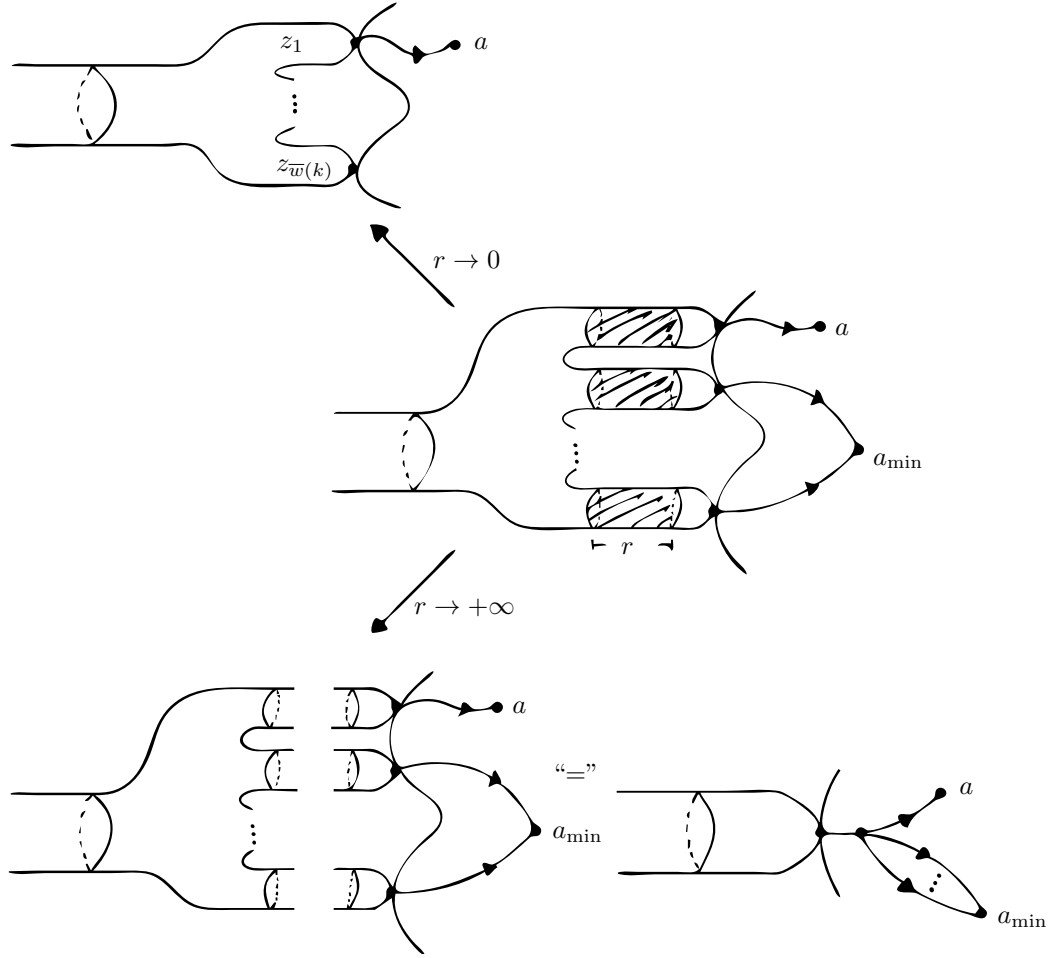
\caption{This is a schematic picture of the proof of Proposition \ref{prop:he}. In particular, the degenerations of the middle picture, with respect to the parameter $r$ measuring the length of the shaded region, to the top and bottom pictures depicts the commutativity of the diagram \eqref{eqn:commutativityfig}; moreover, the ``equality'' in the bottom picture depicts the equality \eqref{eqn:equalityfig}, i.e., the fact that GP's low-energy log PSS morphism is an isomorphism of rings. Note, on the left side of the ``equality'', each of the $\overline{w}(k)$ intersection points has multiplicity 1; meanwhile, on the right side of the ``equality'', the single intersection point has multiplicity $\overline{w}(k)$.}
\label{fig:bordismargumentlowenergy}
\end{figure}

\begin{proof}
First, given any $u\in\scrR^{H^\ell,J^\ell_k}(x)$, we denote by $c^{\mathrm{rel},k}_1$ the relative first Chern number of $u^*TM$. By the relation
    \begin{equation}
    c_1(M)=-c_1(\calL_m)+\sum_{\nu=1}^\alpha c_1(\calL_{n_\nu})\operatorname{rank}_\bbR F_\nu,
    \end{equation}
it is straightforward to see that
    \begin{equation}
    c^{\mathrm{rel},k}_1=-mk+k\sum_{\nu=1}^\alpha n_\nu\operatorname{rank}_\bbR F_\nu.
    \end{equation}

The main case of interest is when $k\geq1$; we will say a word about the $k=0$ case at the conclusion of the proof. Consider the induced map on homology:
    \begin{equation}
    \mathrm{LePSS}^k_{\log,*}:H^{-*}\big((S_DM)^{-V_{\overline{w}(k)}};\bbZ\big)\to SH^{-*+n}_{\overline{w}(k),\overline{w}(k-1)}(X;\bbZ).
    \end{equation}
Since $V_k$ is orientable, we may use the usual Thom isomorphism to view $\mathrm{LePSS}^k_{\log,*}$ as a map
    \begin{equation}
    H^{-*+2c^{\mathrm{rel},\overline{w}(k)}_1-2\overline{w}(k)}(S_DM;\bbZ)\to SH^{-*}_{\overline{w}(k),\overline{w}(k-1)}(X;\bbZ).
    \end{equation}
The above map is induced by the following cochain map. We define
    \begin{align}
    \lelogpss^{k,\ell}:CM^{-*+2c^{\mathrm{rel},\overline{w}(k)}_1-2\overline{w}(k)}(S_DM;\bbZ)&\to HF^{-*}_{\overline{w}(k),\overline{w}(k-1)}(X;H^\ell,J^\ell;\bbZ) \\
    \abs{\frako_a}&\mapsto\sum\mu_u\abs{\frako_x}, \nonumber
    \end{align}
where the sum is over all $u\in\overline{\scrR}^{H^\ell,J^\ell_{\overline{w}(k)}}(a,x)$ such that $x\in\chi_k(X;H^\ell)$ satisfies $\deg(x)+2c^{\mathrm{rel},\overline{w}(k)}_1-2\overline{w}(k)=\deg(a)$; clearly, this map (1) is compatible with the Floer continuation maps and (2),  after passing to cohomology and taking the colimit, equals $\mathrm{LePSS}^k_{\log,*}$. Meanwhile, we will consider a different cochain map. 
\begin{itemize}
\item First, we have a cochain map
    \begin{align}
    i_{\overline{w}(k)}:CM^{-*+2c^{\mathrm{rel},\overline{w}(k)}_1-2\overline{w}(k)}(S_DM;\bbZ)&\to CM^{-*+2c^{\mathrm{rel},1}_1-2}(S_DM;\bbZ)^{\otimes_\bbZ\overline{w}(k)} \\
    \abs{\frako_a}&\mapsto\sum\abs{\frako_a}\otimes\abs{\frako_{a_\mathrm{min}}}\otimes\cdots\otimes\abs{\frako_{a_\mathrm{min}}}, \nonumber
    \end{align}
where the sum is over the minima of $f$ (i.e., there is one minimum per component of $S_DM$).
\item Second, we have a cochain map
    \begin{equation}
    CM^{-*+2c^{\mathrm{rel},1}_1-2}(S_DM;\bbZ)^{\otimes_\bbZ\overline{w}(k)}\to HF^{-*}_{\overline{w}(k),\overline{w}(k-1)}(X;H^\ell,J^\ell;\bbZ)^{\otimes_\bbZ\overline{w}(k)}
    \end{equation}
given by the $k$-fold tensor product of GP's low-energy log PSS 1-morphism: 
    \begin{equation}
    \lelogpss^{GP,1,\ell}\otimes\cdots\otimes\lelogpss^{GP,1,\ell}.
    \end{equation}
\item Third, we have a cochain map
    \begin{equation}
    HF^{-*}_{\overline{w}(k),\overline{w}(k-1)}(X;H^\ell,J^\ell;\bbZ)^{\otimes_\bbZ\overline{w}(k)}\to HF^{-*}_{\overline{w}(k),\overline{w}(k-1)}(X;H^\ell,J^\ell;\bbZ)
    \end{equation}
given by the $\overline{w}(k)$-fold iterated pair-of-pants product:
    \begin{equation}
    \mu^{\overline{w}(k),\ell}_\mathrm{prod}.
    \end{equation}
\end{itemize}
We have that $\lelogpss^{k,\ell}$ is cochain homotopic to the composition 
    \begin{equation}
    \mu^{\overline{w}(k),\ell}_\mathrm{prod}\circ\otimes_k\lelogpss^{GP,1,\ell}\circ i_{\overline{w}(k)};
    \end{equation} 
this follows in the standard way using moduli spaces interpolating between hybrid $\overline{w}(k)$-marked $H^\ell$-thimbles and the composition, see Figure \ref{fig:bordismargumentlowenergy}. In particular, we have a commutative diagram 
    \begin{equation}\label{eqn:commutativityfig}
    \begin{tikzcd}[column sep=large]
    HM^{-*+2c^{\mathrm{rel},\overline{w}(k)}_1-2\overline{w}(k)}(S_DM;\bbZ)\arrow[r,"\mathrm{LePSS}^k_{\log,*}"]\arrow[d,"i_{\overline{w}(k),*}"] & HF^{-*}_{\overline{w}(k),\overline{w}(k-1)}(X;H^\ell,J^\ell;\bbZ) \\
    HM^{-*+2c^{\mathrm{rel},1}_1-2}(S_DM;\bbZ)^{\otimes_\bbZ\overline{w}(k)}\arrow[r,"\otimes_k\mathrm{LePSS}^{GP,1}_{\log,*}",swap] & HF^{-*}_{\overline{w}(k),\overline{w}(k-1)}(X;H^\ell,J^\ell;\bbZ)^{\otimes_\bbZ\overline{w}(k)}\arrow[u,"\mu^{\overline{w}(k),\ell}_\mathrm{prod}"].
    \end{tikzcd}
    \end{equation}
Finally, using the fact that $\mathrm{LePSS}^{GP,1}_{\log,*}$ is a map of rings, we compute: 
    \begin{align}\label{eqn:equalityfig}
    \Big(\mu^{\overline{w}(k),\ell}_\mathrm{prod}\circ\otimes_k\mathrm{LePSS}^{GP,1,\ell}_{\log,*}\circ i_{\overline{w}(k),*}\Big)(\alpha)&=\Big(\mu^{\overline{w}(k),\ell}_\mathrm{prod}\circ\otimes_k\mathrm{LePSS}^{GP,1}_{\log,*}\Big)(\alpha\otimes 1\cdots\otimes 1) \\
    &=\mathrm{LePSS}^{GP,k}_{\log,*}(\alpha), \nonumber
    \end{align}
i.e., we have a chain of equalities
    \begin{equation}
    \mathrm{LePSS}^k_{\log,*}=\mu^{\overline{w}(k),\ell}_\mathrm{prod}\circ\otimes_k\mathrm{LePSS}^{GP,1}_{\log,*}\circ i_{\overline{w}(k),*}=\mathrm{LePSS}^{GP,k}_{\log,*}.
    \end{equation}
Therefore, $\lelogpss^k$ is a map of spectra which induces an isomorphism on integral homology; since $\frakD (S_DM)^{-V_{\overline{w}(k)}}$ and $\frakF^\Lambda_{k,k-1}$ are bounded below spectra, this proposition follows after using a spectral Whitehead theorem for $H\bbZ$-module spectra. (The case $k=0$ is actually simpler than the case $k\geq1$ once we realize $\mathrm{LePSS}^0_{\log,*}$, after using the usual Thom isomorphism, is simply the usual PSS morphism.)
\end{proof}

\section{Spectral Gromov-Witten obstructions}\label{sec:sgwo}
\subsection{Outline of the argument}\label{subsec:outline}
This section is dedicated to proving Theorems \ref{thm:main} and \ref{thm:main2}. Since the argument for Theorem \ref{thm:main} is a bit long-winded, we will briefly describe the idea. Recall, we have the cofiber sequence
    \begin{equation}
    \frakF^\Lambda_{k-1}\to\frakF^\Lambda_k\to\frakF^\Lambda_{k,k-1},\;\;k\geq1.
    \end{equation}
A necessary and sufficient condition for this cofiber sequence to be a split cofiber sequence is that the Puppe connecting map, 
    \begin{equation}
    \frakF^\Lambda_{k,k-1}\to\Sigma\frakF^\Lambda_{k-1},
    \end{equation}
is null-homotopic. In this section, we will produce two maps 
    \begin{equation}
    \bbB^k_1:\frakD(S_DM)^{-V_{\overline{w}(k)}}\to\Sigma\frakD X^{-V_{\overline{w}(k)-1}},\;\;\bbB^k_2:\Sigma\frakD X^{-V_{\overline{w}(k)-1}}\to\Sigma\frakF^\Lambda_{k-1}
    \end{equation}
such that we have a homotopy commutative diagram
    \begin{equation}\label{eqn:outlinediagram}
    \begin{tikzcd}
    \frakD(S_DM)^{-V_{\overline{w}(k)}}\arrow[r,"\bbB^k_1"]\arrow[d,"\lelogpss^k"] & \Sigma\frakD X^{-V_{\overline{w}(k)-1}}\arrow[d,"\bbB^k_2"] \\
    \frakF^\Lambda_{k,k-1}\arrow[r] & \Sigma\frakF^\Lambda_{k-1}
    .
    \end{tikzcd}
    \end{equation}
Finally, we will observe that each $\bbB^k_2$ is actually equivalent to the same map:
    \begin{equation}\label{eqn:auxstructure}
    \calG\calW:\frakD(S_DM)^{-V^{\bbG\bbW}+TX}\to\Sigma\frakD X_+.
    \end{equation}
In particular, if the element
    \begin{equation}
    \calG\calW\in\pi_0\Big((S_DM)^{-V^{\bbG\bbW}+TX}\wedge\Sigma\frakD X_+\Big)
    \end{equation}
determined by \eqref{eqn:auxstructure} vanishes, then each Puppe connecting map is null-homotopic; Theorem \ref{thm:main} follows. We will now construct the relevant data and maps in earnest.

\subsection{Enhanced spheres}\label{subsec:enhancedspheres}
Let $A\in H_2(M;\bbZ)$ be a spherical homology class, i.e., in the image $H^S_2(M;\bbZ)$ of the Hurewicz map $\pi_2(M)\to H_2(M;\bbZ)$. Moreover, we assume $A\cdot D=1$. Let $J$ be an $\omega$-tame almost-complex structure such that 
\begin{enumerate}
\item $J$ preserves $TD$,
\item and the image of the Nijenhuis tensor, when restricted to $TM\vert_D$, is contained in $TD$.
\end{enumerate}
Note, the space of such structures is non-empty and contractible.

We consider the moduli space $\widetilde{\bbG\bbW}(A)$ of maps $u:\bbC P^1\to M$ satisfying 
    \begin{equation}
    \begin{cases}
    \partial_su+J\partial_tu=0 & \\
    u_*[\bbC P^1]=A & \\
    u^{-1}(D)=\{+\infty\} &
    \end{cases}
    .
    \end{equation}
Observe, the two conditions $A\cdot D=1$ and $u^{-1}(D)=\{+\infty\}$ implies $u$ intersects $D$ transversely at $+\infty$. This is, for generic data, a smooth manifold of dimension $2n+2c_1(A)-2$ whose tangent bundle is classified by the index bundle of the family of surjective Fredholm operators 
    \begin{equation}
    D^{\bbG\bbW}\equiv\big\{D^{\bbG\bbW}_u:W^{2,2}_{+\infty}(\bbC P^1;u^*TM)\to W^{1,2}(\bbC P^1;u^*TM)\big\},
    \end{equation}
where $W^{2,2}_{+\infty}(\bbC P^1;u^*TM)$ is the subspace of $W^{2,2}$ sections of $u^*TM$ over $\bbC P^1$ with a constraint at $+\infty$ to lie in $TD$, cf. \cite[Lemma 6.7]{CM07}; we will refer to this as the moduli space of \emph{parameterized 1-pointed relative spheres representing $A$}.

\begin{rem}
Note, we have that 
    \begin{equation}
    c_1(A)=-m+\sum_{\nu=1}^\alpha n_\nu\operatorname{rank}_\bbR F_\nu.
    \end{equation}
\end{rem}

There are two quotients of $\widetilde{\bbG\bbW}(A)$ we will consider. 
\begin{itemize}
\item First, consider the quotient 
    \begin{equation}
    \bbG\bbW(A)\equiv\widetilde{\bbG\bbW}(A)/\bbC^*,
    \end{equation}
where the $\bbC^*$-action is given by automorphisms of the domain fixing $0$ and $+\infty$. This is, for generic data, a smooth manifold of dimension $2n+2c_1(A)-4$; we will refer to this as the moduli space of \emph{(unparameterized) 1-pointed relative spheres representing $A$}.
\item Second, consider the quotient
    \begin{equation}
    \bbG\bbW^{S^1}(A)\equiv\widetilde{\bbG\bbW}(A)/\bbR,
    \end{equation}
where the $\bbR$-action is induced by the $\bbC^*$-action. This is, for generic data, a smooth manifold of dimension $2n+2c_1(A)-3$; we will refer to this as the moduli space of \emph{enhanced 1-pointed relative spheres representing $A$}. Clearly, the natural quotient map 
    \begin{equation}
    \bbG\bbW^{S^1}(A)\to\bbG\bbW(A)
    \end{equation}
turns $\bbG\bbW^{S^1}(A)$ into an $S^1$-bundle over $\bbG\bbW(A)$. Moreover, by fixing an element in $S_{+\infty}\bbC P^1$, we have an enhanced evaluation map at $+\infty$: 
    \begin{equation}
    \eneval_{+\infty}:\bbG\bbW^{S^1}(A)\to S_DM,
    \end{equation}
see Figure \ref{fig:enhancedsphere}.
\end{itemize}

\begin{figure}[ht]
%% Creator: Inkscape 1.4.2 (f4327f4, 2025-05-13), www.inkscape.org
%% Accompanies image file 'enhancedsphere.pdf' (pdf, eps, ps)
\begingroup%
  \makeatletter%
  \providecommand\color[2][]{%
    \errmessage{(Inkscape) Color is used for the text in Inkscape, but the package 'color.sty' is not loaded}%
    \renewcommand\color[2][]{}%
  }%
  \providecommand\transparent[1]{%
    \errmessage{(Inkscape) Transparency is used (non-zero) for the text in Inkscape, but the package 'transparent.sty' is not loaded}%
    \renewcommand\transparent[1]{}%
  }%
  \providecommand\rotatebox[2]{#2}%
  \newcommand*\fsize{\dimexpr\f@size pt\relax}%
  \newcommand*\lineheight[1]{\fontsize{\fsize}{#1\fsize}\selectfont}%
  \ifx\svgwidth\undefined%
    \setlength{\unitlength}{72.08308843bp}%
    \ifx\svgscale\undefined%
      \relax%
    \else%
      \setlength{\unitlength}{\unitlength * \real{\svgscale}}%
    \fi%
  \else%
    \setlength{\unitlength}{\svgwidth}%
  \fi%
  \global\let\svgwidth\undefined%
  \global\let\svgscale\undefined%
  \makeatother%
  \begin{picture}(1,1.13551355)%
    \lineheight{1}%
    \setlength\tabcolsep{0pt}%
    \put(0,0){\includegraphics[width=\unitlength,page=1]{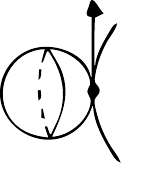}}%
    \put(0.83451038,0.01478761){\color[rgb]{0,0,0}\makebox(0,0)[lt]{\lineheight{1.25}\smash{\begin{tabular}[t]{l}$D$\end{tabular}}}}%
  \end{picture}%
\endgroup%

\caption{This is an enhanced sphere. The arrow at $z_1$ indicates our auxiliary choice of real tangent ray made in order to define the enhanced evaluation map (of course, we can not make such a choice if we had quotiented by the $S^1$-action); in future figures, we will omit this choice from the picture.}
\label{fig:enhancedsphere}
\end{figure}

We define 
    \begin{equation}
    \widetilde{\bbG\bbW}\equiv\coprod_{\substack{A\in H^S_2(M;\bbZ) \\ A\cdot D=1}}\widetilde{\bbG\bbW}(A);
    \end{equation}
we analogously define $\bbG\bbW^{S^1}$ and $\bbG\bbW$. The evaluation map at 0 resp. $+\infty$ on $\widetilde{\bbG\bbW}$, denoted $\eval_0$ resp. $\eval_{+\infty}$, straightforwardly descends to the quotients. The following result is a straightforward consequence of standard Gromov-compactness and positivity of intersection. (Cf. \cite[Lemma 3.31]{GP21}.)

\begin{lem}
$\eval_0:\bbG\bbW^{S^1}\to X$ resp. $\eval_0:\bbG\bbW\to X$ is proper.
\end{lem}

Let $u\in\widetilde{\bbG\bbW}$; we will require two families of Fredholm operators. First, consider a Cauchy Riemann operator 
    \begin{equation}
    D^{\bbG\bbW,\calL_m}_u:W^{1,2}(\bbC P^1;u^*\calL_m)\to L^2(\bbC P^1;u^*\calL_m).
    \end{equation}
Clearly, these define a family of Fredholm operators $D^{\bbG\bbW,\calL_m}$ on $\widetilde{\bbG\bbW}$.
    
\begin{lem}
We have a canonical isomorphism of real virtual bundles
    \begin{equation}
    \ind D^{\bbG\bbW,\calL_m}\cong\eval_{+\infty}^*\calL_m+\underline{\bbR}^{2m}.
    \end{equation}
\end{lem}

\begin{proof}
Appropriately modify the proof of Lemma \ref{lem:hybridmarkedthimbles}.
\end{proof}

Second, consider a Cauchy Riemann operator
    \begin{equation}
    \Phi^{\bbG\bbW}:W^{1,2}(\bbC P^1;\bbC^d)\to L^2(\bbC P^1;\bbC^d).
    \end{equation}
Clearly, these define a constant family of Fredholm operators $\Phi^{\bbG\bbW}$ on $\widetilde{\bbG\bbW}$ such that 
    \begin{equation}
    \ind\Phi^{\bbG\bbW}\cong\underline{\bbR}^{2d}.
    \end{equation}
    
We may now construct a twisted stable framing of $\bbG\bbW^{S^1}$. For notational convenience, we define
    \begin{align}
    T^{\bbG\bbW}&\equiv TM+\underline{\bbC}^d-N_DM-\sum_{\nu=1}^\alpha F_\nu^{\oplus 1-2n_\nu}, \\
    \widetilde{T}^{\bbG\bbW}&\equiv\underline{\bbR}^{-2d-2m}+\sum_{\nu=1}^\alpha F_\nu.
    \end{align}
    
\begin{rem}
Observe, $V^{\bbG\bbW}=\widetilde{T}^{\bbG\bbW}+T^{\bbG\bbW}$.
\end{rem}

Let $u\in\widetilde{\bbG\bbW}$. Consider the glued together operator 
    \begin{equation}
    \big(D^{\bbG\bbW}_u\oplus\Phi^{\bbG\bbW}\oplus D^{\bbG\bbW,\calL_m}_u\big)\#D^\Lambda_{\eneval_{+\infty}(u)}.
    \end{equation}
Our various index bundle computations, and the fact we have an identification $T\widetilde{\bbG\bbW}\cong T\bbG\bbW^{S^1}+\underline{\bbR}$, show we have the following canonical isomorphism of real virtual bundles: 
    \begin{equation}
    T\bbG\bbW^{S^1}+\underline{\bbR}\cong\eneval_{+\infty}^*\widetilde{T}^{\bbG\bbW}+\eval_{+\infty}^*T^{\bbG\bbW}.
    \end{equation}
    
\subsection{Partially-incident (PI) marked thimbles}\label{subsec:pimarkedthimbles}
Let $J^\ell_k$ be admissible for marked $H^\ell$-thimbles. For any $x\in\chi(X;H^\ell)$ and any subset $\bfJ\subset\{1,\ldots,k\}$, we consider the moduli space $\scrR^{H^\ell,J^\ell_k}_\bfJ(x)$ of \emph{$k$-marked $\bfJ$-partially-incident (PI) $H^\ell$-thimbles}, i.e., maps $u:\Sigma_k\to M$ satisfying
    \begin{equation}
    \begin{cases}
    \big(du-X_\ell\otimes\eta_k\big)^{0,1}=0 & \\
    \lim_{s\to-\infty}u(\epsilon^-(s,t))=x(t) & \\
    u^{-1}(D)=\big\{z_j:j\in\bfJ^c\equiv\{1,\ldots,k\}-\bfJ\big\} & \\
    \textrm{$u$ intersects $D$ transversely at $z\in\{z_j:j\in\bfJ^c\}$} &
    \end{cases}
    ,
    \end{equation}
see Figure \ref{fig:pimarkedthimble}. Again, we would like to emphasize that the marked points are fixed.
    
\begin{warning}
Reader beware, a $k$-marked $\bfJ$-PI $H^\ell$-thimble intersects $D$ at marked points in $\bfJ^c$!
\end{warning}

\begin{rem}
Of course, $\scrR^{H^\ell,J^\ell_k}_\emptyset(x)=\scrR^{H^\ell,J^\ell_k}(x)$.
\end{rem}

\begin{figure}[ht]
%% Creator: Inkscape 1.4.2 (f4327f4, 2025-05-13), www.inkscape.org
%% Accompanies image file 'pimarkedthimble.pdf' (pdf, eps, ps)
\begingroup%
  \makeatletter%
  \providecommand\color[2][]{%
    \errmessage{(Inkscape) Color is used for the text in Inkscape, but the package 'color.sty' is not loaded}%
    \renewcommand\color[2][]{}%
  }%
  \providecommand\transparent[1]{%
    \errmessage{(Inkscape) Transparency is used (non-zero) for the text in Inkscape, but the package 'transparent.sty' is not loaded}%
    \renewcommand\transparent[1]{}%
  }%
  \providecommand\rotatebox[2]{#2}%
  \newcommand*\fsize{\dimexpr\f@size pt\relax}%
  \newcommand*\lineheight[1]{\fontsize{\fsize}{#1\fsize}\selectfont}%
  \ifx\svgwidth\undefined%
    \setlength{\unitlength}{204.43249079bp}%
    \ifx\svgscale\undefined%
      \relax%
    \else%
      \setlength{\unitlength}{\unitlength * \real{\svgscale}}%
    \fi%
  \else%
    \setlength{\unitlength}{\svgwidth}%
  \fi%
  \global\let\svgwidth\undefined%
  \global\let\svgscale\undefined%
  \makeatother%
  \begin{picture}(1,0.65678217)%
    \lineheight{1}%
    \setlength\tabcolsep{0pt}%
    \put(0,0){\includegraphics[width=\unitlength,page=1]{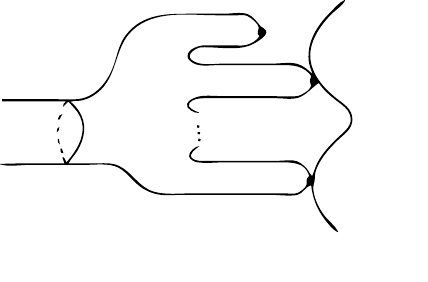}}%
    \put(0.4835991,0.5711816){\color[rgb]{0,0,0}\makebox(0,0)[lt]{\lineheight{1.25}\smash{\begin{tabular}[t]{l}$z_1$\end{tabular}}}}%
    \put(0.4879669,0.4576207){\color[rgb]{0,0,0}\makebox(0,0)[lt]{\lineheight{1.25}\smash{\begin{tabular}[t]{l}$z_2$\end{tabular}}}}%
    \put(0.4923347,0.23049846){\color[rgb]{0,0,0}\makebox(0,0)[lt]{\lineheight{1.25}\smash{\begin{tabular}[t]{l}$z_k$\end{tabular}}}}%
    \put(0.03517829,0.33677984){\color[rgb]{0,0,0}\makebox(0,0)[lt]{\lineheight{1.25}\smash{\begin{tabular}[t]{l}$x$\end{tabular}}}}%
    \put(0.81094303,0.09771132){\color[rgb]{0,0,0}\makebox(0,0)[lt]{\lineheight{1.25}\smash{\begin{tabular}[t]{l}$D$\end{tabular}}}}%
  \end{picture}%
\endgroup%

\caption{This is a $k$-marked $\{2,\ldots,k\}$-PI thimble, $k\geq2$.}
\label{fig:pimarkedthimble}
\end{figure}
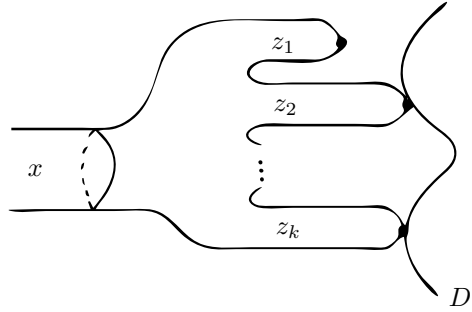

Again, we see that, for generic data, $\scrR^{H^\ell,J^\ell_k}_\bfJ(x)$ is a smooth manifold whose tangent bundle is classified by the index bundle of the family of surjective Fredholm operators 
    \begin{equation}
    D^{\scrR,\bfJ}\equiv\big\{D^{\scrR,\bfJ}_u:W^{2,2}_\bfJ(\Sigma_k;u^*TM)\to W^{1,2}(\Sigma_k;\Lambda^{0,1}_{\Sigma_k}\otimes u^*TM)\big\},
    \end{equation}
where $W^{2,2}_\bfJ(\Sigma_k;u^*TM)$ is the subspace of $W^{2,2}$-sections of $u^*TM$ over $\Sigma_k$ with a constraint to lie in $u^*TD$ at $z\in\{z_j:\in\bfJ^c\}$.

We will fix (if $1\in\bfJ^c$) an element in $S_{z_1}\Sigma_k$ in order to define an enhanced evaluation map at $z_1$:
    \begin{equation}
    \eneval_{z_1}:\scrR^{H^\ell,J^\ell_k}_\bfJ(x)\to S_DM.
    \end{equation}
    
Let $u\in\scrR^{H^\ell,J^\ell_k}_\bfJ(x)$; we will require a single family of Fredholm operators. Consider a Cauchy Riemann operator 
    \begin{equation}
    D_{\calL_m,\bfJ,u}:W^{1,2}(\Sigma_k;u^*\calL_m)\to L^2(\Sigma_k;u^*\calL_m)
    \end{equation}
with asymptotic operator $\partial_t+\tau\cdot\identity$ at the negative puncture; here, we use our fixed identification $\calL_m\vert_X\cong\underline{\bbC}$ given by our choice of $s_D$ to define the asymptotic operator. Clearly, these define a family of Fredholm operators $D_{\calL_m,\bfJ}$ on $\scrR^{H^\ell,J^\ell_k}_\bfJ(x)$.

\begin{lem}
We have a canonical isomorphism of real virtual bundles
    \begin{equation}
    \ind D_{\calL_m,\bfJ}\cong\eval_{z_1}^*\calL_m^{\oplus\abs{\bfJ^c}}+\underline{\bbR}^{1+\abs{\bfJ^c}(2m-2)}.
    \end{equation}
\end{lem}

\begin{proof}
Appropriately modify the proof of Lemma \ref{lem:hybridmarkedthimbles}.
\end{proof}

We may now construct a twisted stable framing of $\scrR^{H^\ell,J^\ell_k}_\bfJ(x)$. Again, let $u\in\scrR^{H^\ell,J^\ell_k}_\bfJ(x)$. Consider the glued together operator 
    \begin{equation}
    T_{\frakF,x}\#\big(D^{\scrR,\bfJ}_u\oplus\Phi^\scrR\oplus D_{\calL_m,\bfJ}\big)\#\oplus_{j\in\bfJ^c}D^\Lambda_{p_j},
    \end{equation}
where $p_j\in S_DM\vert_{u(z_j)}$ (with $p_1=\eneval_{z_1}(u)$ if $1\in\bfJ^c$). Again, our argument will ultimately be independent of the auxiliary choice of $p_j$. Our various index bundle computations show we have the following canonical isomorphism of real virtual bundles:
\begin{itemize}
\item if $1\in\bfJ$:
    \begin{equation}\label{eqn:auxredefining1}
    T\scrR^{H^\ell,J^\ell_k}_\bfJ(x)+\ind T_{\frakF,x}\cong\eval_{z_1}^*(\widetilde{T}_{\abs{\bfJ^c}}+T_{\abs{\bfJ^c}});
    \end{equation}
\item if $1\in\bfJ^c$:
    \begin{equation}
    T\scrR^{H^\ell,J^\ell_k}_\bfJ(x)+\ind T_{\frakF,x}\cong\eneval_{z_1}^*\widetilde{T}_{\abs{\bfJ^c}}+\eval_{z_1}^*T_{\abs{\bfJ^c}}.
    \end{equation}
\end{itemize}

\subsection{Compactifying marked PI thimbles}
In the present article, we will need to consider two Gromov-compactifications of $\scrR^{H^\ell,J^\ell_k}_\bfJ(x)$. 

The first compactification, denoted $\overline{\scrR}^{H^\ell,J^\ell_k}_\bfJ(x)$, we call the \emph{naive compactification}. Again, we have an energy estimate (which holds when $\epsilon_\ell$ is sufficiently small):
    \begin{equation}
    E_\topological(u)\approx\abs{\bfJ^c}-\overline{w}(x)\Bigg(1-\dfrac{\epsilon^2_\ell}{2}\Bigg).
    \end{equation}

\begin{prop}\label{prop:naivecompactification}
For $\overline{w}(x)<\lambda_\ell$ and $\epsilon_\ell$ sufficiently small, the moduli space $\scrR^{H^\ell,J^\ell_k}_\bfJ(x)$ admits a Gromov-compactification $\overline{\scrR}^{H^\ell,J^\ell_k}_\bfJ(x)$ by allowing
\begin{enumerate}
\item breakings at non-divisorial 1-periodic orbits such that the corresponding cylinder components are completely contained in $X$
\item and bubbling, at any $z\in\{z_j:j\in\bfJ^c\}$, of a 1-pointed relative sphere representing a spherical homology class $A$ satisfying $A\cdot D=1$.
\end{enumerate}
See Figure \ref{fig:convergence}.
\end{prop}

\begin{proof}
Since $M$ is compact, standard Gromov-compactness shows that there is a Gromov-compactification $\overline{\scrR}^{H^\ell,J^\ell_k}_\bfJ(x)$ of $\scrR^{H^\ell,J^\ell_k}_\bfJ(x)$ whose elements consist of a marked PI thimble followed by a sequence of Floer trajectories together with various trees of sphere bubbles glued at various points. First, the argument of \cite[Lemma 4.14]{GP21} shows no breaking at divisorial 1-periodic orbits may occur. Second, positivity of intersection shows no bubbling of a sphere bubble contained completely inside of $D$ can occur. In particular, any element of $\overline{\scrR}^{H^\ell,J^\ell_k}_\bfJ(x)$ has a well-defined intersection number with $D$ which we see must equal $\abs{\bfJ^c}$; hence, we cannot bubble a multiply covered sphere. (Recall, our marked points on the thimble are fixed; hence, no sphere bubbling due to domain degenerations can occur.) Therefore, again using positivity of intersection, we see the only possible accumulation points of a sequence in $\scrR^{H^\ell,J^\ell_k}_\bfJ(x)$ are of the form specified in the statement of this proposition.
\end{proof}

In order to describe the second compactification, we will actually have to digress to discuss an auxiliary compactification, denoted $\overline{\scrR}^{H^\ell,J^\ell_k}_{\bfJ,\mathrm{aux}}(x)$. Recall, given a sequence $[u_\nu]\in\scrR^{H^\ell,J^\ell_k}_\bfJ(x)$ which Gromov-converges to $[u]\in\overline{\scrR}^{H^\ell,J^\ell_k}_\bfJ(x)$, a part of the definition of Gromov-convergence is that the domains of $[u_\nu]$ Gromov-converge to the domain of $[u]$ in the usual Deligne-Mumford compactification of stable genus 0 curves with marked points. Moreover, Proposition \ref{prop:naivecompactification} shows that any sphere bubbles of $[u]$ are determined up to the $\bbC^*$-action. Instead of modeling the Gromov-convergence of the domains on the Gromov-convergence of the usual Deligne-Mumford compactification, we could use the real Deligne-Mumford compactification of stable genus 0 curves with marked points. (Recall, the real Deligne-Mumford compactification is a real-oriented blow-up of the usual Deligne-Mumford compactification along the strict normal crossings divisor corresponding to singular stable curves.) Then, any sphere bubbles in the corresponding Gromov-limit of $[u_\nu]$ are determined up to the $\bbR$-action.

\begin{prop}\label{prop:auxcompactification}
For $\overline{w}(x)<\lambda_\ell$ and $\epsilon_\ell$ sufficiently small, the moduli space $\scrR^{H^\ell,J^\ell_k}_\bfJ(x)$ admits a Gromov-compactification $\overline{\scrR}^{H^\ell,J^\ell_k}_{\bfJ,\mathrm{aux}}(x)$ by allowing 
\begin{enumerate}
\item breakings at non-divisorial 1-periodic orbits such that the corresponding cylinder components are completely contained in $X$
\item and bubbling, at any $z\in\{z_j:j\in\bfJ^c\}$, of an enhanced 1-pointed relative sphere representing a spherical homology class $A$ satisfying $A\cdot D=1$.
\end{enumerate}
See Figure \ref{fig:convergence}.
\end{prop}

\begin{proof}
The main point is to show how to define the parameterization of a sphere bubble in the Gromov-limit $[u]$ of a sequence $[u_\nu]\in\scrR^{H^\ell,J^\ell_k}_\bfJ(x)$ in order to actually define $\overline{\scrR}^{H^\ell,J^\ell_k}_{\bfJ,\mathrm{aux}}(x)$ (since we are only quotienting by the $\bbR$-action). Let $C$ be a stable genus 0 curve with marked points in the real Deligne-Mumford compactification with irreducible components $\{C_1,\ldots,C_s\}$. Recall, what differentiates $C$ as an element of the real Deligne-Mumford compactification, as opposed to an element of the usual Deligne-Mumford compactification, is the following: if $z_{ij}$ is a node connecting $C_i$ to $C_j$, then we decorate $z_{ij}$ with a choice of \emph{relative phase parameter} 
    \begin{equation}
    \theta\in\big(S_{z_{ij}}C_i\times S_{z_{ij}}C_j\big)/S^1,
    \end{equation}
where $S^1$ acts diagonally by rotation.

We have already fixed (if $1\in\bfJ^c$) an element in $S_{z_1}\Sigma_k$. We now fix an auxiliary choice of element $\theta_j\in S_{z_j}\Sigma_k$, $j\in\{1,\ldots,k\}$, which agrees with the choice at $1$ if $1\in\bfJ^c$; ultimately, the auxiliary choices for $j\in\{2,\ldots,k\}$ will not matter, cf. Remark \ref{rem:auxcompactification}. Now, consider a sequence $[u_\nu]\in\scrR^{H^\ell,J^\ell_k}_\bfJ(x)$ with Gromov-limit $[\widetilde{u}]\in\overline{\scrR}^{H^\ell,J^\ell_k}_\bfJ(x)$; we assume a sphere bubble $[v^j]\in\bbG\bbW$ appears at $z_j$. For each choice of $\phi\in S_0\bbC P^1$, we may choose a representative $v^j_\phi\in[v^j]$ such that 
    \begin{equation}
    [v^j_\phi]\neq[v^j_{\phi'}]\in\bbG\bbW^{S^1}.
    \end{equation}
In order for $[u_\nu]$ to converge to $[\widetilde{u}]$ in $\overline{\scrR}^{H^\ell,J^\ell_k}_\bfJ(x)$, we require that there exists a sequence of stabilizations $C_\nu$ of the domains of $[u_\nu]$ resp. a stabilization $\widetilde{C}$ of the domain of $[u]$ such that $C_\nu$ converges to $\widetilde{C}$ in the corresponding usual Deligne-Mumford compactification. Observe, there is an irreducible component $\widetilde{C}_j$ of $\widetilde{C}$ which is used as the domain of $[v^j]$. Now, instead consider $C_\nu$ as a sequence in the associated real Deligne-Mumford compactification via our choice of $\theta_j$'s; we have that $C_\nu$ converges to an element $C$. By our previous discussion, $C$ differs from $\widetilde{C}$ by a choice of relative phase parameters, i.e., under the projection from the real Deligne-Mumford compactification to the usual Deligne-Mumford compactification collapsing the $S^1$-fibers, we see $C$ maps to $\widetilde{C}$. In particular, there is an irreducible component $C_j$ of $C$ corresponding to $\widetilde{C}_j$. Let $z_j$ be the node connecting $C_j$ to $C$; we now use the relative phase parameter
    \begin{equation}
    \big[(\theta_j,\phi)\big]\in\big(S_{z_j}\Sigma_k\times S_{z_j}C_j\big)/S^1
    \end{equation}
to define the parameterization $[v^j_\phi]$ we will use to construct the Gromov-limit $[u]$ of $[u_\nu]$ in $\overline{\scrR}^{H^\ell,J^\ell_k}_{\bfJ,\mathrm{aux}}(x)$. In particular, we have now defined $\overline{\scrR}^{H^\ell,J^\ell_k}_{\bfJ,\mathrm{aux}}(x)$; moreover, it is clear that the only possible accumulation points of a sequence in $\overline{\scrR}^{H^\ell,J^\ell_k}_{\bfJ,\mathrm{aux}}(x)$ are of the form specified in the statement of this proposition (this is because the corresponding statements hold in Proposition \ref{prop:naivecompactification}).
\end{proof}

There is a map 
    \begin{equation}
    \varphi_\mathrm{aux}:\overline{\scrR}^{H^\ell,J^\ell_k}_{\bfJ,\mathrm{aux}}(x)\to\overline{\scrR}^{H^\ell,J^\ell_k}_\bfJ(x)
    \end{equation}
induced by the map 
    \begin{equation}
    \bbG\bbW^{S^1}\to\bbG\bbW.
    \end{equation}
The second compactification, which we call the \emph{enhanced compactification}, is the quotient 
    \begin{equation}
   \overline{\scrR}^{H^\ell,J^\ell_k}_{\bfJ,S^1}(x)\equiv\overline{\scrR}^{H^\ell,J^\ell_k}_{\bfJ,\mathrm{aux}}(x)/\sim,
    \end{equation}
where $[u]^{S^1,\mathrm{aux}}\sim[u']^{S^1,\mathrm{aux}}$ if 
\begin{enumerate}
\item $[u]^{S^1,\mathrm{aux}}$ and $[u']^{S^1,\mathrm{aux}}$ are contained in the same fiber of $\varphi_\mathrm{aux}$
\item and the sphere bubble at $z_1$ of $[u]^{S^1,\mathrm{aux}}$ equals the sphere bubble at $z_1$ of $[u']^{S^1,\mathrm{aux}}$ in $\bbG\bbW^{S^1}$ (if such a sphere bubble exists).
\end{enumerate}
See Figure \ref{fig:convergence}.

\begin{rem}\label{rem:auxcompactification}
It is the enhanced compactification that we are interested in. In particular, we may define an enhanced evaluation map on $\scrR^{H^\ell,J^\ell_k}_\bfJ(x)$ at any of the marked points. However, these enhanced evaluation maps will not extend to $\overline{\scrR}^{H^\ell,J^\ell_k}_\bfJ(x)$ because, here, sphere bubbles are only well-defined up to the entire $\bbC^*$-action. Meanwhile, these enhanced evaluation maps will extend to $\overline{\scrR}^{H^\ell,J^\ell_k}_{\bfJ,\mathrm{aux}}(x)$. We only require the enhanced evaluation at $z_1$, this is why we further quotient to $\overline{\scrR}^{H^\ell,J^\ell_k}_{\bfJ,S^1}(x)$ (also, this is why we mentioned in the proof of Proposition \ref{prop:auxcompactification} that our auxiliary choices would not matter). In particular, in $\overline{\scrR}^{H^\ell,J^\ell_k}_{\bfJ,S^1}(x)$, sphere bubbles at $z_1$ are well-defined up to the $\bbR$-action while sphere bubbles at other marked points are well-defined up to the $\bbC^*$-action; hence, the former is codimension 1 while the latter is codimension 2.
\end{rem}

\begin{figure}[ht]
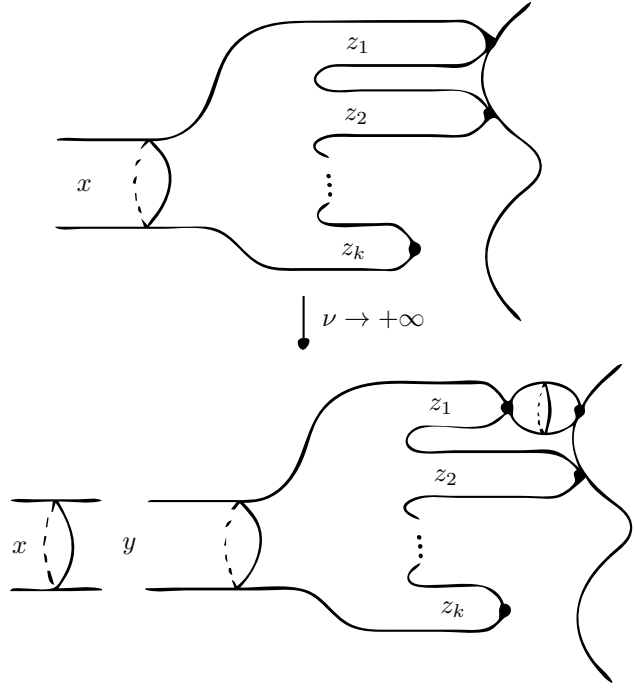
\caption{This picture is multi-purpose. In the context of the naive compactification, the sphere bubble at $z_1$ is well-defined up to the $\bbC^*$-action (as would be other sphere bubbles at other marked points). In the context of the auxiliary compactification, the sphere bubble at $z_1$ is well-defined up to the $\bbR$-action (as would be other sphere bubbles at other marked points). Finally, in the context of the enhanced compactification, the sphere bubble at $z_1$ is well-defined up to the $\bbR$-action (however, other sphere bubbles at other marked points are only well-defined up to the entire $\bbC^*$-action).}
\label{fig:convergence}
\end{figure}

We now move on to investigating the smooth structure of our various Gromov-compactifications of $\scrR^{H^\ell,J^\ell_k}_\bfJ(x)$.

\begin{prop}\label{prop:smoothstructure1}
$\overline{\scrR}^{H^\ell,J^\ell_k}_{\bfJ,\mathrm{aux}}(x)$ can be given the structure of a compact smooth manifold with corners whose codimension 1 boundary is enumerated by gluing maps of the form
    \begin{align}
    \overline{\scrR}^{H^\ell,J^\ell_k}_{\bfJ,\mathrm{aux}}(y)\times\bbF^{H^\ell,J^\ell}(y,x)&\to\overline{\scrR}^{H^\ell,J^\ell_k}_{\bfJ,\mathrm{aux}}(x), \\
    \overline{\scrR}^{H^\ell,J^\ell_k}_{\bfJ\cup\{j'\},\mathrm{aux}}(x)\times_X\bbG\bbW^{S^1}&\to\overline{\scrR}^{H^\ell,J^\ell_k}_{\bfJ,\mathrm{aux}}(x),
    \end{align}
where the fiber product is over $\eval_{z_{j'}}\times\eval_0$.
\end{prop}

\begin{proof}
This follows as in \cite[Section 6]{Lar21} (also, cf. \cite[Section 8]{PS24b} and \cite[Section 6]{PS25c}). First, when considering the gluing of a Floer trajectory to a marked PI thimble, since the gluing is occurring in the complement of the divisor, the gluing analysis is exactly the same as the gluing analysis of cylindrical ends in \emph{loc. cit.} Second, when considering the gluing of an enhanced 1-pointed sphere to a marked PI thimble, since the sphere is enhanced (i.e., we do not quotient by the $S^1$-action), the gluing analysis reduces to the gluing analysis of cylindrical ends in \emph{loc. cit.}
\end{proof}

For notational convenience, we define 
    \begin{align}
    \partial_{\bfJ,yx}&\equiv\overline{\scrR}^{H^\ell,J^\ell_k}_{\bfJ,\mathrm{aux}}(y)\times\bbF^{H^\ell,J^\ell}(y,x), \\
    \partial^{S^1}_{\bfJ\cup\{j'\},x}&\equiv\overline{\scrR}^{H^\ell,J^\ell_k}_{\bfJ\cup\{j'\},\mathrm{aux}}(x)\times_X\bbG\bbW^{S^1}.
    \end{align}
We have the following two basic relations. First, two short exact sequences of real vector bundles
    \begin{align}
    0\to\underline{\bbR}\to &T\overline{\scrR}^{H^\ell,J^\ell_k}_{\bfJ,\mathrm{aux}}(x)\vert_{\partial_{\bfJ,yx}}\to T\overline{\scrR}^{H^\ell,J^\ell_k}_{\bfJ,\mathrm{aux}}(y)\times T\bbF^{H^\ell,J^\ell}(y,x)\to0, \\
    0\to\underline{\bbR}\to &T\overline{\scrR}^{H^\ell,J^\ell_k}_{\bfJ,\mathrm{aux}}(x)\vert_{\partial^{S^1}_{\bfJ\cup\{j'\},x}}\to T\overline{\scrR}^{H^\ell,J^\ell_k}_{\bfJ\cup\{j'\},\mathrm{aux}}(x)\times_{TX}T\bbG\bbW^{S^1}\to 0
    \end{align}
given by taking a collar neighborhood of a codimension 1 boundary stratum. Second, two short exact sequences of real vector bundles 
    \begin{align}
    0\to\underline{\bbR}\to&\ind D\overline{\partial}_{H^\ell,J^\ell}\to T\bbF^{H^\ell,J^\ell}(y,x)\vert_{\operatorname{int}\bbF^{H^\ell,J^\ell}(y,x)}\to0, \\
    0\to\underline{\bbR}\to&\ind D^{\bbG\bbW}\to T\bbG\bbW^{S^1}\vert_{\operatorname{int}\bbG\bbW^{S^1}}\to0
    \end{align}
given by the translational direction. There are straightforward relations when passing to higher codimension boundary strata.

\begin{prop}
There is an extension of $\ind D^{\scrR,\bfJ}$ to $\overline{\scrR}^{H^\ell,J^\ell_k}_{\bfJ,\mathrm{aux}}(x)$ whose restriction to the interior of a codimension 1 boundary stratum is of the form
    \begin{align}
    \ind D^{\scrR,\bfJ}\vert_{\operatorname{int}\partial_{\bfJ,yx}}&=\ind D^{\scrR,\bfJ}\vert_{\operatorname{int}\partial_{\bfJ,yx}}+\ind\overline{\partial}_{H^\ell,J^\ell}\vert_{\operatorname{int}\partial_{\bfJ,yx}}, \\
    \ind D^{\scrR,\bfJ}\vert_{\operatorname{int}\partial^{S^1}_{\bfJ\cup\{j'\},x}}&=\ind D^{\scrR,\bfJ\cup\{j'\}}\vert_{\operatorname{int}\overline{\scrR}^{H^\ell,J^\ell_k}_{\bfJ\cup\{j'\},\mathrm{aux}}(x)}\times_{TX}\ind D^{\bbG\bbW}\vert_{\operatorname{int}\bbG\bbW^{S^1}}.
    \end{align}
The natural diagrams associated to codimension 1 boundary strata,
    \begin{equation}
    \begin{tikzcd}[column sep=small, center picture]
    & & \underline{\bbR}\arrow[d] \\
    & \ind D^{\scrR,\bfJ}\vert_{\operatorname{int}\partial_{\bfJ,yx}}\arrow[r,"\sim"]\arrow[d,equals] & T\overline{\scrR}^{H^\ell,J^\ell_k}_{\bfJ,\mathrm{aux}}(x)\vert_{\operatorname{int}\partial_{\bfJ,yx}}\arrow[d] \\
    \underline{\bbR}\arrow[r] & \ind D^{\scrR,\bfJ}\vert_{\operatorname{int}\partial_{\bfJ,yx}}+\ind\overline{\partial}_{H^\ell,J^\ell}\vert_{\operatorname{int}\partial_{\bfJ,yx}}\arrow[r] & T\overline{\scrR}^{H^\ell,J^\ell_k}_{\bfJ,\mathrm{aux}}(y)\vert_{\operatorname{int}\partial_{\bfJ,yx}}\times T\bbF^{H^\ell,J^\ell}(y,x)\vert_{\operatorname{int}\partial_{\bfJ,yx}}
    \end{tikzcd}
    \end{equation}
and 
    \begin{equation}
    \begin{tikzcd}[column sep=small, center picture]
    & & \underline{\bbR}\arrow[ddd] \\
    & & \\ 
    & & \\
    & \ind D^{\scrR,\bfJ}\vert_{\operatorname{int}\partial^{S^1}_{\bfJ\cup\{j'\},x}}\arrow[r,"\sim"]\arrow[d,equals] & T\overline{\scrR}^{H^\ell,J^\ell_k}_{\bfJ,\mathrm{aux}}(x)\vert_{\operatorname{int}\partial^{S^1}_{\bfJ\cup\{j'\},x}}\arrow[d] \\ 
    \underline{\bbR}\arrow[r] & \ind D^{\scrR,\bfJ\cup\{j'\}}\vert_{\operatorname{int}\overline{\scrR}^{H^\ell,J^\ell_k}_{\bfJ\cup\{j'\},\mathrm{aux}}(x)}\times_{TX}\ind D^{\bbG\bbW}\vert_{\operatorname{int}\bbG\bbW^{S^1}}\arrow[r] & T\overline{\scrR}^{H^\ell,J^\ell_k}_{\bfJ\cup\{j'\},\mathrm{aux}}(x)\vert_{\operatorname{int}\overline{\scrR}^{H^\ell,J^\ell_k}_{\bfJ\cup\{j'\},\mathrm{aux}}(x)}\times_{TX}T\bbG\bbW^{S^1}\vert_{\operatorname{int}\bbG\bbW^{S^1}},
    \end{tikzcd}
    \end{equation}
commute. Moreover, the natural diagrams associated to higher codimension boundary strata commute.
\end{prop}

\begin{proof}
Again, this follows as in \cite[Section 7]{Lar21} (also, cf. \cite[Section 8]{PS24b} and \cite[Section 6]{PS25c}) since the gluing analysis of Proposition \ref{prop:smoothstructure1} is identical.
\end{proof}

\begin{rem}
In words, the previous proposition is simply saying that the collar directions of a gluing map associated to a boundary stratum are identified with the translational directions in the direct sum of index bundles of that boundary stratum.
\end{rem}

\begin{cor}\label{cor:tsfaux}
Our constructed twisted stable framing on $\scrR^{H^\ell,J^\ell_k}_\bfJ(x)$ extends to a twisted stable framing on $\overline{\scrR}^{H^\ell,J^\ell_k}_{\bfJ,\mathrm{aux}}(x)$ which, when restricted to a codimension 1 boundary stratum, agrees with the already constructed twisted stable framing on that boundary stratum.
\end{cor}

\begin{proof}
This follows essentially because our twisted stable framings (in Subsections \ref{subsec:enhancedspheres} and \ref{subsec:pimarkedthimbles}) were constructed via gluing various Fredholm operators. First, we consider a codimension 1 boundary stratum of the form 
    \begin{equation}
    \overline{\scrR}^{H^\ell,J^\ell_k}_{\bfJ,\mathrm{aux}}(y)\times\bbF^{H^\ell,J^\ell}(y,x)\to\overline{\scrR}^{H^\ell,J^\ell_k}_{\bfJ,\mathrm{aux}}(x);
    \end{equation}
we appropriately modify the proof of Corollary \ref{corollary:extendingtsf}.

Second, we consider a codimension 1 boundary stratum of the form 
    \begin{equation}
    \overline{\scrR}^{H^\ell,J^\ell_k}_{\bfJ\cup\{j'\},\mathrm{aux}}(x)\times_X\bbG\bbW^{S^1}\to\overline{\scrR}^{H^\ell,J^\ell_k}_{\bfJ,\mathrm{aux}}(x).
    \end{equation}
Observe, we have a canonical isomorphism of real virtual bundles 
    \begin{equation}
    \ind D^{\scrR,\bfJ\cup\{j'\}}\times_{TX}\ind D^{\bbG\bbW}\cong\ind D^{\scrR,\bfJ\cup\{j'\}}+\ind D^{\bbG\bbW}-\Delta,
    \end{equation}
where $\Delta\subset(\eval_{z_j'}\times_X\eval_0)^*(TX\times TX)$ is the diagonal. Let 
    \begin{equation}
    \big(u_1,[u_2]\big)\in\operatorname{int}\big(\overline{\scrR}^{H^\ell,J^\ell_k}_{\bfJ\cup\{j'\},\mathrm{aux}}(x)\times_X\bbG\bbW^{S^1}\big).
    \end{equation}
The prescribed twisted stable framing at $\big(u_1,[u_2]\big)$ is achieved by considering the two glued together operators 
    \begin{align}
    T_{\frakF,x}\#\big(D^{\scrR,\bfJ\cup\{j'\}}_{u_1}\oplus\Phi^\scrR\oplus D_{\calL_m,\bfJ\cup\{j'\},u_1}\big)&\#\oplus_{j\in(\bfJ\cup\{j'\})^c}D^\Lambda_{p_j}, \\
    \big(D^{\bbG\bbW}_{u_2}\oplus\Phi^{\bbG\bbW}\oplus D^{\bbG\bbW,\calL_m}_{u_2}\big)&\#D^\Lambda_{\eneval_{+\infty}(u_2)},
    \end{align}
and then adding together their index bundles minus $\Delta$. Observe, the sum of the index bundles of the two aforementioned glued together operators minus $\Delta$ is canonically isomorphic to the index bundle of the glued together operator 
    \begin{multline}
    T_{\frakF,x}\#\big(D^{\scrR,\bfJ\cup\{j'\}}_{u_1}\oplus\Phi^\scrR\oplus D_{\calL_m,\bfJ\cup\{j'\},u_1}\big)\# \\
    \big(D^{\bbG\bbW}_{u_2}\oplus\Phi^{\bbG\bbW}\oplus D^{\bbG\bbW,\calL_m}_{u_2}\big)\#\oplus_{j\in \bfJ^c}D^\Lambda_{p_j},
    \end{multline}
where $D^\Lambda_{p_{j'}}\equiv D^\Lambda_{\eneval_{+\infty}(u_2)}$. The corollary follows by observing the glued together operator 
    \begin{equation}
    \big(D^{\scrR,\bfJ\cup\{j'\}}_{u_1}\oplus\Phi^\scrR\oplus D_{\calL_m,\bfJ\cup\{j'\},u_1}\big)\#\big(D^{\bbG\bbW}_{u_2}\oplus\Phi^{\bbG\bbW}\oplus D^{\bbG\bbW,\calL_m}_{u_2}\big)
    \end{equation}
may be canonically deformed to the operator
    \begin{equation}
    D^{\scrR,\bfJ}_{u_1\#u_2}\oplus\Phi^\scrR\oplus D_{\calL_m,\bfJ,u_1\#u_2},
    \end{equation}
where $u_1\#u_2$ is the gluing of $u_1$ and $u_2$.
\end{proof}

As before, there is a map 
    \begin{equation}
    \varphi_{S^1}:\overline{\scrR}^{H^\ell,J^\ell_k}_{\bfJ,\mathrm{aux}}(x)\to\overline{\scrR}^{H^\ell,J^\ell_k}_{\bfJ,S^1}(x)
    \end{equation}
induced by the map 
    \begin{equation}
    \bbG\bbW^{S^1}\to\bbG\bbW.
    \end{equation}
If we restrict $\varphi_{S^1}$ to the interior of a codimension $r$ boundary stratum consisting only of $r$ enhanced 1-pointed sphere bubbles, $\varphi_{S^1}$ exhibits the interior of that boundary stratum as a smooth $(S^1)^r$-bundle over its image.

\begin{prop}\label{prop:smoothstructure2}
$\overline{\scrR}^{H^\ell,J^\ell_k}_{\bfJ,S^1}(x)$ can be given the structure of a compact smooth manifold with corners whose codimension 1 boundary is enumerated by gluing maps of the form
    \begin{align}
    \overline{\scrR}^{H^\ell,J^\ell_k}_{\bfJ,S^1}(y)\times\bbF^{H^\ell,J^\ell}(y,x)&\to\overline{\scrR}^{H^\ell,J^\ell_k}_{\bfJ,S^1}(x), \\
    \overline{\scrR}^{H^\ell,J^\ell_k}_{\bfJ\cup\{1\},S^1}(x)\times_X\bbG\bbW^{S^1}&\to\overline{\scrR}^{H^\ell,J^\ell_k}_{\bfJ,S^1}(x)
    \end{align}
(where the second form of a codimension 1 boundary stratum only appears if $1\in\bfJ^c$).
\end{prop}

\begin{proof}
The proof follows from the classical fact that 
    \begin{equation}
    O(2)\to\operatorname{Diff}(S^1)
    \end{equation}
is a homotopy equivalence, i.e., the total space of an $S^1$-bundle with structure group $\operatorname{Diff}(S^1)$ can always be filled to be the total space of a $\bbD$-bundle with structure group $O(2)$. In particular, if we restrict $\varphi_{S^1}$ to the interior of a codimension 1 boundary stratum corresponding to bubbling an enhanced 1-pointed sphere bubble, then we may induce a smooth structure on its image by choosing a filling and collapsing the fibers. Moreover, we may choose these fillings compatibly on higher codimension boundary strata, hence the smooth manifold with corners structure. Note, we do not fill in $S^1$-bundles associated to bubbling an enhanced 1-pointed sphere bubble at $z_1$.
\end{proof}

\begin{cor}\label{cor:tsfs1}
The twisted stable framing on $\overline{\scrR}^{H^\ell,J^\ell_k}_{\bfJ,\mathrm{aux}}(x)$ descends to a twisted stable framing on $\overline{\scrR}^{H^\ell,J^\ell_k}_{\bfJ,S^1}(x)$.
\end{cor}

\begin{proof}
This is immediate from the construction since, if $[u],[u']\in\bbG\bbW^{S^1}$ are in the same fiber of the map $\bbG\bbW^{S^1}\to\bbG\bbW$, then we have a canonical isomorphism of real virtual vector spaces 
    \begin{equation}
    \ind D^\Lambda_{\eneval_{+\infty}(u)}\cong\ind D^\Lambda_{\eneval_{+\infty}(u')}
    \end{equation}
by Proposition \ref{prop:canonicaltransverseoperator}. I.e., the twisted stable framing of $\overline{\scrR}^{H^\ell,J^\ell_k}_{\bfJ,\mathrm{aux}}(x)$ extends over each $\bbD$-filling; the corollary follows.
\end{proof}

\subsection{Hybrid enhanced spheres}
For any $a\in\crit(f_S)$, we consider the moduli space $\bbG\bbW^{S^1}(a)$ of \emph{hybrid enhanced 1-pointed relative spheres}, i.e., the standard (partial) Gromov-bordification of 
    \begin{equation}
    \eneval_{+\infty}^{-1}\big(W^s(a;f_S)\big),\;\;\eneval_{+\infty}:\bbG\bbW^{S^1}\to S_DM
    \end{equation}
given by allowing breaking at critical points, see Figure \ref{fig:hybridenhancedsphere}.

\begin{figure}[ht]
%% Creator: Inkscape 1.4.2 (f4327f4, 2025-05-13), www.inkscape.org
%% Accompanies image file 'hybridenhancedsphere.pdf' (pdf, eps, ps)
\begingroup%
  \makeatletter%
  \providecommand\color[2][]{%
    \errmessage{(Inkscape) Color is used for the text in Inkscape, but the package 'color.sty' is not loaded}%
    \renewcommand\color[2][]{}%
  }%
  \providecommand\transparent[1]{%
    \errmessage{(Inkscape) Transparency is used (non-zero) for the text in Inkscape, but the package 'transparent.sty' is not loaded}%
    \renewcommand\transparent[1]{}%
  }%
  \providecommand\rotatebox[2]{#2}%
  \newcommand*\fsize{\dimexpr\f@size pt\relax}%
  \newcommand*\lineheight[1]{\fontsize{\fsize}{#1\fsize}\selectfont}%
  \ifx\svgwidth\undefined%
    \setlength{\unitlength}{102.82471064bp}%
    \ifx\svgscale\undefined%
      \relax%
    \else%
      \setlength{\unitlength}{\unitlength * \real{\svgscale}}%
    \fi%
  \else%
    \setlength{\unitlength}{\svgwidth}%
  \fi%
  \global\let\svgwidth\undefined%
  \global\let\svgscale\undefined%
  \makeatother%
  \begin{picture}(1,0.64910155)%
    \lineheight{1}%
    \setlength\tabcolsep{0pt}%
    \put(0,0){\includegraphics[width=\unitlength,page=1]{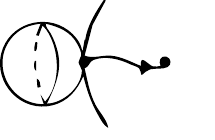}}%
    \put(0.54578686,0.01542843){\color[rgb]{0,0,0}\makebox(0,0)[lt]{\lineheight{1.25}\smash{\begin{tabular}[t]{l}$D$\end{tabular}}}}%
    \put(0.84221103,0.32588402){\color[rgb]{0,0,0}\makebox(0,0)[lt]{\lineheight{1.25}\smash{\begin{tabular}[t]{l}$a$\end{tabular}}}}%
  \end{picture}%
\endgroup%

\caption{This is a hybrid enhanced 1-pointed relative sphere.}
\label{fig:hybridenhancedsphere}
\end{figure}

\begin{rem}
Observe, we use the term \emph{partial Gromov-bordification} since we have only added some boundary to the moduli space; however, $\bbG\bbW^{S^1}(a)$ is not necessarily compact -- the evaluation map $\eval_0:\bbG\bbW^{S^1}(a)\to X$ is only proper.
\end{rem}

This is, for generic data, a smooth manifold with corners. By splitting the short exact sequence
    \begin{equation}
    0\to T\bbG\bbW^{S^1}(a)\to T\bbG\bbW^{S^1}\to T\overline{W}^u(a;f_S)\to0,
    \end{equation}
we may endow $\bbG\bbW^{S^1}(a)$ with a twisted stable framing
    \begin{equation}\label{eqn:auxredefining2}
    T\bbG\bbW^{S^1}(a)+\underline{\bbR}\cong\underline{\bbR}^{-I(a)}+\eneval_{+\infty}^*\widetilde{T}^{\bbG\bbW}+\eval_{+\infty}^*T^{\bbG\bbW}
    \end{equation}
which, when restricted to a codimension 1 boundary stratum, agrees with the already constructed twisted stable framing on that boundary stratum.

\subsection{Hybrid marked PI thimbles}
For any: $x\in\chi_k(X;H^\ell)$ with $k\geq1$, $a\in\crit(f_S)$, and $\bfJ$ with $1\in\bfJ^c$; we consider the moduli space $\overline{\scrR}^{H^\ell,J^\ell_k}_{\bfJ,S^1}(a,x)$ of \emph{hybrid $k$-marked $\bfJ$-PI $H^\ell$-thimbles}, i.e., the standard Gromov-compactification of 
    \begin{equation}
    \eneval_{z_1}^{-1}\big(W^s(a;f_S)\big),\;\;\eneval_{z_1}:\overline{\scrR}^{H^\ell,J^\ell_k}_{\bfJ,S^1}(x)\to S_DM
    \end{equation}
given by allowing breaking at critical points, see Figure \ref{fig:hybridpimarkedthimble}.

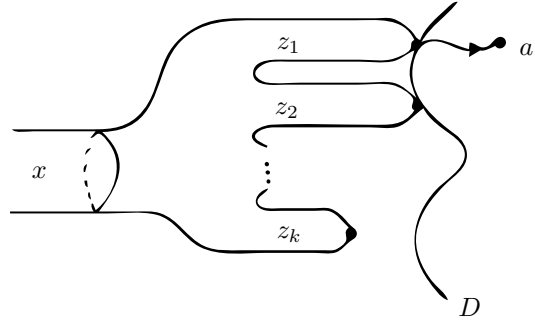
\begin{figure}[ht]
%% Creator: Inkscape 1.4.2 (f4327f4, 2025-05-13), www.inkscape.org
%% Accompanies image file 'hybridpimarkedthimble.pdf' (pdf, eps, ps)
\begingroup%
  \makeatletter%
  \providecommand\color[2][]{%
    \errmessage{(Inkscape) Color is used for the text in Inkscape, but the package 'color.sty' is not loaded}%
    \renewcommand\color[2][]{}%
  }%
  \providecommand\transparent[1]{%
    \errmessage{(Inkscape) Transparency is used (non-zero) for the text in Inkscape, but the package 'transparent.sty' is not loaded}%
    \renewcommand\transparent[1]{}%
  }%
  \providecommand\rotatebox[2]{#2}%
  \newcommand*\fsize{\dimexpr\f@size pt\relax}%
  \newcommand*\lineheight[1]{\fontsize{\fsize}{#1\fsize}\selectfont}%
  \ifx\svgwidth\undefined%
    \setlength{\unitlength}{215.50366019bp}%
    \ifx\svgscale\undefined%
      \relax%
    \else%
      \setlength{\unitlength}{\unitlength * \real{\svgscale}}%
    \fi%
  \else%
    \setlength{\unitlength}{\svgwidth}%
  \fi%
  \global\let\svgwidth\undefined%
  \global\let\svgscale\undefined%
  \makeatother%
  \begin{picture}(1,0.5607993)%
    \lineheight{1}%
    \setlength\tabcolsep{0pt}%
    \put(0,0){\includegraphics[width=\unitlength,page=1]{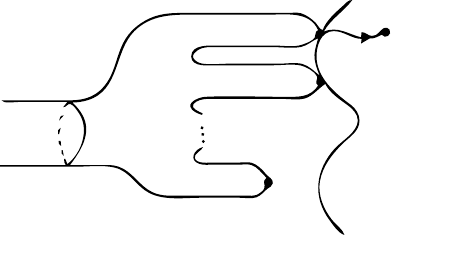}}%
    \put(0.46619488,0.47839802){\color[rgb]{0,0,0}\makebox(0,0)[lt]{\lineheight{1.25}\smash{\begin{tabular}[t]{l}$z_1$\end{tabular}}}}%
    \put(0.46276228,0.36425632){\color[rgb]{0,0,0}\makebox(0,0)[lt]{\lineheight{1.25}\smash{\begin{tabular}[t]{l}$z_2$\end{tabular}}}}%
    \put(0.46619488,0.14636735){\color[rgb]{0,0,0}\makebox(0,0)[lt]{\lineheight{1.25}\smash{\begin{tabular}[t]{l}$z_k$\end{tabular}}}}%
    \put(0.89081687,0.46724128){\color[rgb]{0,0,0}\makebox(0,0)[lt]{\lineheight{1.25}\smash{\begin{tabular}[t]{l}$a$\end{tabular}}}}%
    \put(0.7843994,0.01067585){\color[rgb]{0,0,0}\makebox(0,0)[lt]{\lineheight{1.25}\smash{\begin{tabular}[t]{l}$D$\end{tabular}}}}%
    \put(0.03776115,0.25097364){\color[rgb]{0,0,0}\makebox(0,0)[lt]{\lineheight{1.25}\smash{\begin{tabular}[t]{l}$x$\end{tabular}}}}%
  \end{picture}%
\endgroup%

\caption{This is a hybrid $k$-marked $\{1,\ldots,k-1\}$-PI thimble, $k\geq2$.}
\label{fig:hybridpimarkedthimble}
\end{figure}

This is, for generic data, a compact smooth manifold with corners whose codimension 1 boundary is enumerated by gluing maps of the form 
    \begin{align}
    \overline{\scrR}^{H^\ell,J^\ell_k}_{\bfJ,S^1}(a,y)\times\bbF^{H^\ell,J^\ell}(y,x)&\to\overline{\scrR}^{H^\ell,J^\ell_k}_{\bfJ,S^1}(a,x), \\
    \overline{\scrR}^{H^\ell,J^\ell_k}_{\bfJ\cup\{1\},S^1}(x)\times_X\bbG\bbW^{S^1}(a)&\to\overline{\scrR}^{H^\ell,J^\ell_k}_{\bfJ,S^1}(a,x), \\
    \bbD\bbM_{S_DM}(a,b)\times\overline{\scrR}^{H^\ell,J^\ell_k}_{\bfJ,S^1}(b,x)&\to\overline{\scrR}^{H^\ell,J^\ell_k}_{\bfJ,S^1}(a,x).
    \end{align}
By splitting the short exact sequence 
    \begin{equation}
    0\to T\overline{\scrR}^{H^\ell,J^\ell_k}_{\bfJ,S^1}(a,x)\to T\overline{\scrR}^{H^\ell,J^\ell_k}_{\bfJ,S^1}(x)\to T\overline{W}^u(a;f_S)\to0,
    \end{equation}
we may endow $\overline{\scrR}^{H^\ell,J^\ell_k}_{\bfJ,S^1}(a,x)$ with a twisted stable framing 
    \begin{equation}
    T\overline{\scrR}^{H^\ell,J^\ell_k}_{\bfJ,S^1}(a,x)+\ind T_{\frakF,x}\cong\underline{\bbR}^{-I(a)}+\eneval_{z_1}^*\widetilde{T}_{\abs{\bfJ^c}}+\eval_{z_1}^*T_{\abs{\bfJ^c}}
    \end{equation}
which, when restricted to a codimension 1 boundary stratum, agrees with the already constructed twisted stable framing on that boundary stratum.

\subsection{The Puppe connecting map}
As mentioned in Subsection \ref{subsec:outline}, we have the cofiber sequence
    \begin{equation}
    \frakF^\Lambda_{k-1}\to\frakF^\Lambda_k\to\frakF^\Lambda_{k,k-1},
    \end{equation}
and a sufficient condition for this cofiber sequence to split is that the associated Puppe connecting map,
    \begin{equation}
    \frakF^\Lambda_{k,k-1}\to\Sigma\frakF^\Lambda_{k-1},
    \end{equation}
is null-homotopic. In this subsection, we will prove Theorem \ref{thm:main} by producing a sufficient condition for all of the Puppe connecting maps associated to the weight filtration of $\frakF^\Lambda$ to be null-homotopic via the vanishing of a certain ``spectral Gromov-Witten obstruction''.

\subsubsection{A model for the cofiber}
To start, we will require the general flow-categorical model for the cofiber of a (framed) flow bimodule; therefore, let us recall a bit of the general theory. If $\bbB:\bbX\to\bbY$ is a flow bimodule from a flow category $\bbX$ to a flow category $\bbY$, the cofiber of $\bbB$ may be explicitly defined as the flow category $\bbC_\bbB$ with objects $\ob\bbC_{\bbB}\equiv\ob\bbX\amalg\ob\bbY$ and morphisms
    \begin{equation}
    \bbC_\bbB(x,y)\equiv\begin{cases}
    \bbX(x,y), & x,y\in\ob\bbX \\
    \bbY(x,y), & x,y\in\ob\bbY \\
    \bbB(x,y), & x\in\ob\bbX,y\in\ob\bbY \\
    \emptyset, & \mathrm{otherwise}
    \end{cases}
    .
    \end{equation}
The flow bimodule $\bbI:\bbY\to\bbC_\bbB$ representing the inclusion is defined as 
    \begin{equation}
    \bbI(x,y)\equiv\begin{cases}
    s\bbY(x,y), & x,y\in\ob\bbY \\
    \emptyset, & \mathrm{otherwise}
    \end{cases}
    ,
    \end{equation}
where $s\bbY(x,y)$ is the compact smooth manifold with corners used in the definition of the identity flow bimodule $s\bbY:\bbY\to\bbY$ (also known as the ``diagonal flow bimodule'', cf. \cite[Subsection 6.2]{AB24}). The flow bimodule $\bbP:\bbC_\bbB\to\Sigma\bbX$ representing the Puppe connecting map is defined as 
    \begin{equation}
    \bbP(x,y)\equiv\begin{cases}
    s\bbX(x,y), & x,y\in\ob\bbX \\
    \emptyset, & \mathrm{otherwise}
    \end{cases}
    .
    \end{equation}
The obvious modifications are made for framed flow categories, cf. \cite[Sections 7.4 \& 7.5]{AB24}.

\subsubsection{Redefining the spectral low-energy log PSS morphism}
We now return to the special case considered in the present article. The inclusion framed flow bimodules 
    \begin{equation}
    \bbI^\ell_k:\bbF^{H^\ell,J^\ell,\Lambda}_{k-1}\to\bbF^{H^\ell,J^\ell,\Lambda}_k
    \end{equation}
may be defined using constant Floer continuation data connecting $(H^\ell,J^\ell)$ to itself, i.e., 
    \begin{equation}
    \bbI^\ell_k(x,y)\equiv\begin{cases}
    \overline{\frakc}_{\ell,\ell}(x,y), & x,y\in\chi_{k-1}(X;H^\ell) \\
    \emptyset, & \mathrm{otherwise}
    \end{cases}
    .
    \end{equation}
In order to differentiate, we will write $\widetilde{\chi}_{\leq k-1}(X;H^\ell)$ for the subset 
    \begin{equation}
    \chi_{\leq k-1}(X;H^\ell)\subset\chi_{\leq k}(X;H^\ell).
    \end{equation}
The cofiber $\bbC_{\bbI^\ell_k}$ may now be identified as the framed flow category with objects $\ob\bbC_{\bbI^\ell_k}\equiv\chi_{\leq k-1}(X;H^\ell)\amalg\chi_{\leq k}(X;H^\ell)$ and morphisms 
    \begin{equation}
    \bbC_{\bbI^\ell_k}(x,y)\equiv\begin{cases}
    \bbF^{H^\ell,J^\ell}(x,y), & x,y\in\chi_{\leq k-1}(X;H^\ell) \\
    \bbF^{H^\ell,J^\ell}(x,y), & x,y\in\chi_{\leq k}(X;H^\ell) \\
    \overline{\frakc}_{\ell,\ell}(x,y), & x\in\chi_{\leq k-1}(X;H^\ell),y\in\widetilde{\chi}_{\leq k-1}(X;H^\ell) \\
    \emptyset, & \mathrm{otherwise}
    \end{cases}
    .
    \end{equation} 
Now, the Puppe connecting framed flow bimodule 
    \begin{equation}
    \bbP^\ell_k:\bbC_{\bbI^\ell_k}\to\Sigma\bbF^{H^\ell,J^\ell,\Lambda}_{k-1}
    \end{equation}
may be identified by using constant Floer continuation data connecting $(H^\ell,J^\ell)$ to itself, i.e.,
    \begin{equation}
    \bbP^\ell_k(x,y)\equiv\begin{cases}
    \Sigma^\flow\overline{\frakc}_{\ell,\ell}(x,y), & x,y\in\chi_{\leq k-1}(X;H^\ell) \\
    \emptyset, & \mathrm{otherwise}
    \end{cases}
    ,
    \end{equation}
where $\Sigma^\flow\overline{\frakc}_{\ell,\ell}(y,x)$ denotes the compact smooth manifold with corners $\overline{\frakc}_{\ell,\ell}(x,y)$ together with its standard stable framing \eqref{eqn:floercontinuationframing} stabilized by adding $\underline{\bbR}$ to both sides (i.e., we need to stabilize in order for the codomain of $\bbP^\ell_k$ to be $\Sigma\bbF^{H^\ell,J^\ell,\Lambda}_{k-1}$).

We may now ``redefine'' the low-energy log PSS morphism as a framed flow bimodule
    \begin{equation}
    \lelogpss^{k,\ell}:\bbD\bbM_{S_DM,V_{\overline{w}(k)}}\to\bbC_{\bbI^\ell_k},\;\;k\geq1
    \end{equation}
via
    \begin{equation}
    \lelogpss^{k,\ell}(a,x)\equiv\begin{cases}
    \overline{\scrR}^{H^\ell,J^\ell_{\overline{w}(k)}}(a,x), & x\in\chi_k(X;H^\ell) \\
    \overline{\scrR}^{H^\ell,J^\ell_{\overline{w}(k)}}_{\{1\},S^1}(x)\times_X\bbG\bbW^{S^1}(a), & \mathrm{otherwise}
    \end{cases}
    ,
    \end{equation}
where $a\in\crit(f_S)$. The following result is straightforward to see.

\begin{lem}
We have a commutative diagram
    \begin{equation}
    \begin{tikzcd}[row sep=tiny]
    & \bbC_{\bbI^\ell_k} \arrow[dd,"\sim"] \\
    \bbD\bbM_{S_DM,V_{\overline{w}(k)}}\arrow[ur,"\lelogpss^{k,\ell}"]\arrow[dr,"\lelogpss^{k,\ell}",swap] & \\
    & \bbF^{H^\ell,J^\ell,\Lambda}_{k,k-1}
    \end{tikzcd}
    ,
    \end{equation}
where the equivalence $\bbC_{\bbI^\ell_k}\xrightarrow{\sim}\bbF^{H^\ell,J^\ell,\Lambda}_{k,k-1}$ is the obvious one realizing $\bbF^{H^\ell,J^\ell,\Lambda}_{k,k-1}$ as the cofiber of $\bbI^\ell_k$.
\end{lem}

Finally, the composition
    \begin{equation}
    \bbD\bbM_{S_DM,V_{\overline{w}(k)}}\xrightarrow{\lelogpss^{k,\ell}}\bbC_{\bbI_k}\xrightarrow{\bbP^\ell_k}\Sigma\bbF^{H^\ell,J^\ell,\Lambda}_{k-1}
    \end{equation}
is seen to be the framed flow bimodule defined by
    \begin{equation}
    \Big(\bbP^\ell_k\circ\lelogpss^{k,\ell}\Big)(a,x)=\Sigma^\flow\Big(\overline{\scrR}^{H^\ell,J^\ell_k}_{\{1\},S^1}(x)\times_X\bbG\bbW^{S^1}(a)\Big),
    \end{equation}
where, again, $\Sigma^\flow\Big(\overline{\scrR}^{H^\ell,J^\ell_k}_{\{1\},S^1}(x)\times_X\bbG\bbW^{S^1}(a)\Big)$ denotes the compact smooth manifold with corners $\overline{\scrR}^{H^\ell,J^\ell_k}_{\{1\},S^1}(x)\times_X\bbG\bbW^{S^1}(a)$ together with its twisted stable framing given by taking the twisted stable framing on the fiber product, induced from \eqref{eqn:auxredefining1} and \eqref{eqn:auxredefining2}, stabilized by adding $\underline{\bbR}$ to both sides.

\subsubsection{Auxiliary hybrid marked PI thimbles}
For any $x\in\chi_k(X;H^\ell)$ with $k\geq1$ and $b\in\crit(f_X)$, we consider the moduli space $\overline{\scrR}^{H^\ell,J^\ell_{\overline{w}(x)}}_{\{1\},S^1}(b,x)$ of \emph{auxiliary hybrid $\{1\}$-PI $\overline{w}(x)$-marked $H^\ell$-thimbles}, i.e., the standard Gromov-compactification of
    \begin{equation}
    \eval_{z_1}^{-1}\big(W^s(b;f_X)\big),\;\;\eval{z_1}:\overline{\scrR}^{H^\ell,J^\ell_{\overline{w}(x)}}_{\{1\},S^1}(x)\to X
    \end{equation}
given by allowing breaking at critical points, see Figure \ref{fig:auxhybridpimarkedthimble}. 

\begin{figure}[ht]
%% Creator: Inkscape 1.4.2 (f4327f4, 2025-05-13), www.inkscape.org
%% Accompanies image file 'auxhybridpimarkedthimble.pdf' (pdf, eps, ps)
\begingroup%
  \makeatletter%
  \providecommand\color[2][]{%
    \errmessage{(Inkscape) Color is used for the text in Inkscape, but the package 'color.sty' is not loaded}%
    \renewcommand\color[2][]{}%
  }%
  \providecommand\transparent[1]{%
    \errmessage{(Inkscape) Transparency is used (non-zero) for the text in Inkscape, but the package 'transparent.sty' is not loaded}%
    \renewcommand\transparent[1]{}%
  }%
  \providecommand\rotatebox[2]{#2}%
  \newcommand*\fsize{\dimexpr\f@size pt\relax}%
  \newcommand*\lineheight[1]{\fontsize{\fsize}{#1\fsize}\selectfont}%
  \ifx\svgwidth\undefined%
    \setlength{\unitlength}{208.86145164bp}%
    \ifx\svgscale\undefined%
      \relax%
    \else%
      \setlength{\unitlength}{\unitlength * \real{\svgscale}}%
    \fi%
  \else%
    \setlength{\unitlength}{\svgwidth}%
  \fi%
  \global\let\svgwidth\undefined%
  \global\let\svgscale\undefined%
  \makeatother%
  \begin{picture}(1,0.65131381)%
    \lineheight{1}%
    \setlength\tabcolsep{0pt}%
    \put(0,0){\includegraphics[width=\unitlength,page=1]{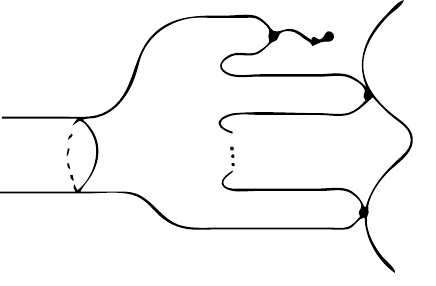}}%
    \put(0.0260355,0.27984415){\color[rgb]{0,0,0}\makebox(0,0)[lt]{\lineheight{1.25}\smash{\begin{tabular}[t]{l}$x$\end{tabular}}}}%
    \put(0.52998208,0.55234525){\color[rgb]{0,0,0}\makebox(0,0)[lt]{\lineheight{1.25}\smash{\begin{tabular}[t]{l}$z_1$\end{tabular}}}}%
    \put(0.54153297,0.41961711){\color[rgb]{0,0,0}\makebox(0,0)[lt]{\lineheight{1.25}\smash{\begin{tabular}[t]{l}$z_2$\end{tabular}}}}%
    \put(0.54153297,0.16052637){\color[rgb]{0,0,0}\makebox(0,0)[lt]{\lineheight{1.25}\smash{\begin{tabular}[t]{l}$z_{\overline{w}(k)}$\end{tabular}}}}%
    \put(0.78745005,0.54916248){\color[rgb]{0,0,0}\makebox(0,0)[lt]{\lineheight{1.25}\smash{\begin{tabular}[t]{l}$b$\end{tabular}}}}%
    \put(0.92358215,0.00682843){\color[rgb]{0,0,0}\makebox(0,0)[lt]{\lineheight{1.25}\smash{\begin{tabular}[t]{l}$D$\end{tabular}}}}%
  \end{picture}%
\endgroup%

\caption{This is an auxiliary hybrid marked $\{1\}$-PI thimble.}
\label{fig:auxhybridpimarkedthimble}
\end{figure}
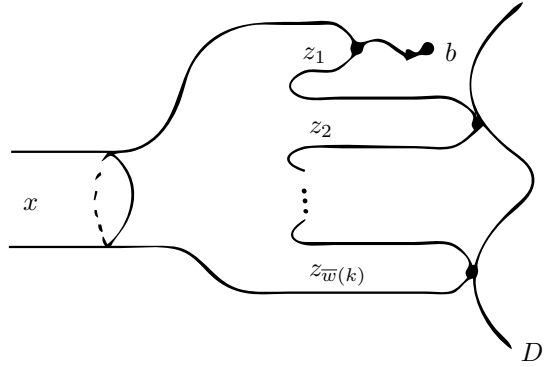

This is, for generic data, a compact smooth manifold with corners. By splitting the short exact sequence 
    \begin{equation}
    0\to T\overline{\scrR}^{H^\ell,J^\ell_{\overline{w}(x)}}_{\{1\},S^1}(b,x)\to T\overline{\scrR}^{H^\ell,J^\ell_{\overline{w}(x)}}_{\{1\},S^1}(x)\to T\overline{W}^u(b;f_X)\to0,
    \end{equation}
we may endow $\overline{\scrR}^{H^\ell,J^\ell_{\overline{w}(x)}}_{\{1\},S^1}(b,x)$ with a twisted stable framing 
    \begin{equation}
    T\overline{\scrR}^{H^\ell,J^\ell_{\overline{w}(x)}}_{\{1\},S^1}(b,x)+\ind T_{\frakF,x}\cong\underline{\bbR}^{-I(b)}+\eval_{z_1}^*\widetilde{T}_{\overline{w}(x)-1}+\eval_{z_1}^*T_{\overline{w}(x)-1}
    \end{equation}
which, when restricted to a codimension 1 boundary stratum, agrees with the already constructed twisted stable framing on that boundary stratum.  In particular, the moduli spaces of auxiliary hybrid $\{1\}$-PI $\overline{w}(x)$-marked $H^\ell$-thimbles assemble into a framed flow bimodule 
    \begin{equation}
    \bbB^{\ell,k}_2:\Sigma\bbD\bbM_{X,V_{\overline{w}(k)-1}}\to\Sigma\bbF^{H^\ell,J^\ell,\Lambda}_{k-1};
    \end{equation}
this is, after taking a colimit, one of the two desired maps for the diagram \eqref{eqn:outlinediagram}.
    
\subsubsection{Auxiliary hybrid enhanced spheres}\label{subsubsec:auxenhancedspheres}
For any $a\in\crit(f_S)$ and $b\in\crit(f_X)$, we consider the moduli space $\bbG\bbW^{S^1}(a,b)$ of \emph{auxiliary hybrid enhanced 1-pointed relative spheres}, i.e., the standard Gromov-compactification of 
    \begin{equation}
    \eval_0^{-1}\big(W^u(b;f_X)\big),\;\;\eval_0:\bbG\bbW^{S^1}(a)\to X
    \end{equation}
given by allowing breaking at critical points, see Figure \ref{fig:auxenhancedsphere}.

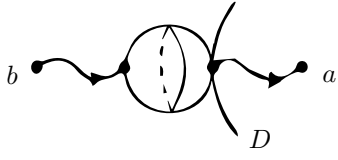
\begin{figure}[ht]
%% Creator: Inkscape 1.4.2 (f4327f4, 2025-05-13), www.inkscape.org
%% Accompanies image file 'auxenhancedsphere.pdf' (pdf, eps, ps)
\begingroup%
  \makeatletter%
  \providecommand\color[2][]{%
    \errmessage{(Inkscape) Color is used for the text in Inkscape, but the package 'color.sty' is not loaded}%
    \renewcommand\color[2][]{}%
  }%
  \providecommand\transparent[1]{%
    \errmessage{(Inkscape) Transparency is used (non-zero) for the text in Inkscape, but the package 'transparent.sty' is not loaded}%
    \renewcommand\transparent[1]{}%
  }%
  \providecommand\rotatebox[2]{#2}%
  \newcommand*\fsize{\dimexpr\f@size pt\relax}%
  \newcommand*\lineheight[1]{\fontsize{\fsize}{#1\fsize}\selectfont}%
  \ifx\svgwidth\undefined%
    \setlength{\unitlength}{134.57669115bp}%
    \ifx\svgscale\undefined%
      \relax%
    \else%
      \setlength{\unitlength}{\unitlength * \real{\svgscale}}%
    \fi%
  \else%
    \setlength{\unitlength}{\svgwidth}%
  \fi%
  \global\let\svgwidth\undefined%
  \global\let\svgscale\undefined%
  \makeatother%
  \begin{picture}(1,0.42183789)%
    \lineheight{1}%
    \setlength\tabcolsep{0pt}%
    \put(0,0){\includegraphics[width=\unitlength,page=1]{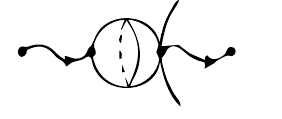}}%
    \put(-0.00506144,0.19136887){\color[rgb]{0,0,0}\makebox(0,0)[lt]{\lineheight{1.25}\smash{\begin{tabular}[t]{l}$b$\end{tabular}}}}%
    \put(0.87943967,0.19672921){\color[rgb]{0,0,0}\makebox(0,0)[lt]{\lineheight{1.25}\smash{\begin{tabular}[t]{l}$a$\end{tabular}}}}%
    \put(0.67305631,0.01178821){\color[rgb]{0,0,0}\makebox(0,0)[lt]{\lineheight{1.25}\smash{\begin{tabular}[t]{l}$D$\end{tabular}}}}%
  \end{picture}%
\endgroup%

\caption{This is an auxiliary enhanced sphere.}
\label{fig:auxenhancedsphere}
\end{figure}

This is, for generic data, a compact smooth manifold with corners. By splitting the short exact sequence 
    \begin{equation}
    0\to T\bbG\bbW^{S^1}(a,b)\to T\bbG\bbW^{S^1}(a)\to T\overline{W}^s(b;f_X)\to0,
    \end{equation}
we may endow $\bbG\bbW^{S^1}(a,b)$ with a twisted stable framing 
    \begin{equation}
    T\bbG\bbW^{S^1}(a,b)+\eval_0^*TX+\underline{\bbR}^{-I(b)}\cong\underline{\bbR}^{-I(a)-1}+\eneval_{+\infty}^*\widetilde{T}^{\bbG\bbW}+\eval_{+\infty}^*T^{\bbG\bbW}.
    \end{equation}
By rearranging and stabilizing, we may rewrite this as 
    \begin{multline}\label{eq:tsfaux1}
    T\bbG\bbW^{S^1}(a,b)+\underline{\bbR}^{-I(b)}+\eval_0^*V_{\overline{w}(k)-1}\cong \\
    \underline{\bbR}^{-I(a)-1}+\eneval_{+\infty}^*\widetilde{T}^{\bbG\bbW}+\eval_{+\infty}^*T^{\bbG\bbW}+\eval_0^*(V_{\overline{w}(k)-1}-TX).
    \end{multline}
Meanwhile, we may endow $\bbG\bbW^{S^1}(a,b)$ with a twisted stable framing using the gluing map $\overline{\scrR}^{H^\ell,J^\ell_{\overline{w}(x)}}_{\{1\},S^1}(x)\times_X\bbG\bbW^{S^1}(a)\to\overline{\scrR}^{H^\ell,J^\ell_{\overline{w}(x)}}(a,x)$:
    \begin{multline}\label{eq:tsfaux2}
    T\bbG\bbW^{S^1}(a,b)+\underline{\bbR}^{-I(b)}+\eval_{z_1,\{1\}}^*V_{\overline{w}(k)-1}\cong \\
    \underline{\bbR}^{-I(a)-1}+\eneval_{z_1,\emptyset}^*\widetilde{T}_{\overline{w}(k)}+\eval_{z_1,\emptyset}^*T_{\overline{w}(k)}
    \end{multline}
(here, we have notated the domain of the (enhanced) evaluation map in the subscript). The equivalence of the two twisted stable framings \eqref{eq:tsfaux1} and \eqref{eq:tsfaux2} follows by an operator gluing argument completely analogous to the one used in the proof of Corollary \ref{cor:tsfaux}. In particular, using \eqref{eq:tsfaux2}, the moduli spaces of auxiliary hybrid enhanced 1-pointed relative spheres assemble into a framed flow bimodule
    \begin{equation}
    \bbB^k_1:\bbD\bbM_{S_DM,V_{\overline{w}(k)}}\to\Sigma\bbD\bbM_{X,V_{\overline{w}(k)-1}};
    \end{equation}
this is one of the two desired maps for the diagram \eqref{eqn:outlinediagram}. Meanwhile, using \eqref{eq:tsfaux1}, the moduli spaces of auxiliary hybrid enhanced 1-pointed relative spheres assemble into an equivalent framed flow bimodule
    \begin{equation}\label{eqn:spectralgwobstruction}
    \calG\calW:\bbD\bbM_{S_DM,V^{\bbG\bbW}-TX}\to\Sigma\bbD\bbM_X.
    \end{equation}

\begin{rem}
Here, we have used the fact that we have an identification
    \begin{equation}
    V_k\cong V^{\bbG\bbW}+V_{k-1}-TX.
    \end{equation}
\end{rem}

\begin{defin}\label{defin:spectralgwobstruction}
We refer to the element 
    \begin{equation}
    \calG\calW\in\pi_0\Big((S_DM)^{-V^{\bbG\bbW}+TX}\wedge\Sigma\frakD X_+\Big)
    \end{equation}
determined by \eqref{eqn:spectralgwobstruction} as a \emph{spectral Gromov-Witten obstruction}.
\end{defin}

\subsubsection{Putting it all together, part 1}
Consider the composition
    \begin{equation}
    \bbB^{\ell,k}\equiv\bbB^{\ell,k}_2\circ\bbB^k_1:\bbD\bbM_{S_DM,V_{\overline{w}(k)}}\to\Sigma\bbF^{H^\ell,J^\ell,\Lambda}_{k-1};
    \end{equation}
this is a framed flow bimodule via defining
    \begin{equation}
    \bbB^{\ell,k}(a,x)\equiv\dfrac{\coprod_{b\in\crit(f_X)}\bbG\bbW^{S^1}(a,b)\times\overline{\scrR}^{H^\ell,J^\ell_{\overline{w}(x)}}_{\{1\},S^1}(b,x)}{\sim},
    \end{equation}
where the equivalence relation identifies the images
    \begin{equation}
    \begin{tikzcd}[column sep=-20ex]
    & \bbG\bbW^{S^1}(a,b)\times\bbD\bbM_X(b,b')\times\overline{\scrR}^{H^\ell,J^\ell_{\overline{w}(x)}}_{\{1\},S^1}(b',x)\arrow[dl]\arrow[dr] & \\
    \bbG\bbW^{S^1}(a,b')\times\overline{\scrR}^{H^\ell,J^\ell_{\overline{w}(x)}}_{\{1\},S^1}(b',x) & & \bbG\bbW^{S^1}(a,b)\times\overline{\scrR}^{H^\ell,J^\ell_{\overline{w}(x)}}_{\{1\},S^1}(b,x).
    \end{tikzcd}
    \end{equation}
Observe, implicit in this statement is that $\bbB^{\ell,k}(a,x)$ is a compact smooth manifold with corners with codimension 1 boundary strata enumerated by gluing maps of the form
    \begin{align}
    \bbD\bbM_{S_DM}(a,a')\times\bbB^{\ell,k}(a',x)&\to\bbB^{\ell,k}(a,x), \\
    \bbB^{\ell,k}(a,x')\times\bbF^{H^\ell,J^\ell}_{k-1}(x',x)&\to\bbB^{\ell,k}(a,x)
    \end{align}
which, moreover, can be given a twisted stable framing which, when restricted to a codimension 1 boundary stratum, agrees with the already constructed twisted stable framing on that boundary stratum, cf. \cite[Section 5]{AB24}.\footnote{Note, these claims about $\bbB^{\ell,k}(a,x)$ are what is needed to prove the inner-horn filling condition associated to a 2-simplex in $\flow^\fr$.} For completeness, we note
    \begin{equation}
    T\bbB^{\ell,k}(a,x)+\underline{\bbR}+\ind T_{\frakF,x}\cong\underline{\bbR}^{-I(a)}+\eneval_{+\infty}^*\widetilde{T}_{\overline{w}(k)}+\eval_{+\infty}^*T_{\overline{w}(k)}.
    \end{equation}

\begin{lem}\label{lem:hcomm}
$\bbP^\ell_k\circ\lelogpss^{k,\ell}$ is homotopic to $\bbB^{\ell,k}$ as framed flow bimodules.
\end{lem}

\begin{proof}
In order to prove the claim, we wish to build a compact smooth manifold with corners $\bbW(a,x)$ with codimension 1 boundary strata enumerated by gluing maps of the form 
    \begin{align}
    \Big(\bbP^\ell_k\circ\lelogpss^{k,\ell}\Big)(a,x)&\to\bbW(a,x), \\
    \bbB^{\ell,k}(a,x)&\to\bbW(a,x), \\
    \bbW(a,x')\times\bbF^{H^\ell,J^\ell}_{k-1}(x',x)&\to\bbW(a,x), \\
    \bbD\bbM_{S_DM}(a,a')\times\bbW(a',a)&\to\bbW(a,x),
    \end{align}
together with a twisted stable framing which, when restricted to a codimension 1 boundary stratum, agrees with the already constructed twisted stable framing on that boundary stratum; we proceed as follows.\footnote{I.e., we build a framed flow 2-simplex, cf. \cite[Definition 4.19]{AB24}.}

\begin{figure}[ht]
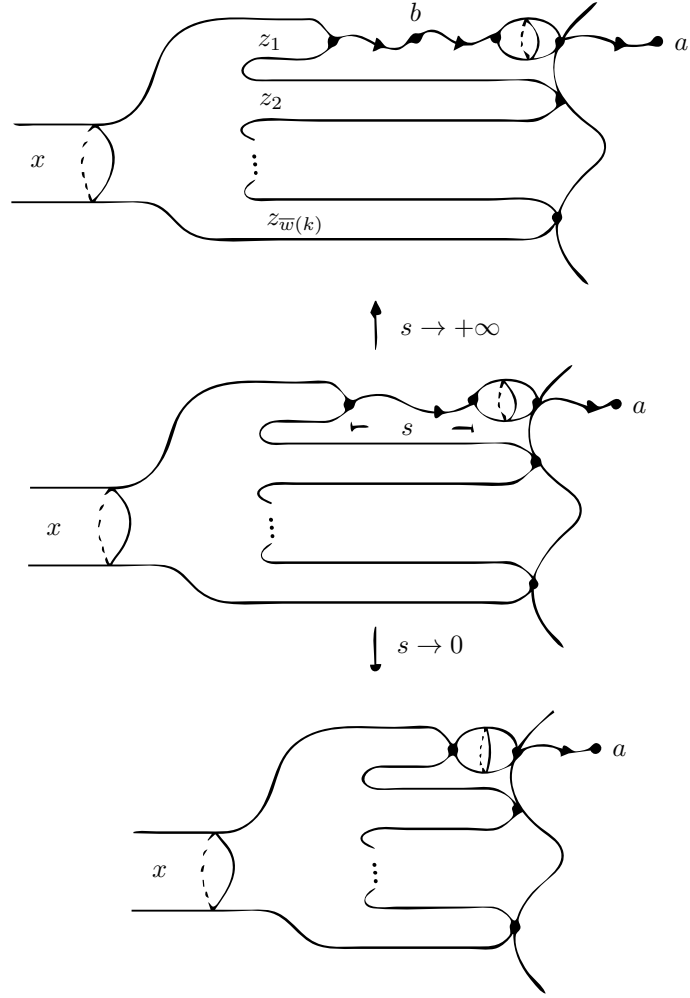
\caption{This is a schematic picture of the proof of Lemma \ref{lem:hcomm}.}
\label{fig:bordismargument}
\end{figure}

Let $\varphi_s$ denote the flow of $-\nabla f_X$. Consider the following two maps: 
    \begin{align}
    \eval_{z_1,s}:[0,+\infty]\times\overline{\scrR}^{H^\ell,J^\ell_{\overline{w}(x)}}_{\{1\},S^1}(x)&\to[0,+\infty]\times X \\
    \big(s,[u]\big)&\mapsto\big(s,(\varphi_s\circ\eval_{z_1})(u)\big) \nonumber
    \end{align}
and
    \begin{align}
    \eval_{0,s}:[0,+\infty]\times\bbG\bbW^{S^1}(a)&\to[0,+\infty]\times X \\
    \big(s,[u]\big)&\mapsto\big(s,(\varphi_{-s}\circ\eval_0)(u)\big). \nonumber
    \end{align}
We define
    \begin{equation}
    \bbW(a,x)\equiv\Big([0,+\infty]\times\overline{\scrR}^{H^\ell,J^\ell_{\overline{w}(x)}}_{\{1\},S^1}(x)\Big)\times_{[0,+\infty]\times X}\Big([0,+\infty]\times\bbG\bbW^{S^1}(a)\Big),
    \end{equation}
where the fiber product is over $\eval_{z_1,s}\times\eval_{0,s}$. This is, for generic data, a compact smooth manifold with corners with the desired codimension 1 boundary strata. (Note, the $s=0$ boundary corresponds to $\Big(\bbP^\ell_k\circ\lelogpss^{k,\ell}\Big)(a,x)$ while the $s=+\infty$ boundary corresponds to $\bbB^{\ell,k}(a,x)$.) See Figure \ref{fig:bordismargument}.

It remains to verify the claim about the twisted stable framing. But this follows by direct inspection after we observe that we have a canonical isomorphism of real virtual bundles 
    \begin{align}
    T\bbW(a,b)&\cong T\Big([0,+\infty]\times\overline{\scrR}^{H^\ell,J^\ell_{\overline{w}(x)}}_{\{1\},S^1}(x)\Big)\times_{T([0,+\infty]\times X)}T\Big([0,+\infty]\times\bbG\bbW^{S^1}(a)\Big) \\
    &\cong\underline{\bbR}+T\overline{\scrR}^{H^\ell,J^\ell_{\overline{w}(x)}}_{\{1\},S^1}(x)+\underline{\bbR}+T\bbG\bbW^{S^1}(a)-\Delta^s, \nonumber
    \end{align}
where 
    \begin{equation}
    \Delta^s\subset(\eval_{z_1,s}\times_{[0,+\infty]\times X}\eval_{0,s})^*\big(T([0,+\infty]\times X)\times T([0,+\infty]\times X)\big)
    \end{equation}
is the diagonal.
\end{proof}

Clearly, $\bbB^{\ell,k}$ is compatible with the Floer continuation framed flow bimodules. The upshot is that we have the desired commutative diagram \eqref{eqn:outlinediagram}:
    \begin{equation}
    \begin{tikzcd}
    \frakD(S_DM)^{-V_{\overline{w}(k)}}\arrow[r,"\bbB^k_1"]\arrow[d,"\lelogpss^k"] & \Sigma\frakD X^{-V_{\overline{w}(k)-1}}\arrow[d,"\bbB^k_2"] \\
    \frakF^\Lambda_{k,k-1}\arrow[r] & \Sigma\frakF^\Lambda_{k-1},
    \end{tikzcd}
    \end{equation}
where $\bbB^k_2\equiv\varinjlim_\ell\bbB^{\ell,k}_2$ and the bottom arrow is the Puppe connecting map. Since the left arrow is a homotopy equivalence, this completes the proof of Theorem \ref{thm:main}; we spell out its proof for completeness.

\begin{proof}[Proof of Theorem \ref{thm:main}]
Suppose 
    \begin{equation}
    \calG\calW\in\pi_0\Big((S_DM)^{-V^{\bbG\bbW}+TX}\wedge\Sigma\frakD X_+\Big)
    \end{equation}
is 0, then $\bbB^k_1$ is null-homotopic for all $k\geq1$ by Lemma \ref{lem:hcomm}. Since $\lelogpss^k$ is a homotopy equivalence (cf. Proposition \ref{prop:he}), this means each Puppe connecting map is null-homotopic, i.e., the weight filtration of $\frakF^\Lambda$ splits into its associated graded; this proves the theorem.
\end{proof}

\subsubsection{Putting it all together, part 2}\label{subsubsec:v2}
We are now in a position to prove Theorem \ref{thm:main2}. As already said, we have that $\calG\calW$ equivalently gives an element of 
    \begin{equation}
    \pi_0\Big((S_DM)^{-V_{\overline{w}(1)}}\wedge\Sigma\frakD X^{-V_0}\Big)=H_0\Big(\Sigma\frakD X^{-V_0};(S_DM)^{-V_{\overline{w}(1)}}\Big).
    \end{equation}
In particular, $\calG\calW$ may be viewed as the image of the coevaluation 
    \begin{equation}
    \bbS\to(S_DM)^{-V_{\overline{w}(1)}}\wedge\frakD(S_DM)^{-V_{\overline{w}(1)}}
    \end{equation}
under the map 
    \begin{equation}
    H_0\Big(\frakD(S_DM)^{-V_{\overline{w}(1)}};(S_DM)^{-V_{\overline{w}(1)}}\Big)\to H_0\Big(\Sigma\frakD X^{-V_0};(S_DM)^{-V_{\overline{w}(1)}}\Big)
    \end{equation}
induced by $\calG\calW$ itself.

\begin{proof}[Proof of Theorem \ref{thm:main2}]
Consider the Puppe connecting map at the lowest action-level:
    \begin{equation}
    \begin{tikzcd}
    \frakD(S_DM)^{-V_{\overline{w}(1)}}\arrow[r,"\calG\calW"]\arrow[d,"\sim"] & \Sigma\frakD X^{-V_0}\arrow[d,"\sim"] \\
    \frakF^\Lambda_{1,0}\arrow[r] & \Sigma\frakF^\Lambda_0\arrow[r] & \Sigma\frakF^\Lambda_1;
    \end{tikzcd}
    \end{equation}
of course, the bottom row of the previous diagram is a cofiber sequence. We now pass to homology with $\frakE\equiv(S_DM)^{-V_{\overline{w}(1)}}$-coefficients and obtain the following long exact sequence:
    \begin{equation}
    \cdots\to H_0\Big(\frakD(S_DM)^{-V_{\overline{w}(1)}};\frakE\Big)\to H_0\Big(\Sigma\frakD X^{-V_0};\frakE\Big)\to H_0\Big(\Sigma\frakF^\Lambda_1;\frakE\Big)\to\cdots.
    \end{equation}
As we have already noted, $\calG\calW$ is null-homotopic if and only if the image of the coevaluation 
    \begin{equation}
    \bbS\to(S_DM)^{-V_{\overline{w}(1)}}\wedge\frakD(S_DM)^{-V_{\overline{w}(1)}}
    \end{equation}
under the map 
    \begin{equation}
    H_0\Big(\frakD(S_DM)^{-V_{\overline{w}(1)}};\frakE\Big)\to H_0\Big(\Sigma\frakD X^{-V_0};\frakE\Big)
    \end{equation}
vanishes (since the image of the coevaluation is $\calG\calW$ itself). Independent of whether or not $\calG\calW$ vanishes, we see that, by exactness, the image of $\calG\calW$ under the map 
    \begin{equation}
    H_0\Big(\Sigma\frakD X^{-V_0};\frakE\Big)\to H_0\Big(\Sigma\frakF^\Lambda_1;\frakE\Big)
    \end{equation}
always vanishes. Therefore, the image of $\calG\calW$ under the composition 
    \begin{equation}
    H_0\Big(\Sigma\frakD X^{-V_0};\frakE\Big)\to H_0\Big(\Sigma\frakF^\Lambda_1;\frakE\Big)\to H_0\Big(\Sigma\frakF^\Lambda;\frakE\Big),
    \end{equation}
which is equal to the (suspension of the) PSS morphism 
    \begin{equation}
    H_0\Big(\Sigma\frakD X^{-V_0};\frakE\Big)\to H_0\Big(\Sigma\frakF^\Lambda;\frakE\Big),
    \end{equation}
always vanishes.

We now suppose the existence of exact Lagrangians $L_1,\ldots,L_\mu\subset X$ with the properties described in the statement of the present theorem. Now, $T^*L_\rho$ is a Liouville subdomain of $X$, therefore, we have a Viterbo restriction map 
    \begin{equation}
    \frakF^\Lambda\to\frakF^{\Lambda\vert_{T^*L_\rho}}\implies H_0\Big(\Sigma\frakF^\Lambda;\frakE\Big)\to H_0\Big(\Sigma\frakF^{\Lambda\vert_{T^*L_\rho}};\frakE\Big).
    \end{equation}
We now have the following commutative diagram: 
    \begin{equation}
    \begin{tikzcd}
    & H_0\Big(\frakD(S_DM)^{-V_{\overline{w}(1)}};\frakE\Big)\arrow[d] \\ 
    \bigoplus_{\rho=1}^\mu H_0\Big(\Sigma\frakD L^{-V_0}_\rho;\frakE\Big)\arrow[d,hook] & H_0\Big(\Sigma\frakD X^{-V_0};\frakE\Big)\arrow[d]\arrow[l,hook'] \\
    \bigoplus_{\rho=1}^\mu H_0\Big(\Sigma\frakF^{\Lambda\vert_{T^*L_\rho}};\frakE\Big) & H_0\Big(\Sigma\frakF^\Lambda;\frakE\Big)\arrow[l],
    \end{tikzcd}
    \end{equation}
where the bottom commuting square is simply the fact that the PSS morphism followed by Viterbo restriction commutes with pullback followed by the PSS morphism. Now, following $\calG\calW$ through the composition 
    \begin{equation}
    H_0\Big(\Sigma\frakD X^{-V_0};\frakE\Big)\to H_0\Big(\Sigma\frakF^\Lambda;\frakE\Big)\to H_0\Big(\Sigma\frakF^{\Lambda\vert_{T^*L_\rho}};\frakE\Big)
    \end{equation}
always yields 0. Meanwhile, if we assume $\calG\calW$ does not vanish, then following $\calG\calW$ through the composition 
    \begin{equation}
    H_0\Big(\Sigma\frakD X^{-V_0};\frakE\Big)\hookrightarrow\bigoplus_{\rho=1}^\mu H_0\Big(\Sigma\frakD L^{-V_0}_\rho;\frakE\Big)\hookrightarrow \bigoplus_{\rho=1}^\mu H_0\Big(\Sigma\frakF^{\Lambda\vert_{T^*L_\rho}};\frakE\Big)
    \end{equation}
yields an immediate contradiction; the theorem follows.
\end{proof}

\section{Computations}\label{sec:computations}
This section is dedicated to computations.

\subsection{Splittings}\label{subsec:splittings}
\subsubsection{Warm-ups}\label{part:warm-up}
We begin with a straightforward lemma.

\begin{lem}\label{lem:main}
Let $(M,D)$ satisfy parts (1) and (2) of Assumption \ref{assu:main}. Suppose, in addition, that $(M,D)$ satisfies:
\begin{enumerate}
\item there does not exist any (pseudo)-holomorphic sphere $u:\bbC P^1\to M$ satisfying 
    \begin{equation}
    u_*[\bbC P^1]\cdot D=1;
    \end{equation}
\item and $X$ admits a stable $\bbR$-polarization $\Lambda$: 
    \begin{equation}
    TX\oplus\underline{\bbC}^d\cong\Lambda\otimes_\bbR\underline{\bbC};
    \end{equation}
\end{enumerate}
then we have that
\begin{equation}
    \frakF^\Lambda\simeq\frakD X^{-U_0}\vee\bigvee_{k\geq1}\frakD(S_DM)^{-U_{\overline{w}(k)}},
    \end{equation}
where 
    \begin{equation}
    U_k\equiv\underline{\bbR}^{2dk}+TM^{\oplus k}+\Lambda^{\oplus 1-k}-(\ind D^\Lambda)^{\oplus k}-\underline{\bbR}^d-N_DM^{\oplus 2k}.
    \end{equation}
(Recall, $D^\Lambda$ is our family of canonical transverse operators, cf. Defnition \ref{defin:canonicaltransverseoperator})
    
\end{lem}

\begin{proof}
Since the moduli spaces $\scrR^{H^\ell,J^\ell_k}(x)$ do not have sphere bubbling in their Gromov-compactifications by assumption, we may directly build a framed flow bimodule 
    \begin{equation}
    \frakS^\ell:\bbD\bbM_{X,U_0}\amalg\coprod_{k\geq1}\bbD\bbM_{S_DM,U_{\overline{w}(k)}}\to\bbF^{H^\ell,J^\ell,\Lambda}
    \end{equation}
via defining $\frakS^\ell(a,x)$ as the standard Gromov-compactification of:
\begin{itemize}
\item when $a\in\ob\bbD\bbM_{X,U_0}$,
    \begin{equation}
    \eval_{+\infty}^{-1}\big(W^s(a;f_X)\big),\;\;\eval_{+\infty}:\overline{\scrR}^{H^\ell,J^\ell_0}(x)\to X;
    \end{equation}
\item and when $a\in\ob\bbD\bbM_{S_DM,U_{\overline{w}(k)}}$,
    \begin{equation}
    \eneval_{z_1}^{-1}\big(W^s(a;f_S)\big),\;\;\eneval_{z_1}:\overline{\scrR}^{H^\ell,J^\ell_k}(x)\to S_DM,\;\;k\geq1;
    \end{equation}
\end{itemize}
given by allowing breaking at critical points. The crux is noticing, in this case, we have the following twisted stable framings: 
\begin{itemize}
\item when $k=0$,
    \begin{equation}
    T\overline{\calR}^{H^\ell,J^\ell_k}(x)+\ind T_{\frakF,x}\cong\eval_{+\infty}^*\Lambda-\underline{\bbR}^d;
    \end{equation}
\item and when $k\geq1$,
    \begin{multline}
    T\overline{\calR}^{H^\ell,J^\ell_k}(x)+\ind T_{\frakF,x}\cong \\
    \eneval_{z_1}^*\big(\Lambda^{\oplus 1-k}-(\ind D^\Lambda)^{\oplus k}\big)+\eval_{z_1}^*\big(TM^{\oplus k}+\underline{\bbC}^{dk}-N_DM^{\oplus 2k}\big)-\underline{\bbR}^d.
    \end{multline}
\end{itemize}
Observe, the associated graded of the map $\frakS\equiv\varinjlim_\ell\frakS^\ell$ recovers the low-energy log PSS morphism. In particular, $\frakS$ is a map of filtered spectra, where each filtration is bounded below and exhaustive, which is a homotopy equivalence on each associated graded; thus, $\frakS$ is itself a homotopy equivalence.
\end{proof}

An immediate application of Lemma \ref{lem:main} is the following.

\begin{prop}\label{prop:warmup1}
Let $M$ be $\bbC P^{n+1}$ and $D$ a smooth hypersurface of degree at least 2. Suppose $X$ has a stable $\bbR$-polarization: $TX\oplus\underline{\bbC}^d\cong\Lambda\otimes_\bbR\underline{\bbC}$. We have that
    \begin{equation}
    \frakF^\Lambda\simeq\frakD X^{-U_0}\vee\bigvee_{k\geq1}\frakD(S_DM)^{-U_{\overline{w}(k)}},\;\;\overline{w}(k)=\deg(D)k,
    \end{equation}
where 
    \begin{equation}
    U_k=\underline{\bbR}^{2dk}+TM^{\oplus k}+\Lambda^{\oplus 1-k}-(\ind D^\Lambda)^{\oplus k}-\underline{\bbR}^d-N_DM^{\oplus 2k}.
    \end{equation}
\end{prop}

\begin{rem}\label{rem:cotangentbundleprojectivespace}
In the case that $D$ is degree 2, it is well known we may identify $X$ as $T^*\bbR P^n$ with its standard Weinstein structure. In particular, we may apply Proposition \ref{prop:warmup1} to obtain a splitting of $\frakF^\Lambda$. However, this splitting is not particularly surprising since Cohen resp. Bauer-Crabb-Spreafico have constructed a stable splitting of $\calL \bbR P^n$ via algebro-topological methods, cf. \cite{BCS01,Coh87}.
\end{rem}

\begin{rem}\label{rem:stablepolarization}
It is not always the case that, under the hypotheses of Proposition \ref{prop:warmup1}, $X$ admits a stable $\bbR$-polarization; this is a purely topological problem. For instance, their are obstructions -- every odd degree Chern class must be 2-torsion. Let us restrict to $n\geq3$. By using the Gysin sequence, we may compute the restriction 
    \begin{equation}
    \mathrm{res}:H^2(\bbC P^{n+1};\bbZ)\cong\bbZ\to H^2(X;\bbZ)\cong\bbZ/\deg(D)\bbZ
    \end{equation}
as the mod $\deg(D)$ reduction map. In particular, 
    \begin{equation}
    c_1(X)=(n+2)\mathrm{res}(H),
    \end{equation}
where $H$ is the hyperplane class. For example, if $n=3$ and $\deg(D)=3$, then 
    \begin{equation}
    c_1(X)=2\mathrm{res}(H)\mod 3;
    \end{equation}
clearly, this is not 2-torsion. For other examples, cf. \cite{AGLW25}. One thing to observe is that, for a Weinstein manifold, admitting a stable $\bbR$-polarization is equivalent to admitting a \emph{positive arboreal skeleton}, cf. \cite{AGEN22}.
\end{rem}

Another immediate application of Lemma \ref{lem:main} is the following.

\begin{prop}\label{prop:warmup2}
Let $M$ be a smooth hypersurface in $\bbC P^{n+1}$ of degree at least 2 and $D$ the intersection of $M$ and a smooth hypersurface of degree at least 2. Suppose $X$ has a stable $\bbR$-polarization: $TX\oplus\underline{\bbC}^d\cong\Lambda\otimes_\bbR\underline{\bbC}$. We have that
    \begin{equation}
    \frakF^\Lambda\simeq\frakD X^{-U_0}\vee\bigvee_{k\geq1}\frakD(S_DM)^{-U_{\overline{w}(k)}},\;\;\overline{w}(k)=\deg(D)k,
    \end{equation}
where 
    \begin{equation}
    U_k=\underline{\bbR}^{2dk}+TM^{\oplus k}+\Lambda^{\oplus 1-k}-(\ind D^\Lambda)^{\oplus k}-\underline{\bbR}^d-N_DM^{\oplus 2k}.
    \end{equation}
\end{prop}

\subsubsection{Cotangent bundles of spheres}
In this part, we will prove the following.

\begin{prop}\label{prop:spheresplitting} 
Let $M$ be the standard projective quadric hypersurface $\{z_0^2=z_1^2+\cdots+z_{n+1}^2\}$ in $\bbC P^{n+1}$ and $D$ the hyperplane section $M\cap\{z_0=0\}$. In this case, $X$ is Weinstein equivalent to $T^*S^n$. Let $\Lambda$ be induced by 
    \begin{equation}
    TM\oplus\underline{\bbC}\oplus\scrO_{\bbC P^{n+1}}(2)\vert_M\cong\scrO_{\bbC P^{n+1}}(1)\vert_M^{\oplus n+2}.
    \end{equation}
We have that
    \begin{equation}
    \frakF^\Lambda\simeq\frakD(S^n)^{-V_0}\vee\bigvee_{k\geq1}\frakD(S_DM)^{-V_k},
    \end{equation}
where 
    \begin{equation}
    V_k=\underline{\bbR}^n+TM^{\oplus k}-N_DM^{\oplus k}.
    \end{equation}
\end{prop}

\begin{rem}
Again, this splitting is not particularly surprising since Cohen has constructed a stable splitting of $\calL S^n$ via algebro-topological methods, cf. \cite{Coh87}. In the following part, we will reconcile Cohen's stable splitting of $\calL S^2$ with our splitting of $\frakF^\Lambda$.
\end{rem}

Of course, we proceed as follows; we are trying to show the element
    \begin{equation}
    \calG\calW\in\pi_0\Big((S_DM)^{-V^{\bbG\bbW}+T(T^*S^n)}\wedge\Sigma\frakD(T^*S^n)_+\Big),\;\;V^{\bbG\bbW}=\underline{\bbR}^{2n}+TM-N_DM
    \end{equation}
vanishes. I.e., we are trying to show the map 
    \begin{equation}
    \calG\calW:\bbD\bbM_{S_DM,V^{\bbG\bbW}-TX}\to\bbD\bbM_{T^*S^n,-\underline{\bbR}}
    \end{equation}
is null-homotopic. Since $T^*S^n$ is homotopy equivalent to $S^n$, we may choose $f_X$ such that 
    \begin{equation}
    \crit(f_X)=\{b_\mathrm{min},b\},
    \end{equation}
where $b_\mathrm{min}$ is the minimum and $I(b)=n$; in fact, we can obtain $f_X$ by pulling back the standard height function on $S^n$ and adding a function that goes to $+\infty$ in the fiber direction. Hence, we see that $\bbD\bbM_{T^*S^n,-\underline{\bbR}}$ splits into the disjoint union of two framed flow categories, each one consisting of a single object and no morphisms: 
    \begin{equation}
    \bbD\bbM_{T^*S^n,-\underline{\bbR}}\cong\Sigma^{-n-1}\unit\amalg\Sigma^{-1}\unit.
    \end{equation}
Thus, it suffices to show that $\calG\calW$ becomes null-homotopic after projecting to each factor, i.e., we have two cases.

For the first case, we will rely on the following lemma. 

\begin{lem}
Suppose $\bbB:\bbX\to\bbY$ is a framed flow bimodule such that each $\bbB(x,y)$ admits a fixed point free involution
    \begin{equation}
    T_{xy}:\bbB(x,y)\to\bbB(x,y)
    \end{equation}
compatible with (1) gluing in the obvious way and (2) the equipped (twisted) stable framing, then $\bbB$ is null-homotopic as a map of framed flow categories.
\end{lem}

\begin{proof}
We have that the mapping cylinder $\bbM(x,y)$ of the projection $\bbB(x,y)\to\bbB(x,y)/T_{xy}$ is a smooth manifold with corners equipped with a (twisted) stable framing whose codimension 1 boundary strata are enumerated by gluing maps of the form 
    \begin{align}
    \bbX(x,z)\times\bbM(z,y)&\to\bbM(x,y) \\
    \bbM(x,z)\times\bbY(z,y)&\to\bbM(x,y) \\
    \bbB(x,y)&\to\bbM(x,y)
    \end{align}
such that the equipped (twisted) stable framing restricts to the already defined (twisted) stable framing on that boundary stratum. This is precisely saying that $\bbB$ is null-homotopic as a map of framed flow categories.
\end{proof}

We consider the projection onto the first factor,
    \begin{equation}
    \calG\calW_1:\bbD\bbM_{S_DM,V^{\bbG\bbW}-TX}\to\Sigma^{-n-1}\unit,
    \end{equation}
defined by
    \begin{equation}
    \calG\calW_1(a,*)\equiv\bbG\bbW^{S^1}(a,b).
    \end{equation}
Given any $[u]\in\bbG\bbW^{S^1}$, we have that $u$ is of the form
    \begin{align}
    u:\bbC P^1&\to\bbC P^{n+1} \\
    [x,y]&\mapsto\big[p_0(x,y):\cdots:p_{n+1}(x,y)\big], \nonumber
    \end{align}
where 
\begin{itemize}
\item $p_j(x,y)$ is a homogeneous polynomial of degree 1, 
\item $p_0(0,1)=0$, after identifying $+\infty$ with $[0:1]$,
\item and 
    \begin{equation}
    p_0(x,y)^2=p_1(x,y)^2+\cdots+p_{n+1}(x,y)^2.
    \end{equation}
\end{itemize}
Let $\sigma_{\bbC P^{n+1}}:\bbC P^{n+1}\to\bbC P^{n+1}$ denote the map 
    \begin{equation}
    [z_0:z_1:\cdots:z_{n+1}]\mapsto[-z_0:z_1:\cdots:z_{n+1}];
    \end{equation}
let $\sigma_{\bbC P^1}:\bbC P^1\to\bbC P^1$ denote the map 
    \begin{equation}
    [x:y]\mapsto[-x:y].
    \end{equation}
We denote by $\sigma(u)$ the map
    \begin{equation}
    \sigma_{\bbC P^{n+1}}\circ u\circ\sigma_{\bbC P^1}:\bbC P^1\to M.
    \end{equation}
A local computation in coordinates shows that 
    \begin{equation}
    \eneval_{+\infty}(u)=\eneval_{+\infty}\big(\sigma(u)\big).
    \end{equation}
In particular, we have constructed a fixed point free involution 
    \begin{equation}
    T_{a*}:\calG\calW_1(a,*)\to\calG\calW_1(a,*)
    \end{equation}
which is compatible with gluing; it remains to show $T_{a*}$ is compatible with our constructed twisted stable framing. But this is straightforward. Our constructed twisted stable framing at $u$ is defined by considering the index bundle of the glued together operator 
    \begin{equation}
    \big(D^{\bbG\bbW}_u\oplus\Phi^{\bbG\bbW}\oplus D^{\bbG\bbW,\scrO_{\bbC P^{n+1}}(2)\vert_M}_u\big)\#D^\Lambda_{\eneval_{+\infty}(u)},
    \end{equation}
where we recall the definition of $D^{\bbG\bbW,\scrO_{\bbC P^{n+1}}(2)\vert_M}_u$ given in Subsection \ref{subsec:enhancedspheres}. Meanwhile, the twisted stable framing at $u$ induced by $\sigma$ is defined by first using the isomorphism 
    \begin{equation}
    \ind D^{\bbG\bbW}_u\cong\ind D^{\bbG\bbW}_{\sigma(u)},
    \end{equation}
induced by $\sigma$, and then considering the index bundle of the glued together operator
    \begin{equation}
    \big(D^{\bbG\bbW}_{\sigma(u)}\oplus\Phi^{\bbG\bbW}\oplus D^{\bbG\bbW,\scrO_{\bbC P^{n+1}}(2)\vert_M}_{\sigma(u)}\big)\#D^\Lambda_{\eneval_{+\infty}({\sigma(u)})}
    \end{equation}
and pulling back via $\sigma$. We simply compare operator by operator. First, we have already established 
    \begin{equation}
    D^\Lambda_{\eneval_{+\infty}(u)}=D^\Lambda_{\eneval_{+\infty}({\sigma(u)})}.
    \end{equation}
Second, we have that 
    \begin{equation}
    u^*\scrO_{\bbC P^{n+1}}(2)\vert_M\cong\scrO_{\bbC P^1}(2)\cong\sigma(u)^*\scrO_{\bbC P^{n+1}}(2)\vert_M,
    \end{equation}
hence, we may canonically deform $D^{\bbG\bbW,\scrO_{\bbC P^{n+1}}(2)\vert_M}_u$ to $D^{\bbG\bbW,\scrO_{\bbC P^{n+1}}(2)\vert_M}_{\sigma(u)}$. It follows that $T_{a*}$ is compatible with our constructed twisted stable framing; thus, $\calG\calW_1$ is null-homotopic.

For the second case, we will rely on the following lemma.

\begin{lem}\label{lem:unorientedbordism}
Let $L\to M$ be an ample complex line bundle with associated circle bundle $C\to M$. Suppose $\sigma:M\to L$ is a section which is transverse to 0, then 
    \begin{equation}
    \big[C\vert_Z\to C\big]\in\Omega^\mathrm{MO}_*(C)
    \end{equation}
vanishes, where $Z\equiv\sigma^{-1}(0)$ and $\Omega^\mathrm{MO}_*(\cdot)$ denotes unoriented bordism.
\end{lem}

\begin{proof}
Consider the section 
    \begin{equation}
    \dfrac{\sigma}{\abs{\abs{\sigma}}}:D-Z\to C;
    \end{equation}
this extends to a manifold with boundary $C\vert_Z$.
\end{proof}

We consider the projection onto the second factor,
    \begin{equation}
    \calG\calW_1:\bbD\bbM_{S_DM,V^{\bbG\bbW}-TX}\to\Sigma^{-1}\unit,
    \end{equation}
defined by
    \begin{equation}
    \calG\calW_1(a,*)\equiv\bbG\bbW^{S^1}(a,b_\mathrm{min}).
    \end{equation}
The argument is a bit long-winded, and so we ask the reader to bear with us; we will essentially construct our various twisted stably framed null-bordisms of $\bbG\bbW^{S^1}(a,b_\mathrm{min})$ by exhibiting the bounding twisted stably framed manifolds as submanifolds of the moduli space of enhanced 1-pointed spheres in the ambient $\bbC P^{n+1}$ relative to $H\equiv\{z_0=0\}$.

\begin{claim}
There exists a section 
    \begin{equation}
    \sigma:D\to\scrO_{\bbC P^{n+1}}(1)\vert_D,
    \end{equation}
which is transverse to 0, such that 
    \begin{equation}
    \eneval_{+\infty}:\bbG\bbW^{S^1}(b_\mathrm{min})\to S_DM\vert_{D'}
    \end{equation}
where $D'\equiv\sigma^{-1}(0)$, is a diffeomorphism.
\end{claim}

\begin{proof}[Proof of claim]
We have that 
    \begin{equation}
    X=\big\{1=z_1^2+\cdots+z_{n+1}^2\big\}\subset\bbC^{n+1}.
    \end{equation}
For simplicity, we assume $b_\mathrm{min}=(1,0,\ldots,0)$. We have that the lines on $X$ through $b_\mathrm{min}$ are of the form 
    \begin{equation}
    t\mapsto (1,0,\ldots,0)+t(0,v_2,\ldots,v_{n+1}),\;\;v_2^2+\cdots+v_{n+1}^2=0,
    \end{equation}
where $t\in\bbC P^1$. Moreover, the point at $+\infty$ is $[0:0:v_2:\cdots:v_{n+1}]$, i.e., $\bbG\bbW^{S^1}(b_\mathrm{min})/S^1$ is diffeomorphic to 
    \begin{equation}
    D'\equiv\big\{[0:0:v_2:\cdots:v_{n+1}]:v_2^2+\cdots+v_{n+1}^2=0\big\}\subset D
    \end{equation}
via $\eval_0$. Observe, $D'$ is a quadric of complex dimension $n-2$ contained in $D$, hence, transversely cut out by a section 
    \begin{equation}
    \sigma:D\to\scrO_{\bbC P^{n+1}}(1)\vert_D;
    \end{equation}
the claim follows.
\end{proof}

By the previous claim, we may define a smooth family $\{\ell^{D'}_z\}_{z\in D'}$  of lines in $M$ parameterized by $D'$; namely, let $\ell_z$ be the unique line in $M$ containing $b_\mathrm{min}$ and $z$. Observe, no line in this family is contained in $H$ by construction. Meanwhile, we have the following identifications: 
    \begin{equation}
    N_DM=\scrO_{\bbC P^{n+1}}(1)\vert_D=N_H\bbC P^{n+1}\vert_D;
    \end{equation}
in particular, we have the identification 
    \begin{equation}
    S_DM=S_H\bbC P^{n+1}\vert_D.
    \end{equation}
The previous claim shows we have defined a non-zero section
    \begin{equation}
    \sigma:D-D'\to N_DM\vert_{D-D'}=N_H\bbC P^{n+1}\vert_{D-D'}.
    \end{equation}
Hence, we may define a smooth family $\{\ell^{D-D'}_z\}_{z\in D-D'}$ of lines in $\bbC P^{n+1}$ parameterized by $D-D'$; namely, let $\ell_z$ be the unique line in $\bbC P^{n+1}$ containing $\sigma(z)/\abs{\abs{\sigma(z)}}$ and $z$. Observe, no line in this family is contained in $H$ by construction. The following is the first crucial claim.

\begin{claim}
Let $U\subset D$ be a tubular neighborhood of $D'$ and denote by $\{\ell^{D-U}_z\}_{z\in D-U}$ the smooth family given by restricting $\{\ell^{D-D'}_z\}_{z\in D-D'}$ to $D-U$. We may define a smooth family $\{\ell_z\}_{z\in D}$ of lines in $\bbC P^{n+1}$ parameterized by $D$ such that 
    \begin{equation}
    \ell_z=
    \begin{cases}
    \ell^{D'}_z, & z\in D' \\
    \ell^{D-U}_z, & z\in D-U
    \end{cases}
    .
    \end{equation}
Moreover, $\ell_z\not\subset H$ for all $z\in D$.
\end{claim}

\begin{proof}[Proof of claim]
Given any $z\in D$, we have that the space of lines in $\bbC P^{n+1}$ through $z$ is given by 
    \begin{equation}
    \bbP(T_z\bbC P^{n+1})\cong\bbC P^n;
    \end{equation}
moreover, we have that the space of lines in $H$ through $z$ is given by 
    \begin{equation}
    \bbP\big(T_zH\big)\cong\bbC P^{n-1}.
    \end{equation}
Thus, the space of lines in $\bbC P^{n+1}$ through $z$ which are not contained in $H$ is contractible: 
    \begin{equation}
    \bbC P^n-\bbC P^{n-1}\cong\bbC^n.
    \end{equation}
Now, finding a smooth family $\{\ell_z\}_{z\in D}$ of lines in $\bbC P^{n+1}$ parameterized by $D$ of the form specified in the statement of this claim is straightforward since the problem is equivalent to finding a smooth global section of a fiber bundle over $D$, with contractible fibers, which agrees with already defined sections over $D'$ and $D-U$ (i.e., we may find such a global section).
\end{proof}

Let $\bbG\bbW^{S^1}_{n+1,n}$ be the moduli space of enhanced 1-pointed spheres for the pair $(\bbC P^{n+1},H)$, i.e., maps $u:\bbC P^1\to\bbC P^{n+1}$ satisfying 
    \begin{equation}
    \begin{cases}
    \overline{\partial}u=0 & \\
    u_*[\bbC P^1]=1 & \\
    u^{-1}(H)=\{+\infty\} &
    \end{cases}
    \end{equation}
modulo $\bbR$-translations. Also, let $\bbG\bbW_{n+1,n}\equiv\bbG\bbW^{S^1}_{n+1,n}/S^1$. We may define a map 
    \begin{equation}
    \varphi_{D'}:D'\to\bbG\bbW_{n+1,n}
    \end{equation}
by sending $z\in D'$ to the unique element $\varphi_{D'}(z)\in\bbG\bbW_{n+1,n}$ satisfying (1) its image is $\ell_z$ and (2)
    \begin{equation}
    \begin{cases}
    \eval_0\big(\varphi_{D'}(z)\big)=b_\mathrm{min} \\
    \eval_{+\infty}\big(\varphi_{D'}(z)\big)=z
    \end{cases}
    .
    \end{equation}
Analogously, we may define a map 
    \begin{equation}
    \varphi_{D-U}:D-U\to\bbG\bbW_{n+1,n}
    \end{equation}
by sending $z\in D-U$ to the unique element $\varphi_{D-U}(z)\in\bbG\bbW_{n+1,n}$ satisfying (1) its image is $\ell_z$ and (2)
    \begin{equation}
    \begin{cases}
    \eval_0\big(\varphi_{D-U}(z)\big)=\sigma(z)/\abs{\abs{\sigma(z)}} \\
    \eval_{+\infty}\big(\varphi_{D-U}(z)\big)=z
    \end{cases}
    .
    \end{equation}
The following is the second crucial claim.

\begin{claim}
We may define a map 
    \begin{equation}
    \varphi:D\to\bbG\bbW_{n+1,n}
    \end{equation}
such that 
\begin{enumerate}
\item $\varphi\vert_{D'}=\varphi_{D'}$,
\item $\varphi\vert_{D-U}=\varphi_{D-U}$, 
\item $\varphi(z)$ has image $\ell_z$ for all $z\in D$.
\end{enumerate}
\end{claim}

\begin{proof}[Proof of claim]
This is a consequence of the proof of the previous claim. When gluing the two families $\{\ell^{D'}_z\}$ and $\{\ell^{D-D'}_z\}$ together, we essentially extend $\{\ell^{D'}_z\}$ to a family $\{\ell^U_z\}$ and use contractibility of the fiber to linearly interpolate between 
    \begin{equation}
    \ell^U_z\;\;\textrm{and}\;\;\ell^{D-D'}_z,\;\; \forall z\in U.
    \end{equation}
Now, extend $\varphi_{D'}$ resp. $\varphi_{D-U}$ to a map 
    \begin{equation}
    \varphi_U:U\to\bbG\bbW_{n+1,n}\;\;\textrm{resp.}\;\;\varphi_{D-D'}:D-D'\to\bbG\bbW_{n+1,n}
    \end{equation}
with the obvious properties and consider the line $L_z$ in $\bbC P^{n+1}$ connecting 
    \begin{equation}
    \eval_0\big(\varphi_U(z)\big)\;\;\textrm{to}\;\;\eval_0\big(\varphi_{D-D'}(z)\big)=\sigma(z)/\abs{\abs{\sigma(z)}},\;\;z\in U.
    \end{equation}
When performing the linear interpolation between $\ell^U_z$ and $\ell^{D-D'}_z$, we may fix the parameterization of the interpolating line to satisfy the requirement that its evaluation at 0 lives in $L_z$. Given $z\in U$, we define $\varphi(z)$ to be the unique element $\varphi(z)\in\bbG\bbW_{n+1,n}$ satisfying (1) its image is $\ell_z$ and (2)
    \begin{equation}
    \begin{cases}
    \eval_0\big(\varphi(z)\big)\in L_z \\
    \eval_{+\infty}\big(\varphi(z)\big)=z
    \end{cases}
    ;
    \end{equation}
the claim follows.
\end{proof}

We may define a lift of $\varphi_{D-U}$, 
    \begin{equation}
    \begin{tikzcd}
    & \bbG\bbW^{S^1}_{n+1,n}\arrow[d] \\
    D-U\arrow[ur,"\Phi_{D-U}"]\arrow[r,"\varphi_{D-U}"] & \bbG\bbW_{n+1,n},
    \end{tikzcd}
    \end{equation}
by sending $z\in D-U$ to the unique element $\Phi_{D-U}(z)\in\bbG\bbW^{S^1}_{n+1,n}$ satisfying (1) its image is $\ell_z$ and (2) 
    \begin{equation}
    \begin{cases}
    \eval_0\big(\Phi_{D-U}(z)\big)=\eval_0\big(\varphi_{D-U}(z)\big) \\
    \eval_{+\infty}\big(\Phi_{D-U}(z)\big)=z \\
    \eneval_{+\infty}\big(\Phi_{D-U}(z)\big)=\sigma(z)/\abs{\abs{\sigma(z)}}
    \end{cases}
    .
    \end{equation}
Since $\sigma$ vanishes along $D'$, we see $\Phi_{D-U}$ does not extend across $D'$. We may equivalently define $\Phi_{D-U}$ as a map 
    \begin{align}
    \sigma(D-U)&\to\bbG\bbW^{S^1}_{n+1,n} \\
    \big(z,\sigma(z)/\abs{\abs{\sigma(z)}}\big)&\mapsto\Phi_{D-U}(z). \nonumber
    \end{align}
Now, by Lemma \ref{lem:unorientedbordism}, we have that $\sigma(D-U)$ compactifies to a manifold with boundary given by 
    \begin{equation}
    S_DM\vert_{D'}\cong\bbG\bbW^{S^1}(b_\mathrm{min}).
    \end{equation}
We extend $\Phi_{D-U}$ to a map $\Phi$ on the compactification by sending $(z,v)\in S_DM\vert_{D'}$ to the unique element $\Phi(u)\in\bbG\bbW^{S^1}_{n+1,n}$ satisfying (1) its image is $\ell_z$ and (2) 
    \begin{equation}
    \begin{cases}
    \eval_0\big(\Phi(z)\big)=b_\mathrm{min} \\
    \eval_{+\infty}\big(\Phi(z)\big)=z \\
    \eneval_{+\infty}\big(\Phi(z)\big)=v
    \end{cases}
    .
    \end{equation}
The upshot is that we have exhibited the compactification of $\sigma(D-U)$, to a manifold with boundary $S_DM\vert_{D'}$, as a compact submanifold with boundary of $\bbG\bbW^{S^1}_{n+1,n}$.

\begin{claim}
There exists a twisted stable framing on $\sigma(D-U)$ which extends our constructed twisted stable framing on $\bbG\bbW^{S^1}(b_\mathrm{min})$, i.e., $\bbG\bbW^{S^1}(b_\mathrm{min})$ is twisted stably framed null-bordant.
\end{claim}

\begin{proof}[Proof of claim]
First, we would like to point out the following. Recall, the stable $\bbR$-polarization $\Lambda$ on $X$ is induced by the relation 
    \begin{equation}
    TM\oplus\underline{\bbC}\oplus\scrO_{\bbC P^{n+1}}(2)\vert_M\cong\scrO_{\bbC P^{n+1}}(1)\vert_M^{\oplus n+2}.
    \end{equation}
Meanwhile, we may induce a stable $\bbR$-polarization $\widetilde{\Lambda}$ on $\widetilde{X}\equiv\bbC P^{n+1}-H$ by the relation
    \begin{equation}
    T\bbC P^{n+1}\oplus\underline{\bbC}\cong\scrO_{\bbC P^{n+1}}(1)^{\oplus n+2}.
    \end{equation}
Observe, $TM\oplus\scrO_{\bbC P^{n+1}}(2)\vert_M\cong T\bbC P^{n+1}\vert_M$. In particular, $\Lambda=\widetilde{\Lambda}\vert_X$.

Consider $[u]\in\bbG\bbW^{S^1}_{n+1,n}$ in the image of $\Phi$ (in the following, we will abuse notation and also denote by $\Phi$ the image of the compactification of $\sigma(D-U)$ in $\bbG\bbW^{S^1}_{n+1,n}$; no confusion should arise). We have that the linearization at $u$ is an operator
    \begin{equation}
    D^{\bbG\bbW_{n+1,n}}_u:W^{2,2}_{+\infty}(\bbC P^1;u^*T\bbC P^{n+1})\to W^{1,2}(\bbC P^1;u^*T\bbC P^{n+1}).
    \end{equation}
Let $\widetilde{\bbG\bbW}_{n+1,n}$ be defined essentially as $\bbG\bbW^{S^1}_{n+1,n}$ but without modding out by $\bbR$-translations. We may construct a twisted stable framing of $\Phi$ at $[u]$ by using the identifications 
    \begin{equation}
    \begin{cases}
    T\widetilde{\bbG\bbW}_{n+1,n}\cong T\bbG\bbW^{S^1}_{n+1,n}+\underline{\bbR} \\
    T\bbG\bbW^{S^1}_{n+1,n}\vert_\Phi\cong T\Phi\oplus N_\Phi\bbG\bbW^{S^1}_{n+1,n}
    \end{cases}
    \end{equation}
and considering the index bundle of the glued together operator
    \begin{equation}
    \big(D^{\bbG\bbW_{n+1,n}}_u\oplus\Phi^{\bbG\bbW}\big)\#D^{\widetilde{\Lambda}}_{\eneval_{+\infty}(u)}.
    \end{equation}
This yields a canonical isomorphism of real virtual vector spaces: 
    \begin{multline}
    \ind D^{\bbG\bbW_{n+1,n}}_u+\ind\Phi^{\bbG\bbW}+\ind D^{\widetilde{\Lambda}}_{\eneval_{+\infty}(u)}+ \\
    N_H\bbC P^{n+1}\vert_{u(+\infty)}^{\oplus2}-(T\bbC P^{n+1}\oplus\underline{\bbC})\vert_{\eval_{+\infty}(u)}\cong\widetilde{\Lambda}_{\eneval_{+\infty}(u)}.
    \end{multline}
A computation shows we have the following twisted stable framing:
    \begin{equation}
    T\Phi+N_\Phi\bbG\bbW^{S^1}_{n+1,n}+\underline{\bbR}\cong\eval_{+\infty}^*\big(T\bbC P^{n+1}-N_H\bbC P^{n+1}\big)+\underline{\bbR}^{2(n+2)},
    \end{equation}
cf. Subsection \ref{subsec:enhancedspheres}. In particular, this induces a twisted stable framing on $\partial\Phi$:
    \begin{equation}
    T\partial\Phi+N_{\partial\Phi}\Phi+N_\Phi\bbG\bbW^{S^1}_{n+1,n}\vert_{\partial\Phi}+\underline{\bbR}\cong\eval_{+\infty}^*\big(T\bbC P^{n+1}-N_H\bbC P^{n+1}\big)+\underline{\bbR}^{2(n+2)}.
    \end{equation}
The claim is that this twisted stable framing on $\partial\Phi$ agrees with our already constructed twisted stable framing on $\bbG\bbW^{S^1}(b_\mathrm{min})$ under the equivalence
    \begin{equation}
    \partial\Phi\cong\bbG\bbW^{S^1}(b_\mathrm{min}).
    \end{equation}

Consider $[u]\in\partial\Phi$. We have that the linearization at $u$ thought of as a map to $M$ is an operator
    \begin{equation}
    D^{\bbG\bbW}_u:W^{2,2}_{+\infty}(\bbC P^1;u^*TM)\to W^{1,2}(\bbC P^1;u^*TM).
    \end{equation}
Our constructed twisted stable framing of $\bbG\bbW^{S^1}(b_\mathrm{min})$ at $[u]$ is given by using the identification
    \begin{equation}
    T\widetilde{\bbG\bbW}\cong T\bbG\bbW^{S^1}+\underline{\bbR}
    \end{equation}
and considering the index bundle of the glued together operator 
    \begin{equation}
    \big(D^{\bbG\bbW}_u\oplus\Phi^{\bbG\bbW}\oplus D^{\bbG\bbW,\scrO_{\bbC P^{n+1}}(2)\vert_M}_u\big)\#D^{\widetilde{\Lambda}}_{\eneval_{+\infty}(u)}.
    \end{equation}
This yields a canonical isomorphism of real virtual vector spaces: 
\begin{multline}
    \ind D^{\bbG\bbW}_u+\ind\Phi^{\bbG\bbW}+\ind D^{\bbG\bbW,\scrO_{\bbC P^{n+1}}(2)\vert_M}_u+\ind D^\Lambda_{\eneval_{+\infty}(u)}+ \\
    N_DM\vert_{u(+\infty)}^{\oplus2}-\big(TM\oplus\underline{\bbC}\oplus\scrO_{\bbC P^{n+1}}(2)\vert_M\big)\vert_{\eval_{+\infty}(u)}\cong\Lambda_{\eneval_{+\infty}(u)}.
    \end{multline}
Again, a computation shows we have the following twisted stable framing:
    \begin{equation}
    T\bbG\bbW^{S^1}(b_\mathrm{min})+\eval_0^*TX+\underline{\bbR}\cong\eval^*_{+\infty}(TM-N_DM)+\underline{\bbR}^{2(n+2)-4},
    \end{equation}
cf. Part \ref{subsubsec:auxenhancedspheres}. Meanwhile, we have that the linearization at $u$ thought of as a map to $\bbC P^{n+1}$ is an operator
    \begin{equation}
    D^{\bbG\bbW_{n+1,n}}_u:W^{2,2}_{+\infty}(\bbC P^1;u^*T\bbC P^{n+1})\to W^{1,2}(\bbC P^1;u^*T\bbC P^{n+1}).
    \end{equation}
Recall, we have that 
    \begin{equation}
    u^*T\bbC P^{n+1}\cong u^*\big(TM\oplus\scrO_{\bbC P^{n+1}}(2)\vert_M\big);
    \end{equation}
moreover, $D^{\bbG\bbW_{n+1,n}}_u$ has a condition on sections at $+\infty$ to lie in $T_{u(+\infty)}H$. In particular, we may decompose $D^{\bbG\bbW_{n+1,n}}_u$ into the direct sum of two operators:
    \begin{equation}
    D^{\bbG\bbW_{n+1,n}}_u=D^{\bbG\bbW}_u\oplus D^{\bbG\bbW,\scrO_{\bbC P^{n+1}}(2)\vert_M}_u,
    \end{equation}
where we recall $D^{\bbG\bbW}_u$ has a condition on sections at $+\infty$ to lie in $T_{u(+\infty)}D$. Meanwhile, consider our family of canonical transverse operators $D^{\widetilde{\Lambda}}$ resp. $D^\Lambda$ on $S_H\bbC P^{n+1}$ resp. $S_DM$. By the identification 
    \begin{equation}
    N_DM=N_H\bbC P^{n+1}\vert_D,
    \end{equation}
we have the following canonical isomorphism of real virtual bundles:
    \begin{equation}
    \ind D^{\widetilde{\Lambda}}\vert_{S_H\bbC P^{n+1}\vert_D}\cong\ind D^\Lambda.
    \end{equation}
Hence, we now have the following commutative diagram of canonical isomorphisms of real virtual bundles:
    \begin{equation}
    \begin{tikzcd}[column sep=-10ex, center picture]
    & (T\bbC P^{n+1}-N_H\bbC P^{n+1})\vert_{\eval_{+\infty}(u)}+\bbR^{2(n+2)} \\
    & \widetilde{\Lambda}_{\eneval_{+\infty}(u)}-\ind\Phi^{\bbG\bbW}-\ind D^{\widetilde{\Lambda}}_{\eneval_{+\infty}(u)}-N_H\bbC P^{n+1}\vert_{u(+\infty)}^{\oplus2}+(T\bbC P^{n+1}\oplus\underline{\bbC})\vert_{\eval_{+\infty}(u)}\arrow[u,"\sim"]\arrow[ddd,"\sim",bend left=77] \\
    T_u\partial\Phi+N_{\partial\Phi}\Phi\vert_u+N_\Phi\bbG\bbW^{S^1}_{n+1,n}\vert_u+\bbR\arrow[r,"\sim"] & \ind D^{\bbG\bbW_{n+1,n}}_u\arrow[d,"\sim"]\arrow[u,"\sim"] \\
    T_u\bbG\bbW^{S^1}(b_\mathrm{min})+T_{\eval_0(u)}X+\bbR^7\arrow[r,"\sim"] & \ind\big(D^{\bbG\bbW}_u\oplus D^{\bbG\bbW,\scrO_{\bbC P^{n+1}}(2)\vert_M}_u\big)\arrow[u]\arrow[d,"\sim"] \\
    & \Lambda_{\eneval_{+\infty}(u)}-\ind\Phi^{\bbG\bbW}-\ind D^\Lambda_{\eneval_{+\infty}(u)}-N_DM\vert_{u(+\infty)}^{\oplus2}+\big(TM\oplus\underline{\bbC}\oplus\scrO_{\bbC P^{n+1}}(2)\vert_M\big)\vert_{\eval_{+\infty}(u)}\arrow[d,"\sim"]\arrow[uuu,bend right=77] \\
    & (TM-N_DM)_{\eval_{+\infty}(u)}+\bbR^{2(n+2)}.
    \end{tikzcd}
    \end{equation}
The previous diagram requires an explanation. The chain of arrows connecting the top resp. bottom expression on the left to the top resp. bottom of the diagram is our constructed twisted stable framing (in the case of the bottom expression on the left going to the bottom, it is our constructed twisted stable framing stabilized by $\underline{\bbR}^6$). Meanwhile, the curved arrow expresses that the two twisted stable framings on $\partial\Phi\cong\bbG\bbW^{S^1}(b_\mathrm{min})$ are compatible; in particular, this isomorphism is constructed by using identifications resp. canonical isomorphisms term-by-term: 
    \begin{align}
    \widetilde{\Lambda}_{\eneval_{+\infty}(u)}&=\Lambda_{\eneval_{+\infty}(u)} \\
    \ind\Phi^{\bbG\bbW}&=\ind\Phi^{\bbG\bbW} \\
    \ind D^{\widetilde{\Lambda}}_{\eneval_{+\infty}(u)}&\cong\ind D^\Lambda_{\eneval_{+\infty}(u)} \\
    N_H\bbC P^{n+1}\vert_{u(+\infty)}^{\oplus2}&=N_DM\vert_{u(+\infty)}^{\oplus2} \\
    (T\bbC P^{n+1}\oplus\underline{\bbC})\vert_{\eval_{+\infty}(u)}&\cong\big(TM\oplus\underline{\bbC}\oplus\scrO_{\bbC P^{n+1}}(2)\vert_M\big)\vert_{\eval_{+\infty}(u)};
    \end{align}
this proves the claim.
\end{proof}

Finally, this shows $\calG\calW_2$ is null-homotopic, and completes the proof of Proposition \ref{prop:spheresplitting}, as follows. We have constructed a ``global'' twisted framed null-bordism of $\bbG\bbW^{S^1}(b_\mathrm{min})$; but observe, this global twisted framed-null bordism is compatible with $\eneval_{+\infty}$, i.e., we constructed a twisted framed null-bordism of the map 
    \begin{equation}
    \eneval_{+\infty}:\bbG\bbW^{S^1}(b_\mathrm{min})\to S_DM.
    \end{equation}
In particular, we immediately induce twisted framed null-bordisms of the various $\bbG\bbW^{S^1}(a,b_\mathrm{min})$ compatible with gluing.

\subsubsection{A closer look at the 2-sphere}
In this part, we will perform a sanity check of Theorem \ref{thm:main} by showing Proposition \ref{prop:spheresplitting}, in the special case $n=2$, plus the spectral Viterbo isomorphism recovers Cohen's stable splitting of $(\calL S^2)_+$: 
    \begin{equation}
    (\calL S^2)_+\simeq S^0\vee\bigvee_{k\geq1}S^1_+\wedge_{\bbZ/k}S^k\simeq S^0\vee(S^1\vee S^2)\vee\Sigma\bbR P^2\vee(S^3\vee S^4)\vee\Sigma^3\bbR P^2\vee\cdots.
    \end{equation}

Recall, the following is the spectral Viterbo isomorphism, cf. \cite[Theorem 1.2]{Bla1} for example.

\begin{thm}[Spectral Viterbo isomorphism]
Let $Q$ be a closed smooth manifold and $(T^*Q,\omega_\std=-d\theta_\std)$ its cotangent bundle with the standard Liovuille structure. Equip $T^*Q$ with its standard stable $\bbR$-polarization $\Lambda_\std$ given by the horizontal distribution of $T(T^*Q)$ induced by a choice of Riemannian metric on $Q$. We have that the Floer homotopy type of $T^*Q$ with respect to $\Lambda_\std$ is equivalent to the Thom spectrum of (the pullback of) $-TQ$ to $\calL Q$:
    \begin{equation}
    \Sigma^{-n}\frakF^{\Lambda_\std}_{T^*Q}\simeq\calL Q^{-TQ}.
    \end{equation}
\end{thm}

\begin{rem}
Observe, the previous statement of the spectral Viterbo isomorphism differs from \cite[Theorem 1.2]{Bla1} up to a ($\pm n$)-fold shift only as a choice of convention, cf. Remark \ref{rem:shiftconvention}.
\end{rem}
    
First, we note that the unit sphere bundle of $S^2$ is $\bbR P^3$. In particular, $\bbR P^3$ is stably framed and has trivial 0-th (real) $K$-theory; hence, Proposition \ref{prop:spheresplitting} shows 
    \begin{equation}
    \frakF^\Lambda\simeq S^2_+\vee\bigvee_{k\geq1}\Sigma^{2k-1}\bbR P^3_+.
    \end{equation}
Using the stable homotopy equivalence $\bbR P^3\simeq\bbR P^2\vee S^3$, cf. \cite[Theorem (C)]{Mil85}, we may rewrite this as
    \begin{equation}
    \frakF^\Lambda\simeq S^2_+\vee\bigvee_{k\geq1}\Sigma^{2k-1}\bbR P^2\vee S^{2k+2}\vee S^{2k-1}.
    \end{equation}
Observe, this $\Lambda$ is the one induced by Proposition \ref{prop:spheresplitting}: 
    \begin{equation}
    T(T^*S^2)\oplus\underline{\bbC}^2\cong\underline{\bbC}^4. 
    \end{equation}
Meanwhile, since $S^2$ is stably framed, the spectral Viterbo isomorphism shows 
    \begin{equation}
    \frakF^{\Lambda_\std}\simeq(\calL S^2)_+.
    \end{equation}
Observe, $\Lambda$ resp. $\Lambda_\std+\underline{\bbR}^2$ determines a map
    \begin{equation}
    T(T^*S^2)\to BO(4).
    \end{equation}
Finally, since $\pi_2BO(4)=0$, we see 
    \begin{equation}
    \frakF^\Lambda\simeq\frakF^{\Lambda_\std}; 
    \end{equation}
hence, our stable splitting of $(\calL S^2)_+$ agrees with Cohen's.

\subsubsection{Even degree del Pezzo surfaces}
In this part, let $M$ be the blow-up of $\bbC P^2$ at $\ell$ generic points $Z\equiv\{p_1,\ldots,p_\ell\}$, where $1\leq\ell\leq8$, and $D$ a generic smooth representative of the Poincar\'e dual of the anticanonical bundle $\calL\to M$. (For concreteness, we may take $D$ to be the strict transform of a generic smooth plane cubic passing through $Z$. Topologically, $D$ is a 2-torus.) Recall, the degree of $M$ is equal to $9-\ell$.

\begin{claim}
Suppose $\ell$ is odd, then there exists a totally real subbundle $\Lambda\subset TM\oplus\calL_{-1}$ such that 
    \begin{equation}
    TM\oplus\calL_{-1}\cong\Lambda\otimes_\bbR\underline{\bbC},
    \end{equation}
where we recall $\calL_{-1}=\scrO_M(-D)$.
\end{claim}

\begin{proof}[Proof of claim]
We proceed via a general obstruction theory argument. Since 
    \begin{equation}
    c_1(TM\oplus\calL_{-1})=c_1(M)-c_1(M)=0, 
    \end{equation}
we may reduce the structure group of $TM\oplus\calL_{-1}$ from $U(3)$ to $SU(3)$, i.e., $TM\oplus\calL_{-1}$ is classified by a map 
    \begin{equation}
    \varphi:M\to BSU(3).
    \end{equation}
Consider the standard fiber sequence 
    \begin{equation}
    SO(3)\to SU(3)\to SU(3)/SO(3); 
    \end{equation}
our claim is equivalent to lifting $\varphi$ to $SO(3)$. Let $M^k$ denote the $k$-skeleton of $M$; we have a lift
    \begin{equation}
    \begin{tikzcd}
    & BSO(3)\arrow[d] \\
    M^1\arrow[r,"\varphi"]\arrow[ur,"\varphi^1_\bbR",dashed] & BSU(3)
    \end{tikzcd}
    \end{equation}
since $c_1(TM\oplus\calL_{-1})=0$. The quotient $W\equiv SU(3)/SO(3)$ is known as the \emph{Wu manifold} and has homotopy groups (which are relevant to us) 
    \begin{equation}
    \pi_0(W)=\pi_1(W)=0,\;\;\pi_2(W)\cong\bbZ/2,\;\;\pi_3(W)\cong\bbZ/4.
    \end{equation}
Standard obstruction theory says the obstructions to lifting $\varphi$ lie in 
    \begin{equation}
    H^{j+1}\big(M;\pi_j(W)\big)\cong\begin{cases}
    \bbZ/4, & j=3 \\
    0, & \mathrm{otherwise}
    \end{cases}
    ,
    \end{equation}
where the isomorphism above uses the facts that $M$ is orientable and simply-connected.\footnote{Recall, every Fano variety is simply-connected.} In particular, we have a lift
    \begin{equation}
    \begin{tikzcd}
    & BSO(3)\arrow[d] \\
    M^3\arrow[r,"\varphi"]\arrow[ur,"\varphi^3_\bbR",dashed] & BSU(3).
    \end{tikzcd}
    \end{equation}
It remains to extend $\varphi^3_\bbR$ over the $4$-skeleton, but this is straightforward. Let $(D^4,a:S^3\to M^3)$ be the 4-cell of $M$. By construction, we have a real vector bundle $\Lambda^3\to M^3$ such that
    \begin{equation}
    a^*(TM\oplus\calL_{-1})\cong a^*(\Lambda^3\otimes_\bbR\underline{\bbC}).
    \end{equation}
Since the map 
    \begin{equation}
    \pi_3\big(SO(3)\big)\cong\bbZ\to\pi_3\big(SU(3)\big)\cong\bbZ
    \end{equation}
is multiplication by 4, we may choose an extension of $a^*\Lambda^3$ to $D^4$ compatible with the already defined extension of $a^*(TM\oplus\calL_{-1})$ to $D^4$ if 
    \begin{equation}
    c_2(TM\oplus\calL_{-1})=c_2(M)-c_1(M)^2=0\mod4.
    \end{equation}
The general theory of blow-ups shows 
    \begin{equation}
    c_1(M)=3\phi^*[\bbC P^1]-[E_1]-\cdots-[E_\ell], 
    \end{equation}
where $\phi:M\to\bbC P^2$ is the blow-down map and $\{E_1,\ldots,E_\ell\}$ are the exceptional divisors associated to $Z$; thus, $c_1(M)^2=9-\ell$. We see (by Noether's formula)
    \begin{equation}
    12\chi(\scrO_M)=c_2(M)+c_1(M)^2\iff12-2(9-\ell)=c_2(M)-c_1(M)^2, 
    \end{equation}
where $\chi(\scrO_M)$ is the holomorphic Euler characteristic of $M$; the claim follows.
\end{proof}

\begin{rem}
In fact, since $BSO(3)$ is 2-connected and $X$ is a 2-dimensional CW-complex, we see that $\Lambda\vert_X\cong\underline{\bbR}^3$.
\end{rem}

It is a classical fact in algebraic geometry that the moduli space of lines on $M$ is a finite number of points. These lines are the exceptional divisors together with the strict transforms of the:
\begin{itemize}
\item lines through 2 of the blown-up points,
\item conics through 5 of the blown-up points,
\item cubics through 7 of the blown-up points which are singular at 1 of them,
\item quartics through 8 of the blown-up points which are singular at 3 of them,
\item quintics through 8 of the blown-up points which are singular at 6 of them,
\item and sextics through 8 of the blown-up points which are singular at all 8 of them and contain 1 of them as a triple point.
\end{itemize}
Moreover, it is straightforward to see each such line $L_j$ intersects $D$ transversely in a single point. We define $q_j\equiv D\cap L_j$; note, these points are distinct.

We may compute:
    \begin{equation}
    \dim\bbG\bbW^{S^1}(a,b)=I(b)-I(a)-1.
    \end{equation}
Since $f_X$ is a Morse function whose negative gradient flow points strictly inwards on the boundary and $X$ is a 4-dimensional Weinstein manifold, we see
    \begin{equation}
    I(b)\in\{0,1,2\}.
    \end{equation}
Moreover, since $S_DM$ is a closed smooth 3-manifold, we see 
    \begin{equation}
    I(a)\in\{0,1,2,3\}.
    \end{equation}
We denote by $a_\mathrm{min}$ resp. $b_\mathrm{min}$ the minimum of $f_S$ resp. $f_X$.

\begin{claim}
Generically, we have that 
    \begin{equation}
    \bbG\bbW^{S^1}(a,b)\neq\emptyset\iff I(a)\in\{0,1\}\;\;\textrm{and}\;\;I(b)=2.
    \end{equation}
\end{claim}

\begin{proof}[Proof of claim]
We have seen that we have a decomposition 
    \begin{equation}
    \bbG\bbW^{S^1}(a,b)=\coprod_{L_j}\bbG\bbW^{S^1,L_j},
    \end{equation}
where $[u]\in\bbG\bbW^{S^1,L_j}$ if and only if the image of $u$ is $L_j$. By perturbing $f_S$, we may assume 
    \begin{equation}
    S_DM\vert_{q_j}\cap W^s(a;f_S)\neq\emptyset\iff I(a)\in\{0,1\};
    \end{equation}
this is a straightforward transversality argument. Thus, 
    \begin{equation}
    \bbG\bbW^{S^1}(a,b)\neq\emptyset\iff I(a)\in\{0,1\}\;\;\mathrm{and}\;\;I(b)\in\{0,1,2\}.
    \end{equation}
By perturbing $f_X$, we may assume 
    \begin{equation}
    \dim L_j\cap W^u(b;f_X)=I(b)-2;
    \end{equation}
again, this is a straightforward transversality argument. Thus, 
    \begin{equation}
    \bbG\bbW^{S^1}(a,b)\neq0\iff I(a)\in\{0,1\}\;\;\mathrm{and}\;\;I(b)=2.
    \end{equation}
\end{proof}

We may now prove the following.

\begin{prop}\label{prop:delpezzo1}
Let $M$ be the blow-up of $\bbC P^2$ at $\ell$ generic points, where $1\leq\ell\leq8$ and $\ell$ is odd, and $D$ a generic smooth representative of the Poincar\'e dual of the anticanonical bundle $\calL\to M$. In this case, we may choose $\Lambda$ such that $TM\oplus\calL_{-1}\cong\Lambda\otimes_\bbR\underline{\bbC}$. We have that
    \begin{equation}
    \frakF^\Lambda\simeq\frakD X^{-V_0}\vee\bigvee_{k\geq1}\frakD(S_DM)^{-V_k},
    \end{equation}
where 
    \begin{equation}
    V_k=\underline{\bbR}^{-1+4k}+TM^{\oplus k}+\Lambda^{\oplus 1-2k}-N_DM^{\oplus k}
    \end{equation}
\end{prop}

\begin{proof}
The proof has two main steps. First, we will prove a cohomological splitting. Second, we will use the cohomological splitting to prove the homotopical splitting. 

Observe, the results of the present article show that if 
    \begin{equation}
    \calG\calW\wedge H\bbZ\in H_0\Big((S_DM)^{-V^{\bbG\bbW}+TX}\wedge\Sigma\frakD X_+;\bbZ\Big),
    \end{equation}
where 
    \begin{equation}
    V^{\bbG\bbW}=TM-\underline{\bbR}^2-N_DM,
    \end{equation}
vanishes, then $SH^*(X;\bbZ)$ splits. Using the K\"unneth formula and the Thom isomorphism, we may compute (we drop the dependence on $\bbZ$ for brevity):
    \begin{align}
    H_0\Big((S_DM)^{\underline{\bbR}^2+N_DM}\wedge\Sigma\frakD X_+\Big)\cong&\bigoplus_{i+j=0}H_{i-4}(S_DM)\otimes H^{1-j}(X)\oplus \\
    &\bigoplus_{i+j=-1}\operatorname{Tor}\big(H_{i-4}(S_DM),H^{1-j}(X)\big) \nonumber \\
    \cong&\bigoplus_{i-j=-5}H_i(S_DM)\otimes H^j(X)\oplus \\
    &\bigoplus_{i-j=-6}\operatorname{Tor}\big(H_i(S_DM),H^j(X)\big). \nonumber
    \end{align}
Direct inspection shows the last line vanishes since $H_*(S_DM;\bbZ)$ is concentrated in at most degrees $0,1,2,3$ and $H^*(X;\bbZ)$ is concentrated in at most degrees $0,1,2$; hence, $SH^*(X;\bbZ)$ splits into its associated graded. 

Now, we may prove the homotopical splitting.  In order to prove a homotopical splitting, we must show the moduli spaces defining the map 
    \begin{equation}
    \calG\calW:\frakD(S_DM)^{\underline{\bbR}^2+N_DM}\to\Sigma\frakD X_+
    \end{equation}
are compatibly twisted stably framed null-bordant. Let $\pi_{\leq\mu}MSO$ be the $\mu$-th Postnikov truncation of the oriented bordism spectrum $MSO$, i.e., 
    \begin{equation}
    \pi_n\big(\pi_{\leq\mu}MSO\big)=\begin{cases}
    \pi_nMSO, & n\leq\mu \\
    0, & \mathrm{otherwise}
    \end{cases}
    .
    \end{equation}
For example, 
    \begin{equation}\label{eq:truncation1}
    \pi_{\leq0}MSO=\pi_{\leq1}MSO=\pi_{\leq2}MSO=\pi_{\leq3}MSO=H\bbZ.
    \end{equation}
Thus, our cohomological splitting above shows 
    \begin{equation}\label{eq:truncation2}
    \calG\calW\wedge\pi_{\leq\mu}MSO=0,\;\;\mu\leq3.
    \end{equation}
    
Now, recall that the stable $\infty$-category of $\pi_{\leq\mu}MSO$-modules is equivalent to the stable $\infty$-category of ``$\mu$-truncated oriented flow categories'' -- the $n$-simplices of the latter essentially consists of oriented flow $n$-simplices in the usual sense, but with moduli spaces only up to dimension $\mu$ -- this fact was essentially conjectured in \cite[Conjecture 5.46 \& 5.47]{PS24b}, used in a simplified sense in \cite{Bon25}, and is essentially proven via the techniques of \cite{AB24,HO26}. The reason this notion is important in the present article is that $\calG\calW\wedge MSO$, being defined using only moduli spaces up to dimension 1, is already a 1-truncated oriented flow bimodule. In particular, \eqref{eq:truncation1}, \eqref{eq:truncation2}, and the cohomological splitting above imply that we already know the moduli spaces defining $\calG\calW$ are compatibly oriented null-bordant; it remains to show the twisted stable framings extend compatibly over these oriented null-bordisms.

This is relatively straightforward. Note, a twisted stable framing on a point resp. an interval is simply a choice of orientation. In particular, we see that the moduli spaces defining $\calG\calW$ consist of (1) twisted stably framed compact 0-manifolds and (2) twisted stably framed compact 1-manifolds with (possibly empty) boundary. Hence, the compact 1-manifolds break up into disjoint unions of intervals and circles; therefore, we only need to check that our constructed twisted stable framing on each circle component is not the non-trivial Lie group framing, i.e., that it extends over a disk. But this follows immediately by our construction of the twisted stable framing. Namely, a circle component of $\bbG\bbW^{S^1}(a_\mathrm{min},b)$, $I(b)=2$, is given as the total space of the following $S^1$-bundle:
    \begin{equation}
    S^1\to\bbG\bbW^{S^1,L_j}\to\{L_j\}.
    \end{equation}
In other words, a circle component consists of a choice of $L_j$ and the entire fiber $S_DM\vert_{q_j}$ over it, where the latter is in bijection with a choice of where the enhanced evaluation at $+\infty$ lands. Moreover, our constructed twisted stable framing involves gluing a canonical transverse operator determined by the enhanced evaluation at $+\infty$. But we have already seen the twisted stable framing given by gluing a canonical transverse operator is independent of which point in the fiber we actually use, cf. the proof of Corollary \ref{cor:tsfs1}; the proposition follows.
\end{proof}

\subsubsection{Degree 3 del Pezzo surfaces}
In this part, let $M$ be the blow-up of $\bbC P^2$ at 6 generic points and $D$ a generic smooth representative of the Poincar\'e dual of the anticanonical bundle $\calL\to M$. We may realize $M$ as a generic smooth cubic hypersurface in $\bbC P^3$ via the anticanonical embedding. In particular, $D$ is a smooth hyperplane section. Moreover, we have that
    \begin{equation}
    TM\oplus\underline{\bbC}\oplus\scrO_{\bbC P^3}(3)\vert_M\cong\scrO_{\bbC P^3}(1)\vert_M^{\oplus4}.
    \end{equation}
    
\begin{prop}\label{prop:delpezzo2}
Let $M$ be the blow-up of $\bbC P^2$ at 6 generic points and $D$ a generic smooth representative of the Poincar\'e dual of the anticanonical bundle $\calL\to M$. Let $\Lambda$ be induced by 
    \begin{equation}
    TM\oplus\underline{\bbC}\oplus\scrO_{\bbC P^3}(3)\vert_M\cong\scrO_{\bbC P^3}(1)\vert_M^{\oplus4}.
    \end{equation}
We have that
    \begin{equation}
    \frakF^\Lambda\simeq\frakD X^{-V_0}\vee\bigvee_{k\geq1}\frakD(S_DM)^{-V_k},
    \end{equation}
where 
    \begin{equation}
    V_k=\underline{\bbR}^{2-2k}+TM^{\oplus k}-N_DM^{\oplus k}
    \end{equation}
\end{prop}

\begin{proof}
Appropriately modify the proof of Proposition \ref{prop:delpezzo1}.
\end{proof}

\subsubsection{Short remarks on Conjecture \ref{conj:main}}\label{part:remarks}
The main hiccup when trying to extend the results of the present article to other geometric situations is finding sufficient conditions to construct twisted stable framings on the moduli spaces of both marked thimbles and enhanced spheres; part (3) of Assumption \ref{assu:main} addresses this.

Consider the case where $M$ is the blow-up of $\bbC P^2$ at 2 generic points and $D$ is a generic smooth representative of the anticanonical divisor of $M$. By \cite[Section 3]{BC19}, we may realize $M$ as a complete intersection of type $\big((1,0,1),(0,1,1)\big)$ in $(\bbC P^1)^2\times\bbC P^2$:
    \begin{equation}
    M=\Big\{\big([x_0:x_1],[y_0:y_1],[z_0:z_1:z_2]\big):z_0x_1=z_1x_0,z_0y_0=z_2y_1\Big\}.
    \end{equation}
In particular, $D$ is the Poincar\'e dual of the line bundle $\scrO(1,1,1)\vert_M$. We may choose $\Lambda$ to be induced by
    \begin{multline}\label{eqref:conjspec}
    TM\oplus\underline{\bbC}^3\oplus\scrO(1,0,1)\vert_M\oplus\scrO(0,1,1)\vert_M\cong \\
    \scrO(1,0,0)\vert_M^{\oplus2}\oplus\scrO(0,1,0)\vert_M^{\oplus2}\oplus\scrO(0,0,1)\vert_M^{\oplus3}.
    \end{multline}
Thus, the Floer homotopy type (associated to $\Lambda$) of the complement of $D$ exists. Unfortunately, this choice of $\Lambda$ does not fall under our setup in Assumption \ref{assu:main}. However, we observe that $D$ is equivalent as a divisor to the sum of (the pullbacks of) the three line classes in $(\bbC P^1)^2\times\bbC P^2$. The analgous statements hold for:
\begin{itemize}
\item the case where $M$ is the blow-up of $\bbC P^2$ at 4 generic points and $D$ is a generic smooth representative of the anticanonical divisor of $M$ since we may realize $M$ as a complete intersection of type
    \begin{equation}
    \big((1,0,0,0,1),(0,1,0,0,1),(0,0,1,0,1),(0,0,0,1,1)\big)
    \end{equation}
in $(\bbC P^1)^4\times\bbC P^2$ such that $D$ is the Poincar\'e dual of the line bundle $\scrO(1,1,1,1,-1)\vert_M$
\item and the case where $M$ is the blow-up of $\bbC P^2$ at 8 generic points and $D$ is a generic smooth representative of the anticanonical divisor of $M$ since we may realize $M$ as a complete intersection of type
    \begin{equation}
    \big((1,0,0,0,0,0,0,0,1),\ldots,(0,0,0,0,0,0,0,1,1)\big)
    \end{equation}
in $(\bbC P^1)^8\times\bbC P^2$ such that $D$ is the Poincar\'e dual of the line bundle $\scrO(1,\ldots,1,-1)\vert_M$.
\end{itemize}

The author believes it is possible to weaken part (3) of Assumption \ref{assu:main} to something akin to the following. 
\begin{enumerate}
\item[(4)] Suppose $D$ is equivalent as a divisor to $s_1D_1+\cdots+s_\tau D_\tau$, where $D_\gamma$ is an ample smooth divisor. For any $n'\in\bbZ$, we define $\calL^\gamma_{n'}\equiv\scrO_M(n'D_\gamma)$. There exists a non-negative integer $d\in\bbZ_{\geq0}$, integers $m_1,\ldots,m_\beta,n_1,\ldots,n_\alpha\in\bbZ$, oriented real vector bundles $F_1,\ldots,F_\alpha\to M$ whose restrictions to $X$ are spin, and an isomorphism of complex vector bundles
    \begin{equation}
    TM\oplus\underline{\bbC}^d\oplus\bigoplus_{\eta=1}^\beta\calL^{\gamma_\eta}_{m_\eta}\cong\bigoplus_{\nu=1}^\alpha E_\nu\otimes_\bbC\calL^{\gamma_\nu}_{n_\nu}.
    \end{equation}
\end{enumerate}

\begin{rem}
Perhaps one could even (among other things) allow various roots of lines bundles to appear. Motivating examples would be the geometric situations arising in Part \ref{part:warm-up}.
\end{rem}

\subsubsection{Affine part of smooth projective hypersurfaces}\label{subsubsec:affine}
In the part, we will prove the following.

\begin{prop}\label{prop:affine} 
Let $M$ be a smooth hypersurface in $\bbC P^{n+1}$ of degree $d$ at least 2 and $D$ a smooth hyperplane section. Let $\Lambda$ be induced by 
    \begin{equation}
    TM\oplus\underline{\bbC}\oplus\scrO_{\bbC P^{n+1}}(d)\vert_M\cong\scrO_{\bbC P^{n+1}}(1)\vert_M^{\oplus n+2}.
    \end{equation}
We have that
    \begin{equation}
    \frakF^\Lambda\simeq\frakD X^{-V_0}\vee\bigvee_{k\geq1}\frakD(S_DM)^{-V_k},
    \end{equation}
where 
    \begin{equation}
    V_k=\underline{\bbR}^{n+2k(2-d)}+TM^{\oplus k}-N_DM^{\oplus k}.
    \end{equation}
\end{prop}

\begin{rem}
Of course, the case $d=2$ was already proven in Proposition \ref{prop:spheresplitting}, and the case $n=2$ and $d=3$ in Proposition \ref{prop:delpezzo2}.
\end{rem}

\begin{proof}
We proceed by utilizing Theorem \ref{thm:main2}. By the Milnor fibration theorem, $X$ has the homotopy type of a bouquet of $\mu\in\bbZ_{>0}$ $n$-spheres:
    \begin{equation}
    X\simeq\bigvee_{\rho=1}^\mu S^n_\rho.
    \end{equation}
In fact, each sphere is the bouquet arises as a vanishing cycle after fixing a degeneration of $M$ to a nodal degree $d$ hypersurface; hence, we may realize each $S^n_\rho$ in the bouquet as a(n embedded) exact Lagrangian sphere $L_\rho\subset X$. Thus, $T^*L_\rho$ is a Liouville subdomain of $X$.

\begin{claim}
$\Lambda\vert_{T^*L_\rho}$ agrees with the stable $\bbR$-polarization on $T^*S^n_\rho$ induced by Proposition \ref{prop:spheresplitting}.
\end{claim}

\begin{proof}[Proof of claim]
We have that $\Lambda\vert_{T^*L_\rho}$ is induced by the isomorphism 
    \begin{equation}
    T(T^*S^n_\rho)\oplus\underline{\bbC}\oplus\scrO_{\bbC P^{n+1}}(d)\vert_{S^n_\rho}\cong T\bbC P^{n+1}\vert_{S^n_\rho}\oplus\underline{\bbC}\cong\scrO_{\bbC P^{n+1}}(1)\vert_{S^n_\rho}^{\oplus n+2}. 
    \end{equation}
Let $Q_\mathrm{aff}\subset\bbC P^{n+1}$ be the subset $\{1=z_1^2+\cdots+z_{n+1}^n\}$. Meanwhile, we have the stable $\bbR$-polarization on $T^*S^n_\rho$ induced by Proposition \ref{prop:spheresplitting} is induced by the isomorphism 
    \begin{equation}
    T(T^*S^n_\rho)\oplus\underline{\bbC}\oplus\scrO_{\bbC P^{n+1}}(2)\vert_{Q_\mathrm{aff}}\cong T\bbC P^{n+1}\vert_{Q_\mathrm{aff}}\oplus\underline{\bbC}\cong\scrO_{\bbC P^{n+1}}(1)\vert_{Q_\mathrm{aff}}^{\oplus n+2}. 
    \end{equation}
The claim follows by using the chain of isomorphisms 
    \begin{equation}
    \scrO_{\bbC P^{n+1}}(d)\vert_{S^n_\rho}\cong\underline{\bbC}\cong\scrO_{\bbC P^{n+1}}(2)\vert_{S^n_\rho}.
    \end{equation}
\end{proof}

Again, let $\frakE=(S_DM)^{-V_1}$. We have an isomorphism 
    \begin{equation}
    H_0\Big(\Sigma\frakD X^{-V_0};\frakE\Big)\cong\bigoplus_{\rho=1}^\mu H_0\Big(\Sigma\frakD L_\rho^{-V_0};\frakE\Big).
    \end{equation}
Moreover, the previous claim combined with Proposition \ref{prop:spheresplitting} shows the (suspension of the) PSS morphism,
    \begin{equation}
    H_0\Big(\Sigma\frakD L_\rho^{-V_0};\frakE\Big)\to H_0\Big(\frakF^{\Lambda\vert_{T^*L_\rho}};\frakE\Big),
    \end{equation}
is injective; the proposition follows by Theorem \ref{thm:main2}.
\end{proof}

\bibliography{References}{}
\bibliographystyle{alpha.bst}
\end{document}